%% file: main.tex
\documentclass{article}
\usepackage{amsmath,amssymb,amsthm,amsfonts,xypic,xy,makecell,fullpage}
\usepackage{multirow}
\usepackage{tikz}

\usepackage[colorlinks=true, urlcolor=black ,hyperindex,  linkcolor=black, pagebackref=false, citecolor=black]{hyperref}

\input{newcommands.tex}

\renewcommand{\ind}{\mathrm{ind}}

\title{Mod $p$ local-global compatibility for $\GSp_4(\Qp)$ in the ordinary case}
\author{John Enns  and Heejong Lee}

\date{}
\begin{document}
\maketitle

\begin{abstract}
Let $F$ be a totally real field of even degree in which $p$ splits completely. Let $\ov{r}:G_F \rightarrow \GSp_4(\Fpbar)$ be a modular Galois representation unramified at all finite places away from $p$ and upper-triangular, maximally nonsplit, and of parallel weight at places dividing $p$. Fix a place $w$ dividing $p$. Assuming  certain genericity conditions and Taylor--Wiles assumptions, we prove that the $\GSp_4(F_w)$-action on the corresponding Hecke-isotypic part of the space of mod $p$ automorphic forms on a compact mod center form of $\GSp_4$ with infinite level at $w$ determines $\ov{r}|_{G_{F_w}}$.
\end{abstract}

\let\thefootnote\relax


\section{Introduction}
The mod $p$ local Langlands correspondence for the group $\GL_2(\Qp)$ gives a tight connection between continuous Galois representations $G_{\Qp}\rightarrow \GL_2(\Fpbar)$ and admissible smooth $\Fpbar$-representations of $\GL_2(\Qp)$ (see \cite{CDP,CEG+2}) It is hoped that this result generalizes in some form to the group $\GL_n(K)$ for any finite extension $K/\Qp$. However, mod $p$ representations of $\GL_n(K)$ remain poorly understood and the eventual form of the correspondence (should one exist) is still mysterious outside the cases $\GL_2(\Qp)$ or $n=1$. Going further still, in accordance with the general Langlands philosophy one may speculate about the relationship between continuous Galois representations $G_K\rightarrow \LG(\Fpbar)$ (suitably interpreted) and admissible $\Fpbar$-representations of $G(K)$ for any reductive group $G$ over $K$. This paper establishes a concrete connection between certain special classes of these objects when $G=\GSp_4$ and $K=\Qp$.

If $F$ is a number field having $K$ as one of its $p$-adic completions and $G$ is defined over $F$, then generally one hopes that the correspondence should be compatible with the association of Galois representations of $G_F$ to automorphic representations of $G(\A_F)$ -- that is, mod $p$ local-global compatibility. More concretely, if for example $\rhobar:G_K\rightarrow \GL_n(\Fpbar)$ is a local component of a \emph{modular} Galois representation $\rbar:G_F\rightarrow\GL_n(\Fpbar)$ then using spaces of mod $p$ automorphic forms (or \'etale cohomology) it is possible to construct admissible $\Fpbar$-representations $\Pi(\rbar)$ of $\GL_n(K)$ naturally associated with $\rbar$. The question of mod $p$ local-global compatibility is then to study the relationship between $\rhobar$ and $\Pi(\rbar)$. In particular one could ask whether $\Pi(\rbar)$ only depends on $\rhobar$ as opposed to all of $\rbar$. This question has been answered in the affirmative for $\GL_2(\Qp)$ by \cite{Eme11} using the mod $p$ local Langlands correspondence for $\GL_2(\Qp)$, but in no other cases are there results to this effect. Indeed the question is closely related to the existence of a mod $p$ local Langlands correspondence. 

Instead, one can ask the opposite question of whether $\Pi(\rbar)$ uniquely determines $\rhobar$. This is what we do in this paper for $\GSp_4(\Qp)$. Previous work on this question has considered the cases $\GL_2(\Qpf)$ (\cite{BD} and \cite{DL}), $\GL_3(\Qpf)$ (\cite{HLM} and \cite{LMP} for $f=1$, \cite{EnnsThesis} for $f>1$),  $\GL_n(\Qp)$ (\cite{PQ}), and recently $\GL_n(\Qpf)$ (\cite{LLHMPQ}). Except in the most recent work \cite{LLHMPQ},  which was not yet available at the time this paper was developed, one considers a family of representations $\rhobar$ which are upper-triangular with fixed diagonal characters in the Fontaine-Laffaille range (or similar families). One then gives a very explicit recipe using certain group algebra operators in $\Fpbar[G(K)]$ acting on $\Pi(\rbar)$ which pick out the isomorphism class of $\rhobar$ within its family. This knowledge gives us clues as to how the elusive mod $p$ local Langlands correspondence should behave. 

The main idea used in previous work (except \cite{LLHMPQ}) is that $\Pi(\rbar)$ admits a characteristic 0 lift $\tilde{\Pi}$, which is closely related to automorphic representations of $G(\A_F)$. The $p$-adic Galois representations associated to these automorphic representations lift $\rbar$ by construction and one uses their $p$-adic Hodge-theoretic properties at the completion $K$ in order to show that the generic fibre of $\tilde{\Pi}$ recovers $\rhobar$. More specifically, one locates the parameters of $\rhobar$ within potentially crystalline deformation rings of $\rhobar$ of various prescribed types and then uses classical local-global compatibility of automorphic representations to transfer this data to the generic fibre of $\tilde{\Pi}$. Crucially, one must then show that this information can be ``reduced mod $p$'' and is not lost in $\tilde{\Pi}$. This typically involves some form of Taylor-Wiles patching and some subtle modular representation theory with group algebra operators.

In this paper, we work with a version of the patched module $\Minf$ first constructed for $\GL_n(K)$ in \cite{CEG+} instead of $\tilde{\Pi}$. This is an upgrade of $\Pi(\rbar)$ in the sense that it is an $R_{\rhobar}^{\square}[G(K)]$-module, where $R_{\rhobar}^{\square}$ is the universal lifting ring of $\rhobar$, which determines $\Pi(\rbar)$ modulo the maximal ideal of $R_{\rhobar}^\square$. The combined $R_{\rhobar}^\square$- and $G(K)$- actions essentially allow us to axiomatize the relationship between deformations of $\rhobar$ and classical local-global compatibility described in the previous paragraph. This makes $\Minf$ convenient for expressing local-global compatibility. As a result, $\rhobar$ can be recovered by the eigenvalues of certain ``normalized Hecke operators'' acting on the generic fibre of $\Minf$.  In particular, we avoid the complicated arguments with group algebra operators.  An optimistic hope is that $\Minf$ actually realizes the mod $p$ and $p$-adic local Langlands correspondence in general (see \cite[\S6]{CEG+}), so its use here is completely natural.

By taking quotient modulo the maximal ideal of $R_{\rhobar}^\square$,  we recover $\rhobar$ from the eigenvalues of ``normalized Hecke operators'' acting on a certain Iwahori eigenspace in $\Pi(\rbar)$. 
The most subtle part of our argument is proving the non-triviality of the operator.
Namely, we need to show that the Iwahori eigenspace is non-zero and the ``normalized Hecke operator'' acts  on it non-trivially. In the case of $\GL_n$, the former follows from a standard argument using Taylor--Wiles patching and combinatorics between types and weights. Its naive generalization does not work for $\GSp_4$, due to the fixed similitude character in the patching argument. Instead, we construct \textit{a congruent pair of patched modules}, a pair of patched modules with different similitude characters that are congruent modulo $p$. This allows us to perform the usual patching argument in our case.

Another tool we introduce is the Jantzen filtration. The Jantzen filtration arises to compute $p$-divisibility of the image of Carter--Lusztig intertwiners. We use it to describe the image of the mod $p$ reduction of ``normalized Hecke operators'' and thus to determine when the image is non-zero. 

To state our main Theorem, we explain our global setup. We let $\cO$ be a sufficiently large finite extension of $\Z_p$ and $\F$ be the residue field of $\cO$. Let $F$ be a totally real field in which $p$ splits completely, $\ov{r}:G_F \rightarrow \GSp_4(\F)$ be a continuous representation, and $\chi$ be a Hecke character. We write $\omega$ for the mod $p$ cyclotomic character of $G_{\Q_p}$ and $\nr_\xi$ for the unramified character sending geometric Frobenius to $\xi\in \F^\times$. We assume that $\ov{r}$ is unramified at all finite places away from $p$, and upper triangular, maximally nonsplit, and generic at places above $p$. We also assume that $(\rbar,\chi)$ is \emph{potentially diagonalizably automorphic} (Definition \ref{def:pd-automorphic}).     Other technical assumptions on $F,\ov{r},$ and $\chi$, including the usual Taylor--Wiles conditions, can be found in Definition \ref{def:suitable} (also see Remark \ref{rk:assumptions-on-rbar} for comments on the assumptions). The assumptions at places $v|p$ imply that
\begin{align*}
    \ov{r}|_{G_{F_v}} \simeq \begin{pmatrix}
     \omega^{a_3}\nr_{\xi_3} & *_1 &*&* \\ 
     & \omega^{a_2}\nr_{\xi_2} &*_2&*\\
     & & \omega^{a_1}\nr_{\xi_1} &*_3 \\
     & & & \omega^{a_0}\nr_{\xi_0}
    \end{pmatrix}
\end{align*}
where the extensions $*_1,*_2,*_3$ are nonsplit. We call $(a_3,a_2,a_1,a_0)$ the \emph{weight} of $\ov{r}|_{G_{F_v}}$. We choose the weight $(a_3,a_2,a_1,a_0)$ independent of $v|p$. We denote by $S_p$ the set of places of $F$ dividing $p$.  Fix a place $w\in S_p$ and write $\rhobar = \ov{r}|_{G_{F_w}}$.  We parameterize all such $\rhobar$ using Fontaine-Laffaille theory, giving rise to a family of representations depending on two ``Fontaine-Laffaille parameters'' which lie in $\F$. Similarly to previous work, we only consider the ``most generic'' elements of this family. Thus, as in \cite{HLM}, our $\rhobar$ is assumed to be maximally nonsplit with Fontaine-Laffaille invariants lying outside certain special loci. Our precise genericity assumption on $\rhobar$, which we call \emph{strong genericity} can be found in Definition \ref{def:generic}.

Let $\cG$ be an inner form of $\GSp_4$ over $F$ that is compact mod center at infinity and splits at all finite places. Let $\sigma$ be an $\cO$-module with smooth $\GSp_4(\Z_p)$-action. We view $\sigma$ as $\cO[\cG(\cO_{F_v})]$-module for $v|p$ via a chosen isomorphism $\cG(\cO_{F_v}) \simeq \GSp_4(\Z_p)$ and $\sigma^{S_p\backslash \{w\}} := \prod_{v\in S_p \backslash \{w\}}\sigma$ as a smooth representation of $\prod_{v\in S_p\backslash\{w\}}\cG(\cO_v)$. For a dominant weight $\mu$, we let $V(\mu)$ be the algebraic representation of $\prod_{v\in S_p\backslash\{w\}}\cG(\cO_v)$ with highest weight $\mu$.  We choose a level $U^w$ away from $w$ that is sufficiently small and unramified at all places $v\in S_p \backslash\{w\}$. Given $\sigma$, $\mu$ and a Hecke character $\chi'$, we consider the space of mod $p$ automorphic forms
\begin{align*}
    S_{\chi',\mu,\sigma}(U^w,\F) := \varinjlim_{U_w \le \cG(\cO_{F_w})}S_{\chi'}(U^wU_w, V(\mu)\otimes_\cO \sigma^{S_p\backslash\{w\}} \otimes_\cO \F).
\end{align*}
This is an admissible smooth $\GSp_4(F_w)$-module with an action of the abstract Hecke algebra $\T^{P}$ generated by Hecke operators at places away from a finite set $P$ containing $S_p$. Let $\m_{\ov{r}}$ be the maximal ideal of $\T^P$ determined by $\ov{r}$.  Our assumptions on $\ov{r}$  implies that $S_{\chi',\mu,\sigma}(U^w,\F)[\m_{\ov{r}}]\neq 0$ for a certain choice of $\chi', \mu,\sigma$. We take $\chi' = \chi_\pcris$ (defined in \S \ref{sec:patching}),  $\mu=0$, and $\sigma = \sigma(\tau_0)$ where $\tau_0$ is carefully chosen tame principal series type (Definition \ref{def:tau0}) and $\sigma(\tau_0)$ denotes the smooth representation of $\GSp_4(\cO_{F_v})$ corresponding to $\tau_0$ under the inertial local Langlands correspondence for principal series representation (Lemma \ref{lem:K-type}).   Our main result is the following.

\begin{theorem*}[Theorem \ref{thm:main result}]
Following the above notations,  the Fontaine--Laffaille invariants of $\rhobar$ can be recovered from  the admissible smooth $\F[\GSp_4(F_w)]$-module $\Pi(\ov{r}):=S_{\chi_\pcris,0,\sigma(\tau_0)}(U^w,\F)[\m_{\ov{r}}]$. 
\end{theorem*}

We now give a more precise description of the contents of this article. Let  $\rhobar: G_{\Qp} \rightarrow \GSp_4(\F)$ be   a continuous upper triangular, maximally nonsplit, and strongly generic representation. 
The first part of the argument is to find Fontaine-Laffaille parameters of $\rhobar$ inside certain symplectic potentially crystalline deformation rings of $\rhobar$ of type $(\eta,\tau)$, where $\eta=(3,2,1,0)$ is a Hodge type and $\tau$ a tame inertial type. In fact, it will suffice for us to consider the single tame principal series type $\tau_0$. We begin in \S\ref{sec:Kisin modules} by giving an explicit description of these deformation rings by adapting the work of \cite{LLHLM1} and \cite{LLL} (for $\GL_n$) to the case of $\GSp_4$ using some ideas of \cite{KM} about duality for Kisin modules. Using this, we prove that the strong genericity of $\rhobar$ implies $R_{\rhobar}^{\eta,\tau_0}$ is formally smooth (Theorem \ref{thm:main result galois}). (We actually work with a fixed similitude character as well.)

The Fontaine-Laffaille parameters of $\rhobar$ can be found as the reduction modulo the maximal ideal of certain ``universal Frobenius eigenvalues'' in $R_{\rhobar}^{\eta,\tau_0}$. We make this completely explicit in Theorem \ref{thm:main result galois}. In order to apply this result in the setting of the patched module $\Minfty$, we express these universal eigenvalues in terms of a morphism $\Theta_{\rhobar,\tau_0}:\cH(\sigma(\tau_0))\rightarrow R_{\rhobar}^{\eta,\tau_0}[1/p]$ which interpolates the local Langlands correspondence for $\GSp_4(\Qp)$ in the Bernstein block corresponding to $\tau_0$. Here $\cH(\sigma(\tau_0))$ is the Hecke algebra of the $K$-type corresponding to the inertial type $\tau_0$. This follows ideas of \cite{CEG+} and is described in \S\ref{sec:LLC}. 

In \S\ref{subsec:patched module axioms} we prove an abstract local-global compatibility result (Theorem \ref{thm:abstract loc glob}) assuming the existence of a {congruent pair of patched modules} $(\Minf, \Minf^\cris)$ satisfying certain axioms \textbf{(PM1)}--\textbf{(PM6)}. 
This result shows how the Fontaine--Laffaille invariants of $\rhobar$ may be recovered from the admissible smooth $\F[\GSp_4(\Qp)]$-representation $\pi$ related to $\Minf$. See \S\ref{subsec:patched module axioms} for precise construction of $\pi$ using $\Minf$. One of the axioms (which follows from classical local-global compatibility of automorphic representations in the construction of $\Minfty$) states that the universal Frobenius eigenvalues may be recovered as the eigenvalue of a certain ``$U_p$-operator'' $T_{\lambda}$ in the pro-$p$ Iwahori--Hecke algebra acting on well-chosen Iwahori eigenspaces of (the dual of) $\Minf$, up to powers of $p$. One therefore needs to pick out the leading term of these eigenvalues modulo $p$. To do this, one writes the action of $T_{\lambda}$ on $\Minfty$ as the composite of a Carter--Lusztig intertwiner and an Iwahori normalizing element in $\GSp_4(\Qp)$ (see \eqref{eq:Pi decomposition}). The power of $p$ by which the Carter--Lusztig intertwiner is divisible is described by the Jantzen filtration of principal series representations of $\GSp_4(\Fp)$. We show in Theorem \ref{thm:abstract loc glob} that the Carter--Lusztig intertwiner is divisible by the ``correct'' power of $p$ on these Iwahori eigenspaces and consequently that the Fontaine--Laffaille invariant can be recovered from $\pi$. To do this, we show that there is a unique modular Serre weight occurring in $\JH(\overline{\sigma(\tau_0)})$. This requires the combinatorics between types and weights using patching argument. On the Galois side, this requires both potentially crystalline and crystalline lifts of $\rhobar$. However, it is necessary to fix the similitude character in the patching argument. That is why we have the second patched module $M_\infty^\cris$ designed for the crystalline setting, and the congruence between the two patched modules \textbf{(PM5)} allows us to prove the uniqueness of the modular Serre weight. 
Finally, we show that the modular Serre weight appears in the correct layer of the Jantzen filtration.  This requires a somewhat tedious calculation with the Jantzen sum formula. We also interpret this theorem in terms of group algebra operators  in Remark \ref{rk:group algebra operators}. All this is contained in \S\ref{subsec:patched module axioms}. The computation with the Jantzen filtration is contained in \S\ref{sec:jantzen filtration} and \S\ref{sec:representation theory}. 

The goal of \S\ref{sec:existence} is to show the existence of a congruent pair of patched modules $(\Minfty,\Minf^\cris)$ obeying the axioms of \S\ref{subsec:patched module axioms}. We accomplish this under certain conditions by ultrapatching spaces of $p$-adic automorphic forms on a compact mod centre form $\cG$ of $\GSp_4$ over a totally real field $F$, essentially following \cite{CEG+} but with ideas from \cite{BCGP} for the symplectic case. The congruence between patched modules is obtained by choosing the same Taylor--Wiles datum and ultrafilter. 
In order to attach Galois representations to a regular algebraic cuspidal automorphic representation of $\cG$, we need to apply Jacquet--Langlands for $\GSp_4$ proven in \cite{Sor09} under stable and tempered assumptions. In Lemma \ref{lem:stable-tempered}, we show that automorphic representations of $\cG$ of our concern are indeed stable and tempered by using various results on Arthur multiplicity formula \cite{Art13,Taibi,GeeTaibi}. 
In \S\ref{sec:patching}, we assume that $\rhobar$ is a local component of a Galois representation $\ov{r}: G_F \rightarrow \GSp_4(\F)$ modular with respect to $\cG$ satisfying several assumptions \textbf{(A1)}--\textbf{(A5)} including standard Taylor--Wiles type assumptions, but also an unramifiedness away from $p$ assumption that is made mainly for simplicity. In Corollary \ref{lem:modularity of ordinary weight}, we prove the modularity of the obvious weight \textbf{(PM6)} by the change of weight argument using the main result of \cite{PT}. We also show the existence of such globalization $F$ and $\ov{r}$ for given $\rhobar$ by following \cite{EG}; see Corollary \ref{cor:globalization}. Then our main result showing that the Fontaine-Laffaille invariants can be recovered from $\Pi(\rbar)$ is Theorem \ref{thm:main result}.

We expect that most of this article should generalize with the same methods to $\Qpf$ for any $f\geq 1$. However, for simplicity we have tended to state results only in the generality that we need.  

Recently, \cite{LLHMPQ} proved mod $p$ local-global compatibility for $\GL_n(\Q_{p^f})$ and generic Fontaine--Laffaille (but not necessarily upper-triangular) $\rhobar$ using a more geometric argument. We expect that one can adapt this idea to $\GSp_4(\Q_{p^f})$ and prove  mod $p$ local-global compatibility in a similar generality. We plan to investigate this further in our future work.

\subsection{Acknowledgements}
The authors would like to thank Florian Herzig and Bao Viet Le Hung for helpful advice during the undertaking of this project. The first-named author would also like to thank Karol Kozio\l, Stefano Morra, and Zicheng Qian for helpful conversations and correspondence. The second-named author would like to thank Olivier Ta\"ibi for helpful correspondence. We would like to thank Ariel Weiss for a helpful comment on an earlier version of this paper. 

\subsection{Notation and preliminaries}\label{sec:notation}
Throughout we fix a prime $p>2$. Most of our results require $p$ to be larger than this due to genericity assumptions. 

We let $E$ denote a finite extension of $\Qp$ which will serves as a field of coefficients. We always assume that $E$ is sufficiently large. Write $\cO$ and $\F$ for its ring of integers and residue field. We write $\Qpbar$ for a choice of algebraic closure of $E$. Let $\CO$ denote the category of artinian local $\cO$-algebras with residue field $\F$, and $\COhat$ the category of complete noetherian local $\cO$-algebras with residue field $\F$. 

We sometimes write $K$ instead of $\Qp$. Let $(p_n)_{n\geq 0}$ denote a compatible choice of $p^n$th roots of $-p$, with $p_0=-p$ and define $K_\infty=\bigcup_{n\geq 0}\Qp(p_n)$. Let $e:=p-1$ and choose roots $(\pi_n)_{n\geq 0}$ where $\pi_n^e=p_n$ and $\pi_{n+1}^p=\pi_n$. Define $L=\Qp(\pi_1)$ and  $L_\infty=\bigcup_{n\geq 0}\Qp(\varpi_n)$. Let $\Delta:=\Gal(L/K)\cong \Gal(L_\infty/K_\infty)$. Let $\epsilon:G_{\Qp}\rightarrow \Zp\x$ denote the $p$-adic cyclotomic character and write $\omega$ for its reduction mod $p$, which factors through $\Delta$. Let  $\tomega$ denote the Teichm\"uller lift of $\omega$. 

The symbol $F$ will denote a number field. If $F_v$ is a completion at a finite place we write $\Frob_v$ for geometric Frobenius. Write $\Art_{\Qp}:\Qp\x\rightarrow G_{\Qp}\ab$ for the Artin reciprocity map of local class field theory, normalized so that uniformizers correspond to geometric Frobenii. We write $\nr_\xi$ for the unramified character of $G_{\Qp}$ sending geometric Frobenius $\Frob\mapsto \xi$. We choose the convention on Hodge-Tate weights whereby $\epsilon$ has Hodge-Tate weight $+1$.

Write $B_4$ and $T_4$ for the upper-triangular Borel and diagonal maximal torus of $\GL_4$ respectively. Let $S_4$ be the group of permutations of $\{0,1,2,3\}$. We identify $X^*(T_4)\cong \Z^4$ in the usual way and identify the Weyl group of $\GL_4$ with $S_4$ via the embedding $S_4\hookrightarrow \GL_4$ which takes $\sigma\in S_4$ to the monomial matrix $M_{\sigma}$ such that $(A\sigma)_{ij}=A_{\sigma(i)\sigma(j)}$ for all $A\in \GL_4$. Let $w_0\in S_4$ denote the longest element. Let $\tW\dual_4$ denote the extended affine Weyl group of $\GL_4$, identified with $X_*(T_4)\rtimes S_4$ respectively (the dual notation is chosen to be compatible with \cite{LLL}). We write all vectors as row vectors so matrices act on the right. Let $\eta=(3,2,1,0)\in X_*(T_4)$, except in \S\ref{sec:Kisin modules}, where it is allowed to be a more general cocharacter.

Let $\GSp_4$ denote the reductive group over $\Z$ defined by
\begin{equation*}
\GSp_4(R)=\{A\in \GL_4(R)\,|\, A^tJA=\nu J\textrm{ for some }\nu\in R\x\}
\end{equation*}
for any commutative ring $R$, where
\begin{equation*}
J=\begin{pmatrix}&&&1\\&&1&\\&-1&&\\-1&&& \end{pmatrix}.
\end{equation*}  
Write $\nusim:\GSp_4\rightarrow\Gm$ for the similitude character taking $A\mapsto \nu$. We also write $\std:\GSp_4 \rightarrow \GL_4$ for the standard representation and $\std': \GSp_4 \rightarrow \GL_5$ for the composition of the projection $\GSp_4 \onto \SO_5$ with the standard representation of $\SO_5$. Let $T\subset \GSp_4$ be the torus of diagonal matrices and $B\subset \GSp_4$ the Borel subgroup of upper-triangular matrices. 
We identify $X^*(T)$ with $\{(a,b;c)\in \Z^3\,|\,c\equiv a+b\mod 2\}$, where $(a,b;c)$ is the character 
\begin{equation*}
\begin{pmatrix}t_1&&&\\&t_2&&\\&&\nu t_2^{-1}&\\&&&\nu t_1^{-1}\end{pmatrix}\mapsto t_1^a t_2^b \nu^{\frac{c-a-b}{2}}.
\end{equation*}
There is an isomorphism
\begin{align*}
\mathrm{spin}:X^*(T)\xrightarrow{\sim}& X_*(T) \\ 
\mu=(a,b;c)\mapsto& \bar{\mu}
\end{align*}
where 
\begin{equation*}
\bar{\mu}(t):= \begin{pmatrix}t^{\frac{c+a+b}{2}}&&&\\& t^{\frac{c+a-b}{2}}&&\\&&t^{\frac{c-a+b}{2}}&\\&&&t^{\frac{c-a-b}{2}}\end{pmatrix}.
\end{equation*}
The system of positive roots determined by $B$ consists of $R^+=\{\alpha_0,\alpha_1,\alpha_0+\alpha_1,2\alpha_0+\alpha_1\}$, where $\alpha_0=(1,-1;0)$ and $\alpha_1=(0,2;0)$. Let $W$ denote the Weyl group and $s_i\in W$ the simple reflection corresponding to root $\alpha_i$ for $i=0,1$. Then $W=\pair{s_0,s_1\,|\,s_0^2,s_1^2,(s_0s_1)^4}$ has size 8, with the longest element being $(s_0s_1)^2$. 

We implicitly identify $X_*(T)$ with a subgroup of $X_*(T_4)$, in which case we also identify $W$ with the subgroup of $S_4$ of permutations $s$ obeying $w_0sw_0=s$. We fix the following elements which form a basis of $X_*(T)$:
\begin{align*}
\beta_0&=(1,1,1,1)\\
\beta_1&=(1,1,0,0)\\
\beta_2&=(2,1,1,0).
\end{align*}

Let $\Gamma$ be any topological group and $R$ any topological ring. Fix a continuous character $\psi:\Gamma\rightarrow R\x$. By $\Rep^\psi_R(\Gamma)$ we mean the groupoid of pairs $(V,\iota)$ where $V$ is a finite free $R$-module of rank 4 having a continuous $R$-linear $\Gamma$-action, and $\iota:V\xrightarrow{\sim}V\dual\otimes\psi$ is a which obeys the condition that \begin{equation*}
[(\iota\dual)^{-1}\otimes\nu]\circ\iota=-\id_V.
\end{equation*}
 In Section \ref{sec:Kisin modules} we will also implicitly work with the setoid of triples $(V,\iota,\delta)$, where $\delta$ is a morphism from $(V,\iota)$ to $R^4$ with the standard symplectic pairing. The set of equivalence classes of this setoid naturally identifies with the set of continuous homomorphisms $\Gamma\rightarrow \GSp_4(R)$. 

The following lemma is useful in \S\ref{sec:p-adic automorphic forms}.
\begin{lemma}\label{lem:symplectic reduction mod p}
Let $\Gamma$ be a compact group and $\rho:\Gamma\rightarrow \GSp_4(E)$ a continuous homomorphism. Let $\rhobar$ be the semisimple mod $p$ reduction of $\rho$. If $\rhobar$ is irreducible and $\F$ is sufficiently large then there exists a $\GL_4(E)$-conjugate $\rho^{\circ}$ of $\rho$ which is valued in $\GSp_4(\cO)$ and has the same similitude character as $\rho$. The resulting (irreducible) representation $\rhobar^{\circ}:\Gamma\rightarrow \GSp_4(\F)$ does not depend on the choice of $\rho^{\circ}$ up to $\GSp_4(\F)$-conjugation. \qed
\end{lemma}

We recall some aspects of the local Langlands correspondence for $\GSp_4(F_v)$. The symbol $\recGT$ denotes the finite-to-one map from isomorphism classes of irreducible smooth $\C$-representations of $\GSp_4(F_v)$ to equivalence classes of admissible $L$-parameters constructed in \cite{GT}. 
Fix an isomorphism of fields $\iota:\Qpbar\xrightarrow{\sim}\C$. This induces a correspondence $\mathrm{rec}_{\mathrm{GT},\iota}$ over $\Qpbar$. We define a normalized version of the correspondence by
\begin{equation*}
 \recGTp(\pi):=\mathrm{rec}_{\mathrm{GT},\iota}(\pi\otimes |\nusim|^{-3/2}).
 \end{equation*}
It should be true that $\recGTp$ depends only on a choice of square root of $p$ in $\Qpbar$, but we don't need to use this. 

We use $\Ind$ to denote the \emph{unnormalized} parabolic induction functor for $\GSp_4(F_v)$ and $\ind$ for compact induction.
 
 If $M$ is a topological $\cO$-module then $M\dual=\Hom^\cts_\cO(M,E/\cO)$ denotes its Pontryagin dual. We also use Schikof duality $M\mapsto M^d$ which is described in the Notation section of \cite{CEG+}. 

\section{Symplectic Galois deformations}\label{sec:galois side}
Let $\rhobar:G_{\Qp}\rightarrow \GSp_4(\F)$ be a continuous representation. In this section we study symplectic deformations of $\rhobar$ with a fixed similitude character $\psi:G_{\Qp}\rightarrow \cO\x$. Let $R_{\rhobar}^{\square,\psi}$ denote the universal lifting ring for lifts $\rho:G_{\Qp}\rightarrow \GSp_4(\cO)$ such that $\nusim\circ\rho=\psi$. In \S\ref{sec:Kisin modules} we adapt techniques of \cite{LLHLM1} and \cite{KM} to give an explicit description of potentially crystalline quotients of $R_{\rhobar}^{\square,\psi}$ with Hodge type $\eta$ and generic tame principal series inertial type. Using this, in \S\ref{sec:LLC} we construct an interpolation of the characteristic 0 tame principal series local Langlands correspondence for $\GSp_4(\Qp)$, which is an ingredient in the axioms for our patched module. In \S\ref{sec:rhobar} we define a family of representations $\rhobar$ using Fontaine-Laffaille theory and prove the main result Theorem \ref{thm:main result galois}. This theorem shows where to find the data of $\rhobar$ inside certain of the aforementioned deformation rings. In \S\ref{sec:crystalline deformations} we study crystalline quotients of $R_{\rhobar}^{\square,\psi}$ with Hodge-Tate weights in the Fontaine-Laffaille range when $\rhobar$ is ordinary. 

\subsection{Some potentially crystalline deformation rings}\label{sec:Kisin modules} 
If $\tau:I_{\Qp}\rightarrow \GL_4(\cO)$ is an inertial type and $\eta\in X_*(T_4)$ a Hodge type, we let $R_{\rhobar}^{\eta,\tau,\psi}$ denote the unique $p$-torsion free quotient of $R_{\rhobar}^{\square,\psi}$ such that for any local finite $E$-algebra $B$, a morphism of $\cO$-algebras $\zeta:R_{\rhobar}^{\square,\psi}\rightarrow B$ factors through $R_{\rhobar}^{\square,\psi}$ iff $\zeta\circ\rho^\univ$ (considered as a $\GL_4(B)$-valued representation) is potentially crystalline of Hodge type $\eta$ and inertial type $\tau$. This quotient exists by Corollary 2.7.7 of \cite{Kis08}. 

\begin{lemma}\label{lem:dimension of p crys def ring}
If $R_{\rhobar}^{\eta,\tau,\psi}[1/p]$ is nonzero then it is formally smooth over $E$ of dimension 14. 
\end{lemma}
\begin{proof}
This can be proved using the argument of \cite[Proposition 7.2.1]{GG}. 
\end{proof}

In this section we give an explicit description of $R_{\rhobar}^{(3,2,1,0),\tau,\psi}$ when $\tau$ is a generic tame principal series type. For now let $\eta$ be arbitrary. We assume that $\eta$ is contained in $[0,h]$ for some $h\geq 0$, and not contained in $[0,h']$ for any $h'<h$. In order for $R_{\rhobar}^{\eta,\tau,\psi}$ to be nonzero we must have $\psi=\epsilon^h\tomega^b\nr_{\xi}$ for some $b\in \Z$ and $\xi\in\cO\x$, so we assume this for the rest of \S\ref{sec:Kisin modules}. In particular $\nusim\circ\rhobar=\omega^{h+b}\nr_{\bar{\xi}}$. 

The strategy is to relate $R_{\rhobar}^{\eta,\tau,\psi}$ to deformations of Kisin modules, adapting methods of \cite{LLHLM1} and \cite{KM} to the symplectic case. We begin by recalling some notation and results from \cite{LLL} concerning Kisin modules with tame descent data over $\Qp$. 

\begin{remark}
Although we restrict attention to the case of tame principal series types (as this is all we need), everything in this section can be generalized to generic higher niveau tame types over an arbitrary unramified extension of $\Qp$, as in \cite{LLHLM1}. 
\end{remark}

If $\lambda=(b_0,b_1,b_2,b_3)\in X^*(T_4)$, we write
\begin{equation*}
\tau(1,\lambda)=\bigoplus_{i=0}^3 \,\tomega^{b_i}:I_{\Qp}\rightarrow T_4(\cO)
\end{equation*}
for the associated tame principal series type (the notation comes from \cite{LLL}). 
\begin{definition}[cf. Definition 2.2.5 in \cite{LLL}]
Let $\delta\geq0$. We say that a tame principal series type $\tau:I_{\Qp}\rightarrow T_4(\cO)$ is \emph{$\delta$-generic} if it is isomorphic to $\tau(1,\lambda)$ where 
$\delta<b_i-b_j<p-\delta$ for each $0\leq i<j\leq 3$. 

A \emph{lowest alcove presentation} of $\tau$ is a choice of $\lambda$ obeying this condition with $\delta=0$ such that $\tau\cong \tau(1,\lambda)$. 
\end{definition}

A lowest alcove presentation exists precisely when $\tau$ is 0-generic. A necessary condition for $R_{\rhobar}^{\eta,\tau,\psi}$ to be nonzero is that its lowest alcove presentation $\lambda=(b_0,b_1,b_2,b_3)$ obeys $b_0+b_3\equiv b_1+b_2\equiv b\mod e$, so without loss of generality we can assume
\begin{equation}\label{eq:symplecticity condition on tau}
b_0+b_3= b_1+b_2= b.
\end{equation}
From now on we let $\tau$ denote a 1-generic tame principal series type with lowest alcove presentation $\lambda$ satisfying \eqref{eq:symplecticity condition on tau}.

Let $R$ denote a complete noetherian local $\cO$-algebra with residue field $\F'$ which is a finite extension of $\F$. The group $\Delta$ has an action on $R[[u]]$ by ring automorphisms uniquely determined by the formula $\gamma(ru^i)=\tomega(\gamma)^iru^i$ for $\gamma\in \Delta, r\in R$. Set $v=u^e$ and $P(v)=v+p\in R[u]$. Note that $R[[u]]^{\Delta=1}=R[[v]]$. Let $\varphi:R[[u]]\rightarrow R[[u]]$ denote the $R$-linear ring morphism sending $u\mapsto u^p$. 
\begin{definition}
We let $Y^{[0,h],\tau}(R)$ denote the groupoid of rank 4 Kisin modules over $R$ of height $\leq h$ and descent data of type $\tau$. Objects are free $R[[u]]$-modules $M$ of rank $4$ with a $\varphi$-semilinear endomorphism $\varphi_M:M\rightarrow M$ such that $P(v)^h$ kills the cokernel of the linearization of $\varphi_M$. Moreover, $M$ has a semilinear $\Delta$-action which commutes with $\varphi_M$ such that $M/uM\cong \tau^{-1}\otimes_\cO R$ as an $R[\Delta]$-module (note the minus sign). Morphisms in this category are the obvious ones. 
\end{definition}

\begin{definition}
If $M\in Y^{[0,h],\tau}(R)$ an \emph{eigenbasis} of $M$ is an $R[[u]]$-basis $\beta=(\beta_0,\beta_1,\beta_2,\beta_3)$ such that $\Delta$ acts on $\beta_i$ by $\tomega^{-b_i}$ for $0\leq i\leq3$. Eigenbases always exist.
\end{definition}

\begin{definition}
Let $M\in Y^{[0,h],\tau}(R)$ with eigenbasis $\beta$. Let $C_\beta\in \Mat_4(R[[u]])\cap \GL_4(R((u)))$ denote the matrix such that
\begin{equation*}
\varphi_M(\beta)=\beta \cdot C_\beta.
\end{equation*}
Let $A_\beta\in \Mat_4(R[[v]])\cap \GL_4(R((v)))$ denote the matrix defined in Proposition 3.2.9 of \cite{LLL}. 
\end{definition}

We have
\begin{equation}
A_\beta=D^{-1}C_\beta D
\end{equation}
where $D=\Diag(u^{b_0-b_3},u^{b_1-b_3},u^{b_2-b_3},1)$. Note that $A_\beta$ is upper triangular mod $v$.

\begin{definition}[cf. Definition 5.1.1 in \cite{KM}]
If $M\in Y^{[0,h],\tau}(R)$ we define a new object $M\dual\in Y^{[0,h],\tau}(R)$ by setting
\begin{equation*}
M\dual=\Hom_{R[[u]]}(M,R[[u]])
\end{equation*}
and defining $\varphi_{M\dual}:M\dual\rightarrow M\dual$ by 
\begin{equation*}
\varphi_{M\dual}(f)(m)=\xi \cdot\varphi(f(\varphi_M^{-1}(P(v)^hm)))
\end{equation*}
for all $f\in M\dual, m\in M$. This definition makes sense because $\varphi_M$ is automatically injective, by an argument similar to Lemma 1.2.2(1) in \cite{Kis09}.

We endow $M\dual$ with descent data by defining $(g\cdot f)(m)=\tomega(g)^{-b}g\cdot f(g^{-1}\cdot m)$ for $g\in \Delta,f\in M\dual,m\in M$. One checks that this defines an involutive functor
$Y^{[0,h],\tau}(R)\rightarrow Y^{[0,h],\tau}(R)$.
\end{definition}

\begin{lemma}\label{lem:A dual}
Let $M\in Y^{[0,h],\tau}(R)$. If $\beta$ is an eigenbasis of $M$ then $\beta\dual w_0$ and $\beta\dual J$ are eigenbasis of $M\dual$. We have $A_{\beta\dual w_0}=\xi \cdot P(v)^h w_0 A_\beta^{-t} w_0$ and $A_{\beta\dual J}= \xi\cdot P(v)^h J A_\beta^{-t} J$. 
\end{lemma}
\begin{proof}
This follows from a simple computation. 
\end{proof}

Let $\cI(R)$ denote the subgroup of $\GL_4(R[[v]])$ consisting of matrices that are upper triangular mod $v$. In what follows we identify $\tW\dual_4$ with a subgroup of $\GL_4(R((v)))$ by sending $t_\lambda s\mapsto v^\lambda s$ for $\lambda\in X_*(T_4)$ and $s\in S_4$. Thus $\GL_4(\F'((v)))$ is the disjoint union of $\cI(\F')\tw\cI(\F')$ for $\tw\in\tW\dual_4$. 

\begin{definition}[cf. Definition 3.2.11 in \cite{LLL}]
If $M\in Y^{[0,h],\tau}(\F')$ and $\beta$ is an eigenbasis we the \emph{shape of $\beta$} to be the unique $\tw\in \tW\dual_4$ such that $A_{\beta}\in \cI(\F')\tw\cI(\F')$. The 1-genericity of $\tau$ ensures that the shape of $\beta$ is independent of $\beta$, so we call $\tw$ the \emph{shape} of $M$.

If $\tw\in\tW\dual_4$ we define the set of matrices $U_{\tw}(\F')\subset \GL_4(\F'((v)))$ as in the paragraphs before Definition 3.2.23 in \cite{LLL}, and we say that $\beta$ is a \emph{gauge basis} of $M$ if it is an eigenbasis such that $A_{\beta}\in \tw U_{\tw}(\F')$.
 \end{definition}

\begin{lemma}\label{lem:uniqueness of gauge basis mod p}
Assume that $\tau$ is 4-generic. Let $M\in Y^{[0,h],\tau}(\F')$. The set of gauge bases of $M$ is a torsor for $T_4(\F')$. 
\end{lemma}
\begin{proof}
The existence of a gauge basis follows immediately from the definition. The remark after Example 3.2.24 in \cite{LLL} shows that since $\tau$ is 4-generic, gauge bases are unique up to scaling by diagonal matrices.
\end{proof}
\begin{definition}\label{def:gauge basis}
Let $M\in Y^{[0,h],\tau}(R)$ and suppose that $M\mod\m_R$ is of shape $\tw$. We say that an eigenbasis $\beta$ of $M$ is a \emph{gauge basis} if 
\begin{itemize}
\item $\beta\mod \m_R$ is a gauge basis of $M\mod\m_R$ in the sense above, and
\item $A_{\beta}\in {\cI}^-(R)\tw\cap \cI(R)$, where ${\cI}^-(R)$ is defined to be the subgroup of $\GL_n(R[v^{-1}])$ consisting of matrices that are lower-triangular mod $v^{-1}$. 
\end{itemize}
The second condition is equivalent to asking that $A_\beta$ satisfies the degree bounds in Proposition 3.4.3 of \cite{LLL}. 
\end{definition}
\begin{proposition}[cf. Proposition 3.4.3 of \cite{LLL}]\label{prop:uniqueness of gauge basis}
Assume that $\tau$ is 4-generic. Let $M\in Y^{[0,h],\tau}(R)$ and fix a gauge basis $\bar{\beta}$ of $M/\m_R$. Then there exists a gauge basis $\beta$ of $M$ lifting $\bar{\beta}$, and $\beta$ is unique up to multiplication by an element of $\ker(T_4(R)\rightarrow T_4(\F'))$. 
\end{proposition}
\begin{proof}
This is a straightforward generalization of the argument of \cite[\S4]{LLHLM1} from $\GL_3$ to $\GL_4$. 
\end{proof}

\begin{lemma}\label{lem:duality and shapes}
 If $M\in Y^{[0,h],\tau}(\F')$ has shape $\tw$ then $M\dual$ has shape $v^h w_0\tw^{-t}w_0$.
If $\beta$ is a gauge basis of $M\in Y^{[0,h],\tau}(R)$ then $\beta\dual J$ is a gauge basis of $M\dual$. 
\end{lemma}
\begin{proof}
The first statement follows from a computation using Lemma \ref{lem:A dual}. The proof of the second statement then follows in the case $R=\F'$ by a computation with affine root groups (and the definition of $U_{\tw}(\F')$). Given this the case of general $R$ now follows from Lemma \ref{lem:A dual} and the definition of gauge basis. 
\end{proof}

If $M$ is a Kisin module with descent data we let $\cM=M\otimes_{\Zp[[u]]}\cO_{\cE,L}$ denote the associated \'etale $\varphi$-module with descent data defined in \S2.3 of \cite{LLHLM1}. We have the contravariant functor to $\GKinfty$-representations (\emph{loc.\,cit.})
\begin{equation*}
T_{\dd}^*:Y^{[0,h],\tau}(R)\rightarrow \Rep_R(\GKinfty).
\end{equation*}
given by 
$T_{\dd}^*(M):=V_{\dd}^*(\cM):=\Hom_{\varphi,\cO_{\cE,L}}(\cM,\cO_{\cE^{\un},K})$
given a $\GKinfty$-action via $g\cdot f=g\circ f\circ \bar{g}^{-1}$. Here $\bar{g}$ denotes the image of $g$ in $\Gal(L_\infty/K_\infty)\cong \Delta$. 

\begin{proposition}\label{prop:duality and Tdd}
If $M\in Y^{[0,h],\tau}(R)$ there is a canonical isomorphism of $\GKinfty$-representations 
\begin{equation*}
\can:T_{\dd}^*(M\dual)\rightarrow  T_{\dd}^*(M)\dual\otimes_R \psi|_{G_{K_\infty}}
\end{equation*}
where the dual on the right hand side is the $R$-linear dual. Moreover, if $f:M\rightarrow N$ is a morphism of Kisin modules then under the identification above we have $T_{\dd}^*(f\dual)=T_{\dd}^*(f)\dual\otimes_R\psi|_{G_{K_\infty}}$. 
\end{proposition}
\begin{proof}
This can be shown similar to Proposition 3.4.1.7 of \cite{Broshi} (cf. \cite[Proposition 5.2]{KM}). 
\end{proof}

\begin{definition}\label{def:symplectic Kisin modules}
(i) We define $Y^{[0,h],\tau,\psi}(R)$ to be the groupoid of pairs $(M,\alpha)$ where $M\in Y^{[0,h],\tau}(R)$ and $\alpha:M\rightarrow M\dual$ is an isomorphism of Kisin modules such that
\begin{equation*}
(\alpha\dual)^{-1}\circ\alpha=-\id_M.
\end{equation*}
This equation will be referred to as the \emph{alternating condition}. Morphisms in this category are the obvious ones. 

(ii) We say that $(M,\alpha)\in Y^{[0,h],\tau,\psi}(\F')$ has \emph{shape} $\tw\in\tW\dual_4$ if $M$ has shape $\tw$.

(iii) If $(M,\alpha)\in Y^{[0,h],\tau,\psi}(R)$, we define an \emph{eigenbasis} (resp. a \emph{gauge basis}) of $(M,\alpha)$ to be an eigenbasis (resp. gauge basis) $\beta$ of $M$ such that 
\begin{equation}\label{eq:symplectic gauge basis}
\alpha(\beta)=\beta\dual J. 
\end{equation}
\end{definition}

\begin{remark}
(i) If $(M,\alpha)\in Y^{[0,h],\tau,\psi}(\F')$ has shape $\tw$ then we must have $\tw=v^3 w_0\tw^{-t}w_0$ by Lemma \ref{lem:duality and shapes}(i). 

(ii) The existence of a basis satisfying \eqref{eq:symplectic gauge basis} implies that $\alpha$ satisfies the alternating condition. 
\end{remark}

\begin{lemma}\label{lem:symplectic gauge bases} Assume that $\tau$ is 4-generic.

(i) Let $(M,\alpha)\in Y^{[0,h],\tau,\psi}(\F')$. The set of gauge bases of $(M,\alpha)$ is nonempty and is a torsor for $T'(\F')$, where $T'$ is the diagonal torus of $\Sp_4$. 

(ii) Let $(M,\alpha)\in Y^{[0,h],\tau,\psi}(R)$ and fix a gauge basis $\bar{\beta}$ of $M/\m_R$. The set of gauge bases of $(M,\alpha)$ lifting $\bar{\beta}$ is nonempty and is a torsor for $\ker(T'(R)\rightarrow T'(\F'))$.
\end{lemma}
\begin{proof}
(i) Let $\beta$ be any gauge basis of $M$. Then $\alpha(\beta)$ is a gauge basis of $M\dual$. By Lemma \ref{lem:uniqueness of gauge basis mod p} and Lemma \ref{lem:duality and shapes}(ii) we have $\alpha(\beta)=\beta\dual J t^{-1}$ for some $t\in T_4(\F')$. The alternating condition implies that $tJt^{-1}=J$. An arbitrary gauge basis of $M$ is of the form $\beta c^{-1}$ for some $c\in T_4(\F')$ by Lemma \ref{lem:uniqueness of gauge basis mod p}. Replacing $\beta$ by $\beta c^{-1}$, the condition \eqref{eq:symplectic gauge basis} becomes equivalent to $cJc=Jt^{-1}$, which can be solved in $T_4(\F')$ by the condition on $t$ above. This proves the existence of gauge bases.

 If $\beta$ and $\beta c^{-1}$ are both gauge bases of $(M,\alpha)$ for some $c\in T_4(\F')$ then \eqref{eq:symplectic gauge basis} implies $cJc=J$, i.e. $c\in T'(\F')$. This proves the uniqueness. 

(ii) The proof follows the same strategy as (i), but appealing to Proposition \ref{prop:uniqueness of gauge basis}.
\end{proof}

We define a contravariant functor 
\begin{equation}
T_{\dd}^*:Y^{[0,h],\tau,\psi}(R)\rightarrow \Rep_R^\psi(\GKinfty)
\end{equation}
by sending $(M,\alpha)$ to $(T_{\dd}^*(M),\can\circ T_{\dd}^*(\alpha)^{-1})$.

We now put Hodge type $\eta$ conditions on our Kisin modules. Recall from \cite{LLHLM1} the subgroupoid $Y^{\eta,\tau}(R)\subseteq Y^{[0,h],\tau}(R)$. 
\begin{definition}
We define the groupoid $Y^{\eta,\tau,\psi}(R)\subseteq Y^{[0,h],\tau,\psi}(R)$ to consist of the pairs $(M,\alpha)$ with $M\in Y^{\eta,\tau}(R)$. 
\end{definition}

Finally we can establish the connection between Kisin modules and potentially crystalline lifting rings.

\begin{notation}
In what follows, if $R$ is a symplectic lifting ring (perhaps with various super- and subscripts) then $R'$ (with the same decorations minus the fixed similitude character) denotes the corresponding $\GL_4$-valued lifting ring. For example, $R_{\rhobar}^{\eta,\tau,\psi}$ having been defined at the beginning of this section, we write $R_{\rhobar}'^{\eta,\tau}$ for the $\GL_4$-valued potentially crystalline lifting ring of $\rhobar$ of type $(\eta,\tau)$. 
\end{notation}

\begin{proposition}\label{prop:mod p Galois and Kisin}
Assume that $\tau$ is 4-generic. 
\begin{itemize}
\item[(i)] If $R_{\rhobar}'^{\eta,\tau}\neq0$ there exists a \emph{unique} $\Mbar\in Y^{\eta,\tau}(\F)$ up to isomorphism such that $T_{\dd}^*(\Mbar)\cong \rhobar|_{\GKinfty}$.

\item[(ii)] Assume that $R_{\rhobar}^{\eta,\tau,\psi}\neq 0$. Fix a choice of $\Mbar$ as in (i) and an isomorphism $\bar{\delta}:T_{\dd}^*(\Mbar)\rightarrow \rhobar|_{\GKinfty}$. There exists a unique map $\bar{\alpha}:\Mbar\rightarrow \Mbar\dual$ making $(\Mbar,\bar{\alpha})$ an object of $Y^{\eta,\tau,\psi}(\F)$ such that $T_{\dd}^*(\Mbar,\bar{\alpha})$ identifies with the symplectic $\GKinfty$-representation $\rhobar|_{\GKinfty}$ under $\bar{\delta}$. 
\end{itemize}
\end{proposition}

\begin{proof}
(i) The existence of $\Mbar$ follows from Corollary 5.18 in \cite{CL}. The uniqueness follows from a straightforward generalization of Theorem 3.2 of \cite{LLHLM1} (triviality of the Kisin variety).

(ii) Let $\overline{\cM}$ denote the \'etale $\varphi$-module associated with $\Mbar$. By full faithfulness of $V_{\dd}^*$, the given isomorphism $T_{\dd}^*(\Mbar)\rightarrow T_{\dd}^*(\Mbar)\dual\otimes \psi \xrightarrow{\can^{-1}} T_{\dd}^*(\Mbar\dual)$
 induces an alternating isomorphism $\alpha:\overline{\cM}\rightarrow \overline{\cM}\dual$. Both $\alpha(\Mbar)$ and $\Mbar\dual$ are lattices in $\overline{\cM}\dual$ belonging to $Y^{[0,h],\tau}(\F)$. By triviality of the Kisin variety we deduce that $\alpha(\Mbar)=\Mbar\dual$. This proves existence. If there were two such maps $\alpha_1$ and $\alpha_2$ then by faithfulness of $V_{\dd}^*$, $\im(\alpha_1-\alpha_2)$ would be a submodule of $\Mbar\dual$ which vanishes upon tensoring with $\cO_{\cE,L}$, hence must be 0. 
\end{proof}

\begin{definition}
If $R_{\rhobar}'^{\eta,\tau}\neq 0$ we define the \emph{shape} of $(\rhobar,\tau)$ to be the shape of (the unique) $\Mbar$ in Proposition \ref{prop:mod p Galois and Kisin}(i).
\end{definition}

\begin{situation}\label{situation}
From now on specialize to the case $\eta=(3,2,1,0)$ (and $h=3$). Let $\rhobar$ and $\psi$ be as at the beginning of \S\ref{sec:Kisin modules} and let $\tau$ be a 4-generic tame principal series type such that $R_{\rhobar}^{\eta,\tau,\psi}\neq 0$. Fix data $(\Mbar,\bar{\alpha}),\bar{\delta}$ as in Proposition \ref{prop:mod p Galois and Kisin}(ii), as well as a gauge basis $\bar{\beta}$ of $(\Mbar,\bar{\alpha})$. Let $\tw$ denote the shape of $(\rhobar,\tau)$. We also assume that $\ad(\rhobar)$ is \emph{cyclotomic-free} in the sense of \cite[\S3.3]{LLHLM1}.
\end{situation}
Cyclotomic-freeness holds if $\rhobar$ is sufficiently generic and in particular holds in our applications.

 We now define some deformation problems. All data in the following is assumed to be compatible with $\rhobar,(\Mbar,\bar{\alpha}),\bar{\beta},\bar{\delta}$. For $A\in \CO:$
 
 \begin{itemize}
 \item $D_{\rhobar}^{\eta,\tau,\psi}$ is the deformation problem represented by $R_{\rhobar}^{\eta,\tau,\psi}$. 
 \item $D_{(\Mbar,\bar{\alpha}),\rhobar}^{\eta,\tau,\psi,\square}(A)$ is the set of tuples $(M,\alpha,\rho,\delta)$ where $(M,\alpha)\in Y^{\eta,\tau,\psi}(A)$, $\rho\in D_{\rhobar}^{\eta,\tau,\psi}(A)$, and $\delta:T_{\dd}^*(M,\alpha)\xrightarrow{\sim}\rho|_{\GKinfty}$ is a symplectic isomorphism.
 
 \item $D_{(\Mbar,\bar{\alpha}),\bar{\beta},\rhobar}^{\eta,\tau,\psi,\square}(A)$ is the set of tuples $(M,\alpha,\beta,\rho,\delta)$ where $(M,\alpha,\rho,\delta)$ is as above and $\beta$ is a gauge basis of $(M,\alpha)$. 
 
 \item $D_{(\Mbar,\bar{\alpha}),\bar{\beta}}^{\eta,\tau,\psi,\square}(A)$ is the set of tuples $(M,\alpha,\beta,\beta')$ as above, where $\beta'$ is a symplectic basis of $T_{\dd}^*(M,\alpha)$ making $(T_{\dd}^*(M,\alpha),\beta')$ a symplectic framed deformation of $\rhobar|_{\GKinfty}$. 
 
\item $D_{(\Mbar,\bar{\alpha}),\bar{\beta}}^{\eta,\tau,\psi}(A)$ is the set of tuples $(M,\alpha,\beta)$ as above.
\end{itemize}

\begin{proposition}\label{prop:representability}
The map $D_{(\Mbar,\bar{\alpha}),\rhobar}^{\eta,\tau,\psi,\square}\rightarrow D_{\rhobar}^{\eta,\tau,\psi}$ taking $(M,\alpha,\rho,\delta)\mapsto \rho$ is an isomorphism. In particular, $D_{(\Mbar,\bar{\alpha}),\rhobar}^{\eta,\tau,\psi,\square}$ is representable by $R_{\rhobar}^{\eta,\tau,\psi}$. 
\end{proposition}
\begin{proof}
By Corollary 3.6 of \cite{LLHLM1} (actually its natural generalization to $\GL_4$) says that since $\tau$ is 4-generic the map $D_{\Mbar,\rhobar}'^{\eta,\tau,\square}\rightarrow D_{\rhobar}'^{\eta,\tau}$ taking $(M,\rho,\delta)\mapsto \rho$ is an isomorphism. It follows that the deformation problem $D_{\Mbar,\rhobar}^{\eta,\tau,\psi,\square}$ of tuples $(M,\rho,\delta)$ as above but with $\rho\in D_{\rhobar}^{\eta,\tau,\psi}(A)$ is isomorphic to $D_{\rhobar}^{\eta,\tau,\psi}$ via the same map $(M,\rho,\delta)\mapsto \rho$. To conclude the proof, we need to show the following statement: given $(M,\rho,\delta)\in D_{\Mbar,\rhobar}^{\eta,\tau,\psi,\square}(A)$, there exists a unique map $\alpha:M\rightarrow M\dual$ making $(M,\alpha)$ an object of $Y^{\eta,\tau,\psi}(A)$ lifting $(\Mbar,\bar{\alpha})$ such that $\delta$ induces a symplectic isomorphism $T_{\dd}^*(M,\alpha)\cong \rho|_{\GKinfty}$. This now follows from an argument similar to the proof of Proposition \ref{prop:mod p Galois and Kisin}.
\end{proof}

We next give an explicit description of the universal symplectic lift of $(\Mbar,\bar{\alpha},\bar{\beta})$ of type $(\eta,\tau)$.

\begin{definition}\label{def:universal A}
Let $\bar{A}$ denote a matrix belonging to $\tw U_{\tw}(\F)$. The deformation problem which assigns to $A\in\CO$ the set of matrices in $\Mat_4(A[v])$ lifting $\bar{A}$ modulo $\m_A$ and satisfying the degree conditions in Definition \ref{def:gauge basis} is representable by a power series ring $R_{\tw,\bar{A}}$. Let $A^\univ\in\Mat_4(R_{\tw,\bar{A}}[v])$ denote the universal lift.

 Let $I^{\leq\eta}\leq R_{\tw,\bar{A}}$ denote the collection of polynomial equations resulting from the following Hodge type $\leq \eta$ conditions on $A^\univ$:
\begin{itemize}
\item all $2\times 2$ minors of $A^\univ$ are divisible by $P(v)$, and 
\item all $3\times 3$ minors are divisible by $P(v)^3$, and 
\item $\det(A^\univ)$ is a unit times $P(v)^6$.
\end{itemize}
Let $I^{\psi}\leq R_{\tw,\bar{A}}$ denote the collection of polyomial equations resulting from imposing the following symplecticity condition on $A^\univ$:
\begin{itemize}
\item $(A^\univ)^tJA^\univ=\xi P(v)^3J$.
\end{itemize}

\end{definition}

\begin{lemma}\label{lem:representability}
 $D_{(\Mbar,\bar{\alpha}),\bar{\beta}}^{\eta,\tau,\psi}$ is representable by an object $R_{(\Mbar,\bar{\alpha}),\bar{\beta}}^{\eta,\tau,\psi}$ of $\COhat$ isomorphic to $((R_{\tw,A_{\bar{\beta}}}/I^{\leq \eta})^{p\textrm{-flat,red}})/I^\psi$.
\end{lemma}

\begin{proof}
Let $D_{\Mbar,\bar{\beta}}'^{\eta,\tau}$ denote the deformation problem of pairs $(M,\beta)$. Proposition 3.4.7 of \cite{LLL} shows that it is representable by 
\begin{equation*}
R_{\Mbar,\bar{\beta}}'^{\eta,\tau}:=(R_{\tw_j,A_{\bar{\beta}}}/I^{\leq\eta})^{p\textrm{-flat,red}}.
\end{equation*}
The condition for $(M,\beta)\in D_{\Mbar,\bar{\beta}}'^{\eta,\tau}(A)$ to belong to $D_{(\Mbar,\bar{\alpha}),\bar{\beta}}^{\eta,\tau,\psi}(A)$ is for there to exist an isomorphism $\alpha:M\rightarrow M\dual$ such that $\alpha(\beta)=\beta\dual J$. Such an isomorphism is clearly unique and by Lemma \ref{lem:A dual} exists iff $A_\beta^tJA_\beta=\xi P(v)^3J$. The claim follows.
\end{proof}

\begin{remark}\label{rk:formally smooth}
 It follows that the remaining deformation problems defined above are representable. To see this note that the natural map $D_{(\Mbar,\bar{\alpha}),\bar{\beta},\rhobar}^{\eta,\tau,\psi,\square}\rightarrow D_{(\Mbar,\bar{\alpha}),\rhobar}^{\eta,\tau,\psi,\square}$ which forgets the gauge basis is formally smooth of relative dimension 2 by Lemma \ref{lem:symplectic gauge bases}(ii), and the natural map $D_{(\Mbar,\bar{\alpha}),\bar{\beta}}^{\eta,\tau,\psi,\square}\rightarrow D_{(\Mbar,\bar{\alpha}),\bar{\beta}}^{\eta,\tau,\psi}$ which forgets the symplectic basis on $T_{\dd}^*(M,\alpha)$ is formally smooth of relative dimension 10 by elementary considerations. In each case, we denote the representing ring by $R$ with the appropriate decorations. 
\end{remark}

\emph{Monodromy.}   We recall the monodromy condition on $R_{\Mbar,\bar{\beta}}'^{\eta,\tau}$ defined in \cite[\S3.4]{LLL} (cf. \cite[\S5]{LLHLM1} for background), which cuts out the locus of $\Spec(R_{\Mbar,\bar{\beta}}'^{\eta,\tau}[1/p])$ whose closed points correspond to Kisin modules that arise via restriction to $G_{K_\infty}$ of $G_{\Qp}$-representations having Hodge-Tate weights $\leq \eta$ (which are necessarily potentially crystalline of inertial type $\tau$). Let $(M^\univ,\beta^\univ)$ be the universal object over $R_{\Mbar,\bar{\beta}}'^{\eta,\tau}$ and write $A:=A_{\beta^\univ}$. Let
\begin{equation*}
P_N(A):= \left(-ev \frac{d}{dv}A-\left[\Diag(b_0,b_1,b_2,b_3),A \right] \right)P(v)^3 A^{-1} \in \Mat_4(R'^{\eta,\tau}_{\Mbar,\bar{\beta}}[v]).
\end{equation*}
The monodromy condition can be expressed as 
\begin{equation}\label{eq:monodromy condition}
-P_N(A)+M\Big|_{v=-p}=  \frac{d}{dv}(-P_N(A)+M)\Big|_{v=-p}=0
\end{equation}
where $M\in \Mat_4(\cO^{\mathrm{rig}}_{R_{\Mbar,\bar{\beta}}'^{\eta,\tau}})$ is an ``error term'' such that 
\begin{equation*}
M\Big|_{v=-p}, \frac{d}{dv}M\Big|_{v=-p}\in p^{4}\Mat_4 (R_{\Mbar,\bar{\beta}}'^{\eta,\tau}).
\end{equation*}
For this, see the proof of Proposition 3.4.12 in \cite{LLL}. Let $I^{\mathrm{mon}}$ denote the ideal of $R_{\Mbar,\bar{\beta}}'^{\eta,\tau}$ corresponding to the 32 equations in \eqref{eq:monodromy condition}, and let $R_{\Mbar,\bar{\beta}}'^{\eta,\tau,\nabla}$ be the corresponding $p$-flat and reduced quotient. Moreover, define
\begin{equation*}
R_{\Mbar,\bar{\beta}}'^{\eta,\tau,\square,\nabla}:= R_{\Mbar,\bar{\beta}}'^{\eta,\tau,\square}\otimes_{R_{\Mbar,\bar{\beta}}'^{\eta,\tau}}R_{\Mbar,\bar{\beta}}'^{\eta,\tau,\nabla}.
\end{equation*}

This allows us to define two more symplectic deformation problems which take into account the mondromy condition.
\begin{itemize}
\item We define $R_{(\Mbar,\bar{\alpha}),\bar{\beta}}^{\eta,\tau,\psi,\nabla}= R_{(\Mbar,\bar{\alpha}),\bar{\beta}}^{\eta,\tau,\psi}\otimes_{R_{\Mbar,\bar{\beta}}'^{\eta,\tau}} R_{\Mbar,\bar{\beta}}'^{\eta,\tau,\nabla}$.

\item Similarly,  $R_{(\Mbar,\bar{\alpha}),\bar{\beta}}^{\eta,\tau,\psi,\square,\nabla}= R_{(\Mbar,\bar{\alpha}),\bar{\beta}}^{\eta,\tau,\psi,\square}\otimes_{R_{\Mbar,\bar{\beta}}'^{\eta,\tau}} R_{\Mbar,\bar{\beta}}'^{\eta,\tau,\square,\nabla}$.
\end{itemize}

The following diagram relates all the symplectic deformation problems defined above. It is  analogous to \cite[Diagram (5.9)]{LLHLM1}. Hooked arrows denote closed immersions, and f.s. stands for a formally smooth map. The square is cartesian by definition.
\begin{equation}\label{eq:deformation diagram}
\xymatrix{ &  & \Spf R_{(\Mbar,\bar{\alpha}),\bar{\beta}}^{\eta,\tau,\psi,\square,\nabla} \ar@{^(->}[d] \ar[r]^{\textrm{f.s.}}& \Spf R_{(\Mbar,\bar{\alpha}),\bar{\beta}}^{\eta,\tau,\psi,\nabla} \ar@{^(->}[d] \\  & \Spf R_{(\Mbar,\bar{\alpha}),\bar{\beta},\rhobar}^{\eta,\tau,\psi,\square}\ar[d]^{\textrm{f.s. }} \ar@{^(->}[r]_\xi & \Spf R_{(\Mbar,\bar{\alpha}),\bar{\beta}}^{\eta,\tau,\psi,\square} \ar[r]_{\textrm{f.s.}} & \Spf R_{(\Mbar,\bar{\alpha}),\bar{\beta}}^{\eta,\tau,\psi} \\ \Spf R_{\rhobar}^{\eta,\tau,\psi} & \Spf R_{(\Mbar,\bar{\alpha}),\rhobar}^{\eta,\tau,\psi,\square} \ar[l]_\sim^{\ref{prop:representability}} &  & }
\end{equation}

The only map which has yet to be discussed is $\xi$. It is defined to take $(M,\alpha,\beta,\rho,\delta)$ to $(M,\alpha,\beta,\beta')$ where $\beta'$ is the symplectic basis of $T_{\dd}^*(M,\alpha)$ induced by the isomorphism $\delta$. The assumption that $\ad(\rhobar)$ is cyclotomic-free impies that $\xi$ is a closed immersion, by appealing to Proposition 3.12 of \cite{LLHLM1}.

\begin{proposition}\label{prop:xi isomorphism}
In Situation \ref{situation}, the map $\xi$ above induces a closed immersion $\Spf(R_{(\Mbar,\bar{\alpha}),\bar{\beta},\rhobar}^{\eta,\tau,\psi,\square})\hookrightarrow \Spf((R_{(\Mbar,\bar{\alpha}),\bar{\beta}}^{\eta,\tau,\psi,\square,\nabla})^{p\textrm{-flat,red}})$, which is an isomorphism iff every $\Qpbar$-point of the latter corresponds to a potentially crystalline lift of Hodge-Tate weights exactly $\eta$. 
\end{proposition} 
\begin{proof}
Since $R_{\rhobar}^{\eta,\tau,\psi}$ is $p$-flat and reduced (by Lemma \ref{lem:dimension of p crys def ring}), it follows that $\xi$ factors through a map $\Spf(R_{(\Mbar,\bar{\alpha}),\bar{\beta},\rhobar}^{\eta,\tau,\psi,\square})\hookrightarrow \Spf((R_{(\Mbar,\bar{\alpha}),\bar{\beta}}^{\eta,\tau,\psi,\square, \nabla})^{p\textrm{-flat,red}})$. Since both these rings are $p$-flat and reduced, it suffices to show factorization at the level of $\Qpbar$-points. But this follows from the argument in the proof of Theorem 5.12 of \cite{LLHLM1}. 

To see that the induced map is an isomorphism, by the same argument it suffices to show that for  every finite extension $\cO'/\cO$ and $\cO'$-point $(M,\alpha,\beta,\beta')$ of $\Spf((R_{(\Mbar,\bar{\alpha}),\bar{\beta}}^{\eta,\tau,\psi,\square,\nabla})$, $\rho_1:=T_{\dd}^*(M,\alpha)[1/p]:G_{K_\infty}\rightarrow \GSp_4(E')$ extends to a symplectic potentially crystalline representation of $G_{\Qp}$ of type $(\eta,\tau)$. By the condition on Hodge--Tate weights, it only remains to show the symplecticity. This follows rom the isomorphism between $\rho_1$ and $\rho_1\dual \otimes \psi$ and the uniqueness of the extension (\cite[Corollary 3.6]{LLHLM1}). 
\end{proof}

\begin{definition}
We say that $(\rhobar,\tau)$ is \emph{good} if the condition on the Hodge-Tate weights in the previous proposition holds.
\end{definition}
In practice, goodness can be read off the shape of $(\rhobar,\tau)$ after applying the Hodge type $\leq \eta$ conditions to the universal Kisin module. In this situation, Proposition \ref{prop:xi isomorphism} and \eqref{eq:deformation diagram} furnish our desired description of $R_{\rhobar}^{\eta,\tau,\psi}$.

\begin{corollary}\label{cor:explicit deformation ring}
In Situation \ref{situation}, if $(\rhobar,\tau)$ is good then there is an isomorphism
\begin{equation*}
R_{\rhobar}^{\eta,\tau,\psi}[[x_1,x_2]]\cong (R_{\tw,A_{\bar{\beta}}}/(I^{\leq\eta}+I^\psi+I^{ \textrm{mon}}))^{p\textrm{-flat,red}}[[y_1,\ldots,y_{10}]].\qed
\end{equation*}
\end{corollary}

 For the remainder of the section we discuss universal Frobenius eigenvalues. By definition, a closed point $x\in\Spec(R_{\rhobar}^{\eta,\tau,\psi})$ with residue field $E_x$ gives rise to a representation $\rho_x:G_{\Qp}\rightarrow \GSp_4(E_x)$ such that $\WD(\rho_x)|_{I_{\Qp}}\cong \tau$. Write $r_x:=\WD(\rho_x)$. For $0\leq i\leq 3$, let $t_{x,i}\in E_x$ be the eigenvalue of $r_x(\Frob)$ on the $I_{\Qp}=\omega^{b_i}$ eigenspace. By symplecticity we have $t_{x,0}t_{x,3}=t_{x,1}t_{x,2}=\WD(\psi)(\Frob)=\xi p^3$. 
\begin{definition}
For $0\leq i\leq 3$ we define the \emph{universal Frobenius eigenvalue} $\Lambda_i\in R_{(\Mbar,\bar{\alpha}),\bar{\beta}}^{\eta,\tau,\psi}$ to be the image of $(A^\univ)_{ii}\mod v\in R_{\tw,A_{\bar{\beta}}}$ under the quotient of Lemma \ref{lem:representability}. 
\end{definition}

The next proposition explains this terminology. 
\begin{proposition}\label{prop:universal frobenius eigenvalues}
For $0\leq i\leq 3$, let $\theta_i\in R_{(\Mbar,\bar{\alpha}),\bar{\beta},\rhobar}^{\eta,\tau,\psi,\square}$ denote the image of $\Lambda_i$ under the composite morphism in \eqref{eq:deformation diagram}. Then
\begin{itemize}
\item[(i)] $\theta_i$ lies in the subalgebra $R_{(\Mbar,\bar{\alpha}),\rhobar}^{\eta,\tau,\psi,\square}$. 
\item[(ii)] For any closed point $x\in \Spec R_{\rhobar}^{\eta,\tau,\psi}[1/p]$, the image of $\theta_i$ in $E_x$ is equal to $t_{x,i}$. 
\end{itemize}
\end{proposition}

\begin{proof}
The map $R_{(\Mbar,\bar{\alpha}),\rhobar}^{\eta,\tau,\psi,\square}\rightarrow R_{(\Mbar,\bar{\alpha}),\bar{\beta},\rhobar}^{\eta,\tau,\psi,\square}$ which forgets the gauge basis $\beta$ is formally smooth (Remark \ref{rk:formally smooth}). As $\theta_i$ is clearly independent of the choice of $\beta$, it lies in the subalgebra $R_{(\Mbar,\bar{\alpha}),\rhobar}^{\eta,\tau,\psi,\square}$. 

Observe that if $x\in \Spec(R_{(\Mbar,\bar{\alpha}),\bar{\beta},\rhobar}^{\eta,\tau,\psi,\square}[1/p])$ is a closed point with residue field $E_x$ then there is a canonical $\varphi$- and $\Delta$-equivariant isomorphism of $E_x$-modules (cf. \cite[\S2.5(1)]{Kis08})
\begin{equation*}
(M_x/uM_x)[1/p]\cong D_{\st}^*\left(T_{\dd}^*(M_x)[1/p]\big|_{G_{L}}\right).
\end{equation*}
It follows that $D_\st(\rho_x)$ has an eigenbasis $\mathfrak{e}=(e_0,e_1,e_2,e_3)$ such that $\Delta$ acts on $e_i$ by $\omega^{b_i}$ and $\varphi(e_i)=((A_{\beta_x})_{ii}\mod v)^{-1}e_i$. The second claim now follows from the definition of $\WD(\rho_x)$. 
\end{proof}

\begin{remark}\label{rk:units after 1/p}
Note that the symplecticity condition says $\theta_0\theta_3=\theta_1\theta_2=\xi p^3$. In particular, the $\theta_i$ are units in $R_{\rhobar}^{\eta,\tau,\psi}[1/p]$. 
\end{remark}

\subsection{Hecke algebras and deformation rings}\label{sec:LLC}
In this section $G$ denotes $\GSp_4(\Qp)$ and $K=\GSp_4(\Zp)$. 
Let $\rhobar:G_{\Qp}\rightarrow \GSp_4(\F)$ be a continuous representation, $\psi=\epsilon^3\tomega^{b-3}\nr_\xi:G_{\Qp}\rightarrow \cO\x$ a fixed similitude character, and $\tau:I_{\Qp}\rightarrow \GSp_4(\cO)$ a generic tame principal series inertial type. In this section we reformulate Proposition \ref{prop:universal frobenius eigenvalues} as the existence of a morphism of $\cO$-algebras $\Theta_{\rhobar,\tau}:\cH(\sigma(\tau))\rightarrow R_{\rhobar}^{\eta,\tau,\psi}[1/p]$ which interpolates the local Langlands correspondence for $G$ in the Bernstein block corresponding to $\tau$. Here $\cH(\sigma(\tau))$ is a Hecke algebra over $E$ which is defined below. This morphism will be used to express classical local-global compatibility in \S\ref{subsec:patched module axioms}. 

\begin{remark}
In \cite[\S3]{CEG+}, given \emph{any} inertial type $\tau$ for $\GL_n(F)$, the authors show the existence of a ``$K$-type'' $\sigma(\tau)$ for the Bernstein block corresponding to $\tau$ which detects irreducible smooth representations $\pi$ whose corresponding $L$-parameter has trivial monodromy operator. In \S4 they show the existence of a morphism (there denoted $\eta$) analogous to $\Theta_{\rhobar,\tau}$. The case we are concerned with here is more concrete, and in fact we make it completely explicit. 
\end{remark}

\begin{definition}
We say that $\mu=(x,y;z)\in X^*(T)$ is \emph{regular} if its $W$-orbit has size $|W|=8$. 
We say that it is \emph{irreducible} if none of $x,y,x\pm y$ is equal to $\pm1$.
\end{definition}

In this section let $\mu=(x,y;z)\in X^*(T)$ be a regular and irreducible weight. Let $\tau=\tau(\id,\bar{\mu}):I_{\Qp}\rightarrow \GSp_4(\cO)$; this is a lowest alcove presentation. If $\chi=\chi_x\times\chi_y\rtimes\chi_z:T(\Qp)\rightarrow \Qpbar\x$ is a smooth character such that $\chi_i|_{\Zp\x}=\tomega^i\circ\Art_{\Qp}$ for $i\in\{x,y\}$ and $\chi_z|_{\Zp\x}=\tomega^{(z-x-y)/2}\circ\Art_{\Qp}$ the irreducibility of $\mu$ guarantees that the principal series representation $\pi=\Ind_{B(\Qp)}^{\GSp_4(\Qp)
}(\chi)$ is irreducible (see \cite[Proposition 2.4.6]{BCGP}). These $\pi$ form the set of irreducible representations belonging to a Bernstein block which we denote $\Omega_\tau$.

\begin{lemma}\label{lem:Lpacket}
For any principal series representation $\pi=\Ind_{B(\Qp)}^{\GSp_4(\Qp)}(\chi)$ as above,  
\begin{itemize}
\item[(i)] $\recGTp(\pi)=(\rho,0)$ where 
\begin{equation*}
\rho=\begin{pmatrix}\chi_x\chi_y\chi_z|\cdot|^{-3} &&&\\ & \chi_x\chi_z|\cdot|^{-2} &&\\ && \chi_y\chi_z|\cdot|^{-1} & \\ &&&\chi_z \end{pmatrix}\circ \Art_{\Qp}^{-1}.
\end{equation*} 
\item[(ii)] The inverse image of $(\rho,0)$ under $\recGTp$ is equal to $\{\pi\}$. 
\end{itemize}
\end{lemma}
\begin{proof}
See the references in \cite[Proposition 2.4.6]{BCGP}. 
\end{proof}

If $R$ is any $\cO$-algebra we write $\tilde{\mu}_R$ for the character 
\begin{equation*}
\Iw\onto B(\Fp)\onto T(\Fp) \xrightarrow{\mu}\Fp\x\rightarrow R\x
\end{equation*}
where the final map is the Teichm\"uller lift. We define 
\begin{equation*}
\sigma(\tau):=\Ind_{\Iw}^K (\tilde{\mu}_E).
\end{equation*}
This is an irreducible $E$-representation of $K$ because $\mu$ is regular. 
The next lemma expresses the fact that $\sigma(\tau)$ is a $K$-type for $\Omega_\tau$.  
\begin{lemma}\label{lem:K-type}
 If $\pi$ is an irreducible smooth $\Qpbar$-representation of $G$ then $\Hom_K(\sigma(\tau),\pi)\neq 0$ if and only if $\pi\in\Omega_\tau$. In this case $\Hom_K(\sigma(\tau),\pi)$ is 1-dimensional.
\end{lemma}
\begin{proof}
The first claim follows from Theorem 7.7 of \cite{Roche}. The second claim follows from Frobenius reciprocity and the Iwasawa decomposition $G(\Qp)=B(\Qp)K$.
\end{proof}

We define the associated Hecke algebra 
\begin{equation*}
\cH(\sigma(\tau)):=\End_{E[G]}(\ind_K^G(\sigma(\tau))).
\end{equation*}
We will describe the structure of $\cH(\sigma(\tau))$ in terms of the integral pro-$p$ Iwahori-Hecke algebra 
\begin{equation*}
\cH_1:=\End_{\cO[G]}(\ind_{\Iw_1}^G(1_\cO)).
\end{equation*}
For $\tw\in \tW^{\vee(1)}$ let $T_{\tw}=[\Iw_1 \tw\Iw_1]$.  The set $(T_{\tw})_{\tw\in\tW^{\vee(1)}}$ forms an $\cO$-basis of $\cH_1$ (the Iwahori-Matsumoto basis) satisfying the braid and quadratic relations (cf. \cite{Vig}). If $t\in T(\Zp)/T(\Zp)_1$ then we have $T_tT_{\tw}=T_{t\tw}$ and $T_{\tw}T_t=T_{\tw t}$, so for each $\theta:T(\Fp)\rightarrow \cO\x$ there is an idempotent $\epsilon_\theta=\frac{1}{|T(\Fp)|}\sum_{t\in T(\Zp)/T(\Zp)_1}\theta(t)^{-1}T_t\in \cH_1$ such that $
\cH_1^\theta:=\epsilon_\theta \cH_1\epsilon_\theta$ identifies with $\End_{\cO[G]}(\ind_{\Iw}^G(\theta))$.

Let $X_*(T)^+ \subset X_*(T)$ be the submonoid of dominant cocharacters.  There is a homomorphism
\begin{align*}
    E[X_*(T)] &\xrightarrow{} \cH_1[1/p]\\
    \lambda=xy^\mo & \mapsto \delta_B^{1/2}(x(p)) T_{x(p)}  (\delta_B^{1/2}(x(p)) T_{y(p)} )^\mo
\end{align*}
for $\lambda \in X_*(T)$, $x,y\in X_*(T)^+$.

Let $\pi$ be as in Lemma \ref{lem:Lpacket}. By \cite[Proposition 2.4.3, 2.4.4]{BCGP}, there is an isomorphism of $E[X_*(T)]$-modules
\begin{equation}\label{eqn:Iw1-invariant}
    \begin{aligned}
     \pi^{\Iw_1} & \simeq \bigoplus_{w\in W} E(w \cdot \delta_B^{1/2}\chi^\mo) \\ 
     \pi^{\Iw = w \tilde{\mu}} &\simeq E(w \cdot \delta_B^{1/2}\chi^\mo).
\end{aligned}
\end{equation}
For the second isomorphism, we can view $\pi^{\Iw = w\tilde{\mu}}$ as $E[X_*(T)]$-module by a $\cO$-monoid map $\cO[X_*(T)^+] \xrightarrow{} \cH_1^{w\tilde{\mu}}$ which sends $\lambda\in X_*(T)^+$ to $\delta_B^{1/2}(\lambda(p))\epsilon_{w\tilde{\mu}}T_{\lambda(p)}$. When $w=\id$, this extends to  an isomorphism
\begin{equation}\label{eq:explicit Hecke algebra}
E[X_*(T)]\xrightarrow{\sim} \cH(\sigma(\tau))
\end{equation}
by \S5 of \cite{Roche}.

\begin{remark}
Following \cite{Vig04}, we use the \emph{right} action of $\cH_1$ on $\pi^{\Iw_1}$. In particular, this is different from the Hecke algebra action in \cite{BCGP}, which is defined as the \emph{left} action. If $f \in \cH_1$, viewed as a $\Iw_1$-biequivariant function on $G$, right action by $f$ on $\pi^{\Iw_1}$ is the same as the left action by $\tilde{f}$ defined as a function on $G$ sending $g \in G$ to $ f(g^\mo)$.  In particular, the right action by $T_\lambda$ on $\pi^{\Iw_1}$ is the same as the left action by $T_{-\lambda}$. This explains the difference between \eqref{eqn:Iw1-invariant} and \cite[Proposition 2.4.4]{BCGP}.
\end{remark}

Now fix a continuous representation $\rhobar:G_{\Qp}\rightarrow \GSp_4(\F)$ and a similitude character $\psi:G_{\Qp}\rightarrow \cO\x$. Assume $R_{\rhobar}^{\eta,\tau,\psi}\neq 0$.  If $x\in \Spec(R_{\rhobar}^{\eta,\tau,\psi}[1/p])$ is a closed point with residue field $E_x$, let $\rho_x:G_{\Qp}\rightarrow \GSp_4(E_x)$ denote the corresponding lift of $\rhobar$. Since $\WD(\rho_x)|_{I_{\Qp}}\cong \tau$ is regular, by Lemma \ref{lem:Lpacket}(ii) the $L$-packet of $(\WD(\rho_x),0)$ under $\recGTp$ consists of a unique principal series representation $\pi_x\in\Omega_\tau$ defined over $E_x$. 

\begin{proposition}
Assume that $\tau$ is 4-generic. There exists a unique map of $E$-algebras $\Theta_{\rhobar,\tau}:\cH(\sigma(\tau))\rightarrow R_{\rhobar}^{\eta,\tau,\psi}[1/p]$ characterized by the property that for any closed point $x\in \Spec(R_{\rhobar}^{\eta,\tau,\psi}[1/p])$, the tautological action of $\cH(\sigma(\tau))$ on $\Hom_{E[K]}(\sigma(\tau),\pi_x)$ is given by the character
\begin{equation*}
\cH(\sigma(\tau))\xrightarrow{\Theta_{\rhobar,\tau}}R_{\rhobar}^{\eta,\tau,\psi}[1/p]\xrightarrow{\ev_x} E_x.
\end{equation*}
\end{proposition}
\begin{proof}
Given \eqref{eqn:Iw1-invariant} and \eqref{eq:explicit Hecke algebra}, this statement is just a reformulation of Proposition \ref{prop:universal frobenius eigenvalues}. By Lemma \ref{lem:Lpacket}, $\Theta_{\rhobar,\tau}$ must take $\beta_0\mapsto \xi^\mo$, $\beta_1\mapsto p^3 \theta_0^\mo$ and $\beta_2\mapsto p^{5}\theta_0^\mo\theta_1^\mo$. This gives a well defined map because of Remark \ref{rk:units after 1/p}. Uniqueness follows from the fact that $R_{\rhobar}^{\eta,\tau,\psi}[1/p]$ is a reduced and Jacobson ring.
\end{proof}

\begin{remark}
The statement of this proposition is modeled after Theorem 4.1 of \cite{CEG+}. We could prove more general results by using their method, but the situation at hand is rather concrete so we gave an explicit proof.
\end{remark}

If $w\in W$, note that $\sigma(\tau)^w:=\Ind_{\Iw}^K(w\tilde{\mu}_\cO)$ is isomorphic to a $K$-equivariant $\cO$-lattice in $\sigma(\tau)$. Hence $\cH(\sigma(\tau)^w):=\End_{\cO[G]}(\ind_K^G(\sigma(\tau)^w))\cong \cH_1^{w\tilde{\mu}}$ is isomorphic to an $\cO$-subalgebra of $\cH(\sigma(\tau))$. It follows from the definition of $\Theta_{\rhobar,\tau}$ that for each closed $x\in \Spec(R_{\rhobar}^{\eta,\tau,\psi}[1/p])$ the composite map 
\begin{equation}\label{eq:hecke algebra lattices}
\cH_1^{w\tilde{\mu}}\hookrightarrow \cH(\sigma(\tau))\xrightarrow{\Theta_{\rhobar,\tau}} R_{\rhobar}^{\eta,\tau,\psi}[1/p]
\end{equation}
when composed with $\ev_x$ gives the natural action of $\cH_1^{w\tilde{\mu}}$ on $\pi_x^{\Iw=w\tilde{\mu}}$.  We write the composition map \eqref{eq:hecke algebra lattices} by $\Theta^w$.

\begin{corollary}\label{cor:image of Tlambda}
Let $\lambda=(0,0,-1,-1)\in X_*(T)$. Under the map \eqref{eq:hecke algebra lattices}, the image of $T_{\lambda}\epsilon_{w\tilde{\mu}}$ is given by

\centerline{
\begin{tabular}{|c|c|} \hline
$w$ & image in $R_{\rhobar}^{\eta,\tau,\psi}[1/p]$ \\ \hline
$\id,s_0$ & $\theta_3$ \\ \hline
$s_1,s_0s_1$  & $\theta_2$ \\ \hline
$s_1s_0,s_0s_1s_0$  & $\theta_1$ \\ \hline
$s_1s_0s_1,(s_0s_1)^2$  & $\theta_0$\\ \hline
\end{tabular}.
}
\end{corollary}
\begin{proof}
This calculation is similar to the previous one.
\end{proof}

\subsection{Frobenius eigenvalues and extension classes}\label{sec:rhobar}
Let $\psi$ be as in \S\ref{sec:LLC}. We now specify a family of mod $p$ Galois representations. Let $\rhobar:G_{\Qp}\rightarrow \GSp_4(\F)$ be a continuous representation of the form
\begin{equation}\label{eq:rhobar}
\rhobar\sim \begin{pmatrix}
\omega^{a_3}\nr_{\xi_3} &*&*&*\\ &\omega^{a_2}\nr_{\xi_2} &*&*\\ &&\omega^{a_1} \nr_{\xi_1}&*\\ &&&\omega^{a_0}\nr_{\xi_0}
\end{pmatrix}
\end{equation}
for some $a_i\in\Z$ and $\xi_i\in \F\x$. Note that $a_0+a_3=a_1+a_2=b$ and $\xi_0\xi_3=\xi_1\xi_2=\xi$. We say that $\rhobar$ is \emph{maximally nonsplit} if the off-diagonal extension classes are nonzero. For an integer $\delta\geq 0$ we say that $\rhobar$ is \emph{inertially $\delta$-generic} if  
\begin{equation*}
\delta < a_i-a_j < p-\delta \quad \forall i>j.
\end{equation*}
This implies that $p>2\delta$. From now on, assume that $\rhobar$ is maximally nonsplit and inertially 2-generic. As in \S2.1 of \cite{HLM}, this implies that $\rhobar$ is Fontaine-Laffaille with filtration jumps at $a_0,a_1,a_2,a_3$. Moreover, if $M_0$ is a Fontaine-Laffaille module such that $T_{\cris}^*(M_0)\cong \rhobar$, there exists a basis $\und{e}$ unique up to scalar multiple such that 
\begin{equation}\label{eq:FL module matrix}
\Mat_{\und{e}}(\varphi_{M_0})=\begin{pmatrix}\xi_0 & 1 &x_{02}&x_{03}\\ &\xi_1&1&x_{13} \\ &&\xi_2&1\\ &&&\xi_3\end{pmatrix}
\end{equation}
for some $x_{ij}\in\F$.

Let $M_0^\vee$ be the ($b$-twisted) dual Fontaine--Laffaille module of $M_0$ (cf.~\cite[Definition 4.5]{Boo}). It is defined in terms of the following data:
\begin{itemize}
    \item $M_0^\vee = \Hom_\F(M_0, \F)$,
    \item $\Fil^i M_0^\vee = \{f\in M_0^\vee \mid f|_{\Fil^{b-i+1} M_0} = 0 \}$,
    \item for $f\in \Fil^i M_0^\vee$ and $x\in \Fil^j M_0$, $\varphi_{M_0^\vee,i}(f)$ is the unique element satisfying
    \[\varphi_{M_0^\vee,i}(f)(\varphi_{M_0,j}(x)) = \left\{\begin{array}{cc}
          \varphi(f(x)) & \text{if $i+j = b$} \\ 
          0 & \text{if $i+j \neq b$}
    \end{array}\right..
    \]
\end{itemize}
(Here, $\varphi$ denotes the absolute Frobenius.)

\begin{lemma}
We have $T^{*}_\cris (M_0^\vee) \simeq \rhobar^\vee\otimes \omega^b$
\end{lemma}
\begin{proof}
This can be proven using the analogous statement for Breuil modules.   Let $\ov{S}= \F[u]/u^p$ and $\cM_0 = \cF_b(M_0) := M_0 \otimes_\F \ov{S}$ be a Breuil module associated to $M_0$ (\cite[Appendix A]{HLM}). Recall the dual Breuil module $\cM_0^*$ from \cite[Definition 3.2.8]{EGH} (we take $r$ in \emph{loc.\,cit.}\,as $b$). 
Since $T^*_\cris(M_0^\vee)\simeq T^*_\st (\cF_b(M_0^\vee))$ (see, the proof of Proposition 2.2.1 in \cite{HLM}) and $T^*_\st (\cM_0^*) \simeq T^*_\st (\cM_0)^\vee (b)$ (\cite[Definition 3.2.8]{EGH}),  it suffices to show that $\cF_b(M_0^\vee)\simeq \cF_b(M_0)^*$. 

One can identify underlying modules 
\begin{align*}
    \cF_b(M_0^\vee) = \Hom_{\F}(M_0, \F)\otimes_\F \overline{S} = \Hom_{\overline{S}}(M_0\otimes_\F \overline{S}, \overline{S}) = \cF_b(M_0)^*. 
\end{align*}
Using $\Fil^i \overline{S} = u^i \overline{S}$ for $0\le i \le b$, we have
\begin{align*}
    \Fil^b \cF_b (M_0^\vee) &= \Fil^b M_0^\vee \otimes_\F \overline{S} + \Fil^{b-1} M_0^\vee \otimes u\overline{S} + \cdots + M_0^\vee \otimes_\F u^b \overline{S} \\
    \Fil^b \cF_b(M_0)^* &= \{f\in \cF_b(M_0)^* \mid \forall 0\le i \le b, \ f(\Fil^i M_0\otimes u^{b-i}\overline{S}) \subset u^b \overline{S}\}.
\end{align*}
Noting that  $\Fil^i M_0^\vee \otimes_\F u^{b-i}\overline{S}=\{u^{b-i}f\mid f \in \cF_b(M_0)^*, f\vert_{\Fil^{b-i+1}M_0\otimes_\F \overline{S}}=0\}$,  one can easily see that  $\Fil^b \cF_b (M_0^\vee)\subset \Fil^b \cF_b(M_0)^*$. Conversely, note that if $f\in \Fil^b \cF_b(M_0)^*$ satisfies
\begin{align*}
    f|_{\Fil^{b-i+1} M_0}= 0, f|_{\Fil^{b-i} M_0}\neq 0,
\end{align*}
there exists $f'\in \Fil^i M_0^\vee \otimes_\F u^{b-i}\ov{S}$ such that $f-f'|_{\Fil^{b-i} M_0} =0$. Using this argument inductively, we can show that $\Fil^b \cF_b(M_0)^*\subset\Fil^b \cF_b (M_0^\vee) $.

Taking $su^{b-i}\otimes f \in \Fil^{b-i}\overline{S} \otimes \Fil^i M_0^\vee$ and $s'u^{b-j}\otimes m \in \Fil^{b-j}\overline{S}\otimes \Fil^j M_0$, one can explicitly check that Frobenius $\varphi_b$ of both $\cF_b(M_0^\vee)$ and $\cF_b(M_0)^*$ satisfies
\begin{equation*}
    \varphi_b(su^{b-i}\otimes f) (\varphi_b(s'u^{b-j}\otimes m)) = \left\{\begin{array}{cc}
        \varphi(ss'f(m)) & \text{if $i+j = b$} \\ 
         0 & \text{if $i+j \neq b$} 
    \end{array}\right.,
\end{equation*}
and $\varphi_b$ is uniquely characterized by this property. (Here, $c$ is defined in \cite[\S 3.2]{EGH})
\end{proof}

By the symplecticity of $\rhobar$, we have $M_0 \simeq M_0^\vee \otimes \F(\xi) $ where $\F(\xi)$ is a rank 1 Fontaine--Laffaille module with $\varphi=\xi$ and filtration jump at 0. Let $\und{e}^\vee$ be the dual basis of $\und{e}$. A simple computation tells us that
\begin{align*}
    \xi \Mat_{\und{e}^\vee w_0}(\varphi_{M_0^\vee}) & = \xi w_0\Mat_{\und{e}}(\varphi_{M_0})^{-t} w_0  \\ & =   \begin{pmatrix}\xi_0 & -\xi_1 \xi_3^{-1} &\xi_3^{-1}(1 - \xi_2 x_{13}) &-x_{03} + \xi_1^{-1} x_{13} + \xi_2^{-1} x_{02} - \xi^\mo \\ &\xi_1&-1&\xi_0^{-1}(1 - \xi_1 x_{02}) \\ &&\xi_2&-\xi_2\xi_0^{-1}\\ &&&\xi_3\end{pmatrix}.
\end{align*}
Conjugating this by $\Diag(-1, \xi_1\xi_3^\mo, -\xi_1\xi_3^\mo, 1)$, we can show that $M_0^\vee \otimes \F(\xi)$ has a basis with associated Frobenius matrix given by
\begin{equation*}
    \begin{pmatrix}\xi_0 &1  & \xi_1^\mo(1-\xi_2 x_{13}) &x_{03} - \xi_1^\mo x_{13} - \xi_2^\mo x_{02} + \xi^\mo \\ &\xi_1& 1 &\xi_2^\mo( 1- \xi_1 x_{02}) \\ &&\xi_2& 1 \\ &&&\xi_3\end{pmatrix}.
\end{equation*}
We can conclude that the symplecticity of $\rhobar$ is equivalent to
\begin{equation}\label{eq:symplecticity}
\xi_1 x_{02}+\xi_2x_{13}=1.
\end{equation}

\begin{remark}\label{rk:recovering rhobar}
The isomorphism class of $\rhobar$ is therefore uniquely determined by its semisimplification as well as the parameters $x_{02},x_{03}$. 
\end{remark}

\begin{definition}\label{def:tau0}
Associated with $\rhobar$ as above, we define the symplectic inertial type $\tau_0=\tau(\id,\overline{\mu_0})$ where $\mu_0=(a_2-a_0,a_1-a_0+1;a_0+a_3-3)$. This is a lowest alcove presentation.
\end{definition}

We now introduce the genericity assumptions on the extension classes of $\rhobar$ that we will use.
\begin{definition}\label{def:generic}
\begin{itemize}
\item[(i)] A continuous representation $\rhobar$ as in \eqref{eq:rhobar} is \emph{weakly $\delta$-generic (of weight $(a_3,a_2,a_1,a_0)$)} if it is maximally nonsplit, inertially $\delta$-generic, and in \eqref{eq:FL module matrix} we have
\begin{align*}
x_{03}&\neq 0 \\
\xi_1 x_{03}-x_{13}&\neq 0.
\end{align*}
\item[(ii)] A continuous representation $\rhobar$ as in \eqref{eq:rhobar} is called \emph{strongly $\delta$-generic (of weight $(a_3,a_2,a_1,a_0)$)} if it is weakly $\delta$-generic and furthermore 
\begin{equation*}
(a_3-a_0)\xi_2 x_{03} - (a_2-a_1)x_{02} \neq 0.
\end{equation*}
\end{itemize}
\end{definition}

\begin{proposition}\label{prop:Fontaine-Laffaille}
Assume that $\rhobar$ is weakly 4-generic and $R_{\rhobar}^{\eta,\tau_0,\psi}\neq 0$. Then 
\begin{equation*}
\tw(\rhobar,\tau_0)=\tw:=\begin{pmatrix}&&&v^2\\&v^2&&\\&&v&\\v&&& \end{pmatrix}
\end{equation*}
and if $(\Mbar,\bar{\alpha})\in Y^{\eta,\tau_0,\psi}(\F)$ and $\bar{\delta}$ are as in Proposition \ref{prop:mod p Galois and Kisin}(ii), then $(\Mbar,\bar{\alpha})$ has a gauge basis $\bar{\beta}$ such that
\begin{equation}\label{eq:A for rhobar}
A_{\bar{\beta}} = \begin{pmatrix} 0&0&0&-\frac{\xi}{x_{03}} v^2 \\ 0&\frac{\xi x_{03}}{\xi_1x_{03}-x_{13}} v^2&0&\frac{\xi_0}{\xi_1x_{03}-x_{13}}v^2 \\ 0&-\frac{\xi_1}{\xi_0}(1-\frac{x_{13}x_{02}}{x_{03}})v^2&\frac{\xi_1x_{03}-x_{13}}{x_{03}}v&-\frac{x_{13}}{x_{03}}v^2 \\ x_{03}v&-\xi_1x_{02}v^2&\xi_0 v&\xi_0 v^2\end{pmatrix}.
\end{equation}
\end{proposition}
\begin{proof}
Let $F$ denote the matrix in \eqref{eq:FL module matrix}. Let $\cF$ denote the functor from Fontaine-Laffaille modules to \'etale $\varphi$-modules constructed in \cite{HLM}, Appendix A.1. Let $\M=\overline{M}\otimes \F((u))$ denote the \'etale $\varphi$-module with descent data associated with $\overline{M}$. By Proposition 2.2.1 of \cite{HLM}, we must have
\begin{equation*}
\M^{\Delta=1}\cong \cF(M_0).
\end{equation*}
It now follows from Lemma 2.2.7 of \cite{HLM} that $\M^{\Delta=1}$ has a basis $\bar{\beta}''$ such that $\varphi_{\M}(\bar{\beta}'')=\bar{\beta}''\cdot\Diag(v^{a_0},v^{a_1},v^{a_2},v^{a_3})\cdot F$. Now observe that 
\begin{equation*}
\bar{\beta}':=\bar{\beta}''\cdot F^{-1}w_0 \begin{pmatrix}u^{-a_3+1}&&&\\ &u^{-a_2+2}&&\\&&u^{-a_1+1}&\\ &&&u^{-a_0+2}\end{pmatrix}
\end{equation*}
 is an eigenbasis of $\M$ such that 
\begin{equation*}
\varphi_{\M}(\bar{\beta}')=\bar{\beta}'\cdot\begin{pmatrix}\xi_3 u^{e}&&&\\ u^{e-a_3+a_2-1} & \xi_2 u^{2e} &&\\ x_{13}u^{e-a_3+a_1}& u^{2e-a_2+a_1+1}& \xi_1 u^e & \\ x_{03}u^{e-a_3+a_0-1} & x_{02}u^{2e-a_2+a_0} & u^{e-a_1+a_0-1}&\xi_0 u^{2e}\end{pmatrix}
\end{equation*}
It follows from triviality of the Kisin variety that $\bar{\beta}'$ is an eigenbasis of $M$. We compute
\begin{equation*}
A_{\bar{\beta}'}=\begin{pmatrix}\xi_3 v&&&\\ v&\xi_2v^2&&\\ x_{13}v&v^2&\xi_1 v& \\ x_{03}v& x_{02} v^2 & v & \xi_0 v^2 \end{pmatrix}.
\end{equation*}
Using \eqref{eq:symplecticity} and weak genericity, one computes that $A_{\bar{\beta}'}\in\cI(\F)\tw\cI(\F)$, which proves the claim about the shape of $(\rhobar,\tau_0)$. We compute that there exists $X\in\cI_1(\F)$ and $t\in T_4(\F)$ such that $t^{-1}XA_{\bar{\beta}'}t$ is equal to \eqref{eq:A for rhobar}. By \cite[Lemma 2.20]{LLHLM1} this implies that $M$ has an eigenbasis $\bar{\beta}$ obeying \eqref{eq:A for rhobar}. Since $A_{\bar{\beta}}$ obeys the symplecticity condition the proof is complete.
\end{proof}

We can now prove our main theorem on the Galois side. As in the previous section we write $\lambda=(0,0,-1,-1)\in X_*(T)$.

\begin{theorem}\label{thm:main result galois} 
Assume that $\rhobar$ is strongly 4-generic and $R_{\rhobar}^{\eta,\tau_0,\psi}\neq 0$.
Then $R_{\rhobar}^{\eta,\tau_0,\psi}$ is formally smooth over $\cO$, and for $w\in W$ the image of $T_{\lambda}\epsilon_{w\tilde{\mu_0}}$ under the morphism \eqref{eq:hecke algebra lattices} is equal to $p^{k_w}r_w$, where $k_w\in \Z$ and $r_w\in R_{\rhobar}^{\eta,\tau_0,\psi}$ is a unit whose reduction modulo the maximal ideal is given by the following table.
\newline

\centerline{
\begin{tabular}{|c|c|c|} \hline
$w$ & $k_w$ & $r_w\mod\m_{R_{\rhobar}^{\eta,\tau_0,\psi}} \in \F\x$ \\ \hline
$\id,s_0$ & 2 & $\zeta_1$ \\ \hline
$s_1,s_0s_1$ & 1 & $\zeta_2$\\ \hline
$s_1s_0,s_0s_1s_0$ & 2 & $\xi\zeta_2^{-1}$ \\ \hline
$s_1s_0s_1,(s_0s_1)^2$ & 1 & $\xi\zeta_1^{-1}$ \\ \hline
\end{tabular}
}
where 
\begin{align*}
\zeta_1:=& \frac{1}{\xi(a_3-a_0+2)}\left[(a_3-a_0)\xi_0-(a_2-a_1)\frac{\xi_0 x_{02}}{\xi_2x_{03}}\right] \\
\zeta_2:=& \xi_1-\frac{x_{13}}{x_{03}}.
\end{align*}
\end{theorem}

\begin{remark}\label{rk:scalars recover FL invariants}
Observe that given the diagonal characters of $\rhobar$, $\zeta_1$ and $\zeta_2$ uniquely determine $\rhobar$ by Remark \ref{rk:recovering rhobar}.
\end{remark}

\begin{proof}
Fix $(\Mbar,\bar{\alpha},\bar{\delta},\bar{\beta})$ as in Proposition \ref{prop:Fontaine-Laffaille}. Let $A^\univ$ denote the universal object of $R_{\tw,A_{\bar{\beta}}}$ as in Definition \ref{def:universal A}. We may write 
\begin{equation*}
A^\univ = \begin{pmatrix}c_{00} & c_{01}'+c_{01}P(v) & c_{02} & c_{03}''+c_{03}'P(v)+c_{03}^*P(v)^2 \\ 0 & c_{11}''+c_{11}'P(v)+c_{11}^*P(v)^2 & c_{12} & c_{13}''+c_{13}'P(v)+c_{13}P(v)^2 \\ 0 & c_{21}'v+c_{21}vP(v) & c_{22}'+c_{22}^*P(v) & c_{23}''+c_{23}'P(v)+c_{23}P(v)^2 \\ c_{30}^*v & c_{31}'v+c_{31}vP(v) & c_{32}v & c_{33}''+c_{33}'P(v)+c_{33}P(v)^2 \end{pmatrix}
\end{equation*}
with $c_{ij}^\bullet\in R_{\tw,A_{\bar{\beta}}}$, reducing modulo the maximal ideal to \eqref{eq:A for rhobar} (the superscript $*$ indicates a unit). We now apply the Hodge type $\leq \eta$ and symplecticity equations to this matrix and $p$-saturate the result. First, the $2\times 2$ minor condition and $p$-flatness immediately imply that 
\begin{equation*}
c_{11}''=c_{12}=c_{13}''=c_{21}'=c_{22}'=c_{23}''=0.
\end{equation*}
Then applying the $3\times 3$ minor condition on the $(0,3)$- and $(1,1)$-minors immediately gives
\begin{equation*}
c_{11}'=c_{13}'=0.
\end{equation*}
With $A^\univ$ simplified in this way, we may apply the remainder of the Hodge and symplecticity equations to it. After simplifying the resulting numerous equations (using $p$-flatness), there are 10 variables $c_{00},c_{13},c_{21},c_{22}^*,c_{30}^*,c_{31},c_{33}'',c_{33}',c_{33}$ remaining and $A^\univ$ becomes the matrix
\begin{equation*}
\begin{pmatrix}c_{00} & \frac{c_{00}}{c_{30}^*}(c_{31}'+c_{31}P(v)) & \frac{c_{00}c_{13}c_{22}^*}{\xi} & \frac{1}{c_{30}^*}\left[(c_{00}c_{33}'+pc_{00}c_{33}-p^2\xi)+(c_{00}c_{33}-p\xi)P(v)-\xi P(v)^2\right] \\
0 & \frac{\xi}{c_{22}^*}P(v)^2 & 0 & c_{13}P(v)^2 \\
0 & c_{21}vP(v) & c_{22}^*P(v) & c_{22}^*\left[-(\frac{c'_{31}}{c_{30}^*} + \frac{p c_{13}c_{21}c_{22}^*}{\xi})P(v) + (\frac{c_{13}c_{21}c_{22}^*}{\xi}- \frac{c_{31}}{c_{30}^*})P(v)^2\right] \\
c_{30}^*v & c_{31}'v+c_{31}vP(v) & \frac{c_{13}c_{30}^*c_{22}^*}{\xi}v & c_{33}''+c_{33}'P(v)+c_{33}P(v)^2\end{pmatrix}
\end{equation*}
satisfying the \emph{single} equation
\begin{equation}\label{eq:first equation}
c_{00}(c_{33}''+pc_{33}'+p^2c_{33})=\xi p^3.
\end{equation}
We now apply the monodromy equation \eqref{eq:monodromy condition} to this matrix. The result is 32 equations (some empty) each consisting of a polynomial part, which come from the vanishing of $P_N(A^\univ)$ and $(d/dv)P_N(A^\univ)$ at $v=-p$, together with an error term which is $O(p^{4})$ by genericity. 

A tedious calculation shows that after simplification (using $p$-flatness), the effect of applying the monodromy equations and \eqref{eq:first equation} is that $c_{31}'$, $c_{33}''$ and $c_{33}'$ can be solved for in terms of the other variables, leaving 7 variables satisfying the single equation
\begin{equation}\label{eq:second equation}
c_{00}\left[ \frac{e+a_1-a_2+1}{e+a_0-a_3-1} c_{13}c_{31}c_{22}^*+\frac{a_0-a_3 -1-e}{e+a_0-a_3-1}\xi c_{33}\right]=\xi^2 p+O(p^{2})
\end{equation}
The strong genericity condition ensures that the expression in square brackets is a unit. We can therefore use this equation to eliminate $c_{00}$, leaving 6 free variables. From the form of the universal Kisin module it is now clear that $(\rhobar,\tau)$ is good so by Corollary \ref{cor:explicit deformation ring} it follows that $R_{\rhobar}^{\eta,\tau_0,\psi}$ is formally smooth (of relative dimension 14) over $\cO$. 

For the second claim, note first that $\theta_1=(\xi/c_{22}^*)p^2$ and $\theta_2=c_{22}^*p$. Moreover, \eqref{eq:second equation} makes it clear that $\theta_0=c_{00}$ is equal to $p$ times a unit whose reduction modulo the maximal ideal is equal to
\begin{equation*}
\xi(a_3-a_0+2)\left[(a_3-a_0)\xi_0-(a_2-a_1)\frac{\xi_0 x_{02}}{\xi_2x_{03}}\right]^{-1}.
\end{equation*}
The claim now follows from Corollary \ref{cor:image of Tlambda}. 
\end{proof}

\subsection{Crystalline deformation rings}\label{sec:crystalline deformations}

In this section we show that certain symplectic crystalline deformation rings in the Fontaine-Laffaille range are formally smooth. While this result should hold in general, we make the assumption that the residual Galois representation is ordinary and generic in order to use an ad hoc argument.

Let $\lambda=(a_3,a_2,a_1,a_0)\in \Z^4$ and let $\rhobar:G_{\Qp}\rightarrow \GSp_4(\F)$ be as in \eqref{eq:rhobar}. Assume that $a_{i+1}-a_i\geq 2$ for $i=0,1,2$ and $a_3-a_0\leq p-2$. Write $\F^4= \overline{\Fil}^0 \supset \overline{\Fil}^1\supset \cdots\supset \overline{\Fil}^3\supset 0$ for the unique full flag preserved by $\rhobar$. Fix a character $\psi=\epsilon^b\nr_\xi :G_{\Qp}\rightarrow \cO\x$ lifting the similitude character of $\rhobar$. We define a deformation functor $\cD^{\Delta_\lambda}_{\rhobar}$ by sending $A\in \COhat$ to the set of continuous homomorphisms $\rho:G_{\Qp}\rightarrow \GSp_4(A)$ lifting $\rhobar$ of similitude character $\psi$ such that there exists a filtration $A^4=\Fil^0\supset \Fil^1\supset \cdots \supset \Fil^3\supset 0$ by $A$-direct summands which is preserved by $\rho$ such that $\Fil^i\otimes_A \F=\overline{\Fil}^i$ and $(\Fil^i/\Fil^{i+1})|_{I_{\Qp}}=A(\epsilon^{a_i})$ for $0\leq i\leq 3$. Note that if such a filtration exists it is unique by \cite[Lemma 2.4.6]{CHT}. It follows that $\cD^{\Delta_\lambda}_{\rhobar}$ is representable by an object $R_{\rhobar}^{\Delta_\lambda}\in \COhat$.  

\begin{lemma}
$R_{\rhobar}^{\Delta_\lambda}$ is formally smooth of dimension 15. 
\end{lemma}
\begin{proof}
This follows from arguments similar to those of Lemmas 2.4.7 and 2.4.8 of \cite{CHT}. 
\end{proof}

Similar to \S\ref{sec:Kisin modules}, we let $R^{\lambda,\psi}_{\rhobar}$ denote the crystalline symplectic deformation ring of Hodge type $\lambda$. By Lemma 3.1.4 of \cite{GG}, $R_{\rhobar}^{\Delta_\lambda}$ is a quotient of $R_{\rhobar}^{\lambda,\psi}$. 

\begin{proposition}\label{prop:fs crystalline deformation rings}
Under the conditions on $\lambda$ above, the map $R_{\rhobar}^{\lambda,\psi} \rightarrow R_{\rhobar}^{\Delta_\lambda}$ is an isomorphism. In particular $R_{\rhobar}^{\lambda,\psi}$ is formally smooth of dimension 15. 
\end{proposition}
\begin{proof}
We deduce this from the corresponding fact for $\GL_4$. Let ${R}_{\rhobar}^{\prime \lambda,\psi}$ and ${R}_{\rhobar}^{\prime \Delta_\lambda}$ denote the corresponding $\GL_4$-valued deformation rings. Then ${R}_{\rhobar}^{\prime \lambda,\psi}$ is known to be formally smooth of dimension 22 by \cite[Lemma 2.4.1]{CHT}. On the other hand $\mathrm{R}_{\rhobar}^{\Delta_\lambda}$ is also formally smooth of dimension 22 by Lemmas 2.4.7 and 2.4.8 of \cite{CHT}. Hence they are equal and the corresponding result for $\GSp_4$ follows.
\end{proof}

\section{Abstract mod $p$ local-global compatibility}

\subsection{The Jantzen filtration for finite reductive groups}\label{sec:jantzen filtration}
In this section we recall some results from \cite{Jan84}. Let $G$ denote a connected reductive group containing a Borel subgroup $B=TU$ with maximal torus $T$ all defined over $\Fp$. Let $R^+\subset R$ denote the corresponding system of positive roots. Assume that the derived subgroup of $G$ splits over $\Fp$. Let $W=N_G(T)/T\cong N_G(T)(\Fp)/T(\Fp)$ denote the Weyl group. For $w\in W$ set
\begin{equation*}
R_w^+:=\{\alpha\in R^+\,|\, w(\alpha)\in -R^+\}.
\end{equation*}
A character $\lambda\in X^*(T)$ defines a map $\tilde{\lambda}:T(\Fp)\rightarrow \Zp\x$ via the Teichm\"uller lift. If $A$ is any $\Zp$-algebra, we define the induced left $A[G(\Fp)]$-module 
\begin{equation*}
M(\lambda)_A:= A[G(\Fp)]/\mathfrak{I}_\lambda \cong \Ind_{B(\Fp)}^{G(\Fp)}(\tilde{\lambda}_A),
\end{equation*}
where  $\mathfrak{I}_{\lambda}$ is the left ideal of $A[G(\Fp)]$ generated by the elements $\{b-\widetilde{\lambda}_A(b)\cdot 1\,|\,b\in B(\Fp)\}$. Note that $M(\lambda)_A$ is a free $A$-module, and if $A\rightarrow B$ is a morphism of $\Zp$-algebras then $M(\lambda)_A\otimes B=M(\lambda)_B$. 

Let $\cH_{1,A}^{f}$ denote the Hecke algebra $\End(\Ind_{U(\Fp)}^{G(\Fp)}(1_A))$. Identifying $\cH_{1,A}^f$ with the algebra of $U(\Fp)$-biinvariant functions $G(\Fp)\rightarrow A$ under convolution, we let $T_{\dot{w}}\in \cH_{1,A}^f$ denote the characteristic function of $U(\Fp) \dot{w} U(\Fp)$ for any $w\in W$ with lift $\dot{w}\in N_G(T)(\Fp)$. Then $T_{\dot{w}}$ induces an intertwining homomorphism $M(\lambda)_A\rightarrow M(w\lambda)_A$ which we denote by the same letter. Explicitly for $S\in A[G(\Fp)]/\mathfrak{I}_\lambda$ we have
\begin{equation*}
T_{\dot{w}}(S)= S\cdot \sum_{g\in U(\Fp) \dot{w} U(\Fp)}  g^{-1}.
\end{equation*}

\begin{remark}
If $G$ is the base change of a split reductive group over $\Zp$, there is a natural inclusion of $A$-algebras $\cH_{1,A}^f\hookrightarrow \cH_{1,A}$ into the pro-$p$ Iwahori-Hecke algebra of $G(\Qp)$ which identifies it with the subalgebra generated by the elements $(T_{\dot{w}})_{w\in W}$ and $(T_t)_{t\in T(\Fp)}$. In particular, the element $T_{\dot{w}}\in \cH_{1,A}^f$ defined above and the Iwahori-Matsumoto basis element $T_{\dot{w}}\in \cH_{1,A}$ defined in \S\ref{sec:LLC} are identified under this inclusion. We use this without comment later on. 
\end{remark}

Given $\lambda\in X^*(T)$ and $w\in W$, Jantzen in \cite{Jan84} defines a filtration $M(\lambda)_{\Zp}=M(\lambda)_{\Zp}(w,0)\supseteq M(\lambda)_{\Zp}(w,1)\supseteq \cdots$ of subrepresentations by setting
\begin{equation*}
M(\lambda)_{\Zp}(w,i)=\{m\in M(\lambda)_{\Zp}\,|\, T_{\dot{w}}(m)\in p^i M(w\lambda)_{\Zp}\}
\end{equation*}
for $i\geq 0$ (it does not depend on the choice of $\dot{w}$). Let $M(\lambda)_{\Fp}(w,i)$ denote the image of $M(\lambda)_{\Zp}(w,i)$ in $M(\lambda)_{\Fp}$. We get an induced filtration
\begin{equation*}
M(\lambda)_{\Fp}=M(\lambda)_{\Fp}(w,0)\supseteq M(\lambda)_{\Fp}(w,1)\supseteq \cdots.
\end{equation*}

\begin{theorem}[Satz 4.4 of \cite{Jan84}, $q=p$ case]\label{thm:filtration length}t
Let $\lambda\in X(T)$ and $w\in W$. We have $M(\lambda)_{\Fp}(w,i)=0$ iff $i>\ell(w)$. \qed
\end{theorem}
Jantzen's main result concerning this filtration is an explicit ``sum formula'' for 
\begin{equation*}
\nu(\lambda,w):= \sum_{i>0} M(\lambda)_{\Fp}(w,i)
\end{equation*}
in terms of the reduction mod $p$ of Deligne-Lusztig characters. For $(\sigma,\mu)\in W\times X^*(T)$ let $R_{\sigma}(\mu)$ denote the (virtual) Deligne-Lusztig character over $\Qp$ corresponding to it in \cite{Jan81} and let $\overline{R}_{\sigma}(\mu)$ denote its reduction mod $p$. 
\begin{theorem}[Sum formula]\label{thm:sum formula}
Let $\lambda\in X^*(T)$. For all $\alpha\in R$ let $r_{\alpha},m_{\alpha}\in \N$ be such that 
\begin{equation*}
\pair{\lambda,\alpha\dual}=r_{\alpha}+m_{\alpha}(p-1),\quad\quad 0<r_{\alpha}\leq p-1.
\end{equation*}
Then for all $w\in W$, 
\begin{equation*}
\begin{split}
\nu(\lambda,w)= \sum_{\alpha\in R_w^+} \Big[ &\overline{R}_{s_{\alpha}}(\lambda+m_{\alpha}\alpha) \\ &-\frac{1}{2}\sum_{j=1}^{r_{\alpha}-1} \left(\overline{R}_1(\lambda-j\alpha)-\overline{R}_{s_{\alpha}}(\lambda-(j-m_{\alpha})\alpha)  \right)\Big].
 \end{split}
\end{equation*}
\end{theorem}
\begin{proof}
This is the special case $q=p$ of Satz 4.2 in \cite{Jan84}.
\end{proof}

\begin{remark}
Jantzen uses only $\Zp$-coefficients in his paper, but if we define
\begin{equation*}
M(\lambda)_{\cO}(w,i)=\{m\in M(\lambda)_{\cO}\,|\, T_{\dot{w}}(m)\in p^i M(w\lambda)_{\cO}\}
\end{equation*}
and let $M(\lambda)_\F(w,i)$ be the image of $M(\lambda)_\cO(w,i)$ in $M(\lambda)_\F$, then $M(\lambda)_{\cO}(w,i)=M(\lambda)_{\Zp}(w,i)\otimes_{\Zp}\cO$ and $M(\lambda)_\F(w,i)=M(\lambda)_{\Fp}(w,i)\otimes_{\Fp}\F$ and the sum formula continues to hold over $\F$.
\end{remark}

\subsection{Representation theory of $\GSp_4(\Fp)$}\label{sec:representation theory}
The goal of this section is to compute one of the Jantzen filtrations for the group $\GSp_4(\Fp)$, for generic principal series. The method is based on the one used by Jantzen in \cite[\S5]{Jan84} for  $\SL_3(\Fp)$, but the computations in our case are more tedious. Specialize the notation of the previous section to $G=\GSp_{4/\Fp}$ and let $B$ and $T$ the standard upper-triangular Borel and diagonal torus respectively. We use the notation of \S\ref{sec:notation} for roots and the Weyl group. Let $\rho':=(2,1;1)\in \frac{1}{2}\sum_{\alpha\in R^+}\alpha +X^0(T)_\Q$. We make use of the dot action of $W$ on $X^*(T)$ with respect to $\rho'$. 

We refer to \cite{RAGS} for background concerning the representation theory of $G$ like the notion of $p$-alcoves and linkage. Let $X_1(T)$ denote the set of $p$-restricted weights, and $X^0(T)$ the set of $W$-invariant weights.
\begin{definition}\label{def:alcoves}
We use the notations $C_\bullet$, $D_\bullet$, $E_\bullet$ for the various dominant $p$-alcoves defined in Table \ref{table:alcoves} and depicted in Figure \ref{fig:weyl chamber}. In particular $C_0,C_1,C_2,C_3$ comprise all the $p$-restricted dominant alcoves, with $C_0$ being the lowest alcove. 
\end{definition}

\begin{definition}
We say that $\lambda\in X^*(T)$ lies $\delta$-deep in its alcove if
\begin{equation*}
\pair{\lambda+\rho',\alpha\dual}\in (\delta,p-\delta)\mod p
\end{equation*} 
for each $\alpha\in R^+$. 
\end{definition}

\begin{definition}
A \emph{Serre weight} is an irreducible $\Fpbar$-representation of $G(\Fp)$. All Serre weights are defined over $\Fp$. 
\end{definition}

\begin{definition}
If $\lambda \in X^*(T)$ we write $\chi(\lambda)$ for the $G$-character defined in \cite[II.5.7]{RAGS} given by the Weyl character formula. A useful property is that 
\begin{equation}\label{eq:Weyl reflection chi}
\chi(w\cdot\lambda)=(-1)^{\ell(w)}\chi(\lambda)
\end{equation}
for all $w\in W, \lambda\in X^*(T)$. If $\lambda\in X(T)^+$, we write $\chi_p(\lambda)$ for the character of the unique irreducible $G$-representation $L(\lambda)$ of highest weight $\lambda$. 
If $\lambda\in X_1(T)$, we write $F(\lambda)$ for the restriction of $L(\lambda)$ to a $G(\Fp)$-representation. It is irreducible, and in fact the map $\lambda\mapsto F(\lambda)$ induces a bijection (\cite[Lemma 9.2.4]{GHS})
\begin{equation*}
X_1(T)/(p-1)X^0(T) \leftrightarrow \{\textrm{Serre weights}\}.
\end{equation*}
\end{definition}

We now wish to recall a sequence of well-known results in the representation theory of $G$ and $G(\Fp)$ that allow us to compute the right hand side of the sum formula as an explicit sum of Serre weights. The first is Jantzen's decomposition formula which expresses the mod $p$ reductions of Deligne-Lusztig characters in terms of Weyl characters $\chi(-)$: for any $\sigma\in W$ and $\mu\in X^*(T)$ we have
\begin{equation}\label{eq:Jantzen decomposition}
\overline{R}_\sigma(\mu)=\sum_{w_1,w_2\in W} \gamma_{w_1,w_2}'\chi(w_1(\mu-\sigma\epsilon'_{(s_0s_1)^2w_2})+p\rho'_{w_1}-\rho').
\end{equation}
See \cite[\S5.1]{Her09} for this result, and \cite[p14+15]{HT} for the $\gamma_{w_1,w_2}'$ in the case $G=\GSp_4$ at hand. Using this formula as well as \eqref{eq:Weyl reflection chi} we can write the right hand side of the sum formula as a sum of Weyl characters of dominant highest weight  lying inside the alcoves of Definition \ref{def:alcoves}.

Second is the decomposition of $\chi(\lambda)$ into irreducible $G$-representations when $\lambda$ lies in the alcoves of Definition \ref{def:alcoves}. These formulas are easy to deduce inductively using Jantzen's filtration of Weyl modules (\cite[II.8.19]{RAGS}) and its corresponding sum formula.
\begin{lemma}\label{lem:Weyl decomposition}
For $\lambda$ lying inside the stated alcoves, $\chi(\lambda)$ decomposes into irreducibles $\chi_p(-)$ as in Table \ref{table:Weyl decomposition}. \qed
\end{lemma}

The third and final ingredient is the method for decomposing $\chi_p(\lambda)$ as a $G(\Fp)$-representation when $\lambda\in X^*(T)^+$ lies sufficiently deep inside an alcove. If $\lambda\in X_1(T)$, the result is $F(\lambda)$ as explained above. In general, write $\lambda=\lambda_0+p\lambda_1+\cdots +p^n \lambda_n$ with $\lambda_i\in X_1(T)$. Then it follows from Steinberg's theorem (\cite[II.3.17]{RAGS}) that $\chi_p(\lambda)=\prod_{i=0}^n\chi_p(\lambda_i)$ as a $G(\Fp)$-character. This product can be simplified by the tensor product theorem (cf. \cite[Proposition 6.4]{Hum}) in certain cases. In fact, we will only need to consider the case when $n=1$, $\lambda_0$ is sufficiently deep inside a $p$-restricted alcove $C_\bullet$ , and $\lambda_1$ is a ``small weight''. Then the tensor product says
\begin{equation}\label{eq:tensor product theorem}
\chi_p(\lambda)= \sum_{\nu\in \chi_p(\lambda_1)} F(\lambda_0+\nu).
\end{equation}
The notation means that $\nu\in X^*(T)$ runs over all weights occuring in $\chi_p(\lambda_1)$. We illustrate this technique with an example.
\begin{example}
Suppose that $\lambda=(x,y;z)$ lies 1-deep in alcove $E_0$. Then Steinberg's theorem implies that $L(\lambda)\cong L(x-p,y;z+p)\otimes L(1,0;-1)$ as a $G(\Fp)$-representation. Since $L(1,0;-1)$ is 4-dimensional, having weights $\nu=(1,0;-1),(0,1;-1),(0,-1;-1), (-1,0;-1)$, and each weight $(x-p,y;z+p)+\nu$ lies in $C_0$, the tensor product theorem implies that $L(x-p,y;z+p)\otimes L(1,0;-1)\cong \bigoplus_\nu L((x-p,y;z+p)+\nu)$ as a  $G$-representation. We deduce that $\chi_p(\lambda)=\sum_\nu F((x-p,y;z+p)+\nu)$ as a $G(\Fp)$-representation, $\nu$ running over the four weights above.
\end{example}

\begin{lemma}\label{lem:multiplicity 1}
Assume that $\lambda=(x,y;z)$ is 7-deep inside a $p$-restricted alcove $C_\bullet$. Then the principal series representation $\overline{R}_1(\lambda)$ has 20 distinct Jordan-H\"older factors each occurring with multiplicity 1. These factors are listed in Table \ref{table:Jantzen filtration pieces}.
\end{lemma}
\begin{proof}
The list of 20 Jordan-H\"older factors is easily deduced from Jantzen's decomposition theorem \eqref{eq:Jantzen decomposition} and Lemma \ref{lem:Weyl decomposition}. The fact that they each occur with multiplicity 1 can be established by a dimension count. 
\end{proof}
\begin{proposition}\label{prop:graded pieces}
Assume that $\lambda=(x,y;z)$ is 7-deep inside a $p$-restricted alcove $C_\bullet$. Then the characters of the graded pieces of $M(\lambda)_{\F}(s_1s_0s_1,\bullet)$ are as in Table \ref{table:Jantzen filtration pieces}. 
\end{proposition}
\begin{proof}
Since the principal series representation is multiplicity 1 by Lemma \ref{lem:multiplicity 1}, the graded piece of the filtration to which a given Jordan-H\"older factor belongs is uniquely determined by its coefficient in $\nu(\lambda,s_1s_0s_1)$. This can be evaluated explicitly using the method described above via \eqref{eq:Jantzen decomposition}, Lemma \ref{lem:Weyl decomposition}, and \eqref{eq:tensor product theorem}, the latter holding by the assumption of deepness.
\end{proof}

One can obviously use this method to compute the graded pieces of $M(\lambda)(w,\bullet)$ for any $w\in W$, but we only will need to use the case $w=s_1s_0s_1$.

\subsection{Abstract mod $p$ local-global compatibility}\label{subsec:patched module axioms}

In this section, let $G$ and $K$ denote $\GSp_4(\Qp)$ and the maximal compact subgroup $\GSp_4(\Zp)$ respectively and $G^\cJ= \prod_{\cJ} G$, $K^\cJ= \prod_{\cJ} K$ for some finite set $\cJ$. For each $j\in \cJ$, fix continuous representation $\rhobar_j:G_{\Qp}\rightarrow\GSp_4(\F)$ of weight $(a_3,a_2,a_1,a_0)$ and similitude character $\psi_{j,\pcris}=\tomega^{b-3}\epsilon^3\nr_{\xi_j}$ as in $\S \ref{sec:rhobar}$ so that $R_{\rhobar_j}^{\eta,\tau,\psi_{j,\pcris}}\neq 0$. Also we need a crystalline character $\psi_{j,\cris}=\epsilon^{b}\nr_{\xi_j}$ as in  $\S \ref{sec:crystalline deformations}$. We can and do choose $\xi_j$ as the Teichm\"uller lift of $\nu_{\simc} (\rhobar_j)(\Frob)\in \F^\times$ (this is necessary for our global setup; see Lemma \ref{lem:lifting-Galois-characters}). Let $R_\infty$ (resp.  $R_\infty^\cris$) denote $(\ctensor_{j\in\cJ} R_{\rhobar_j}^{\square,\psi_{j,\pcris}})[[x_1,\ldots,x_h]]$  (resp. $(\ctensor_{j\in\cJ}R_{\rhobar_j}^{\square,\psi_{j,\cris}})[[x_1,\ldots,x_h]]$) for some $h\geq0$ and write $\m_\infty$ (resp. $\m_\infty^\cris$) for its maximal ideal. Let $\star\in \{\emptyset,\cris\}$, so that for example $R_\infty^\star = R_\infty$ or $R_\infty^\cris$. 

\begin{definition}
We say that a pair of a $R_\infty$-module and a $R_\infty^\cris$-module $(\Minf,\Minf^\cris)$ is \emph{a congruent pair of patched modules for $(\rhobar_j,\psi_{j,\pcris},\psi_{j,\cris})_{j\in \cJ}$} if it satisfies axioms \textbf{(PM1)}--\textbf{(PM6)} below. 
\end{definition}
\begin{itemize}
\item [\textbf{(PM1)}] $\Minf^\star$ is a finitely generated $R_\infty^\star[[K^\cJ]]$-module with a compatible $R_\infty^\star[G^\cJ]$-action.
\item [\textbf{(PM2)}]$\Minf^\star$ is a projective profinite $\cO[[K^\cJ]]$-module.
\end{itemize}
If $V$ is a finitely generated $\cO$-module with a continuous $K^\cJ$-action, we define the $R_\infty^\star$-module $\Minf^\star(V)=\Hom^{\cts}_{\cO[K^\cJ]}(V,(\Minf^\star)\dual)\dual$. By \textbf{(PM2)} $\Minf^\star(-)$ is an exact functor. Observe that there is a tautological right action of $\cH(V):=\End_{G^\cJ}(\ind_{K^\cJ}^{G^\cJ}(V))$ on $\Minf^\star(V)$. If $V$ is $\cO$-torsion free, then there is a canonical isomorphism
\begin{equation*}
\Minf^\star(V)\cong \Hom_{\cO[K^\cJ]}^\cts(V,(\Minf^\star)^d)^d
\end{equation*}
(\cite[Lemma 4.14]{CEG+}); in particular $\Minf^\star(V)$ is also $\cO$-torsion free. 

\begin{definition}
If $\tau$ is a regular principal series type we set $R_\infty (\sigma(\tau)):=(\ctensor_{j\in\cJ}R_{\rhobar_j}^{\eta,\tau,\psi_{j,\pcris}})[[x_1,\ldots,x_h]]$. 

If $\mu\in X^*(T)^+$ and $\sigma=V(\mu)_{\Qp}\otimes E|_K$ is the restriction to $K$ of a Weyl module for $G$ of highest weight $\mu$ with $E$-coefficients then we define $R_\infty^{\cris}(\sigma):=(\ctensor_{j\in\cJ}R_{\rhobar_j}^{\bar{\mu}+\eta,\psi_{j,\cris}})[[x_1,\ldots,x_h]]$. 
\end{definition}

For any $K$-representation $\sigma$, we write  $\sigma^\cJ$ for the $K^\cJ$-representation  $\otimes_{\cJ}\sigma$. 

\begin{itemize}
\item[\textbf{(PM3)}] 
Let $\sigma = \sigma(\tau)$ and $\sigma^\cris =V(\mu)_{\Q_p}\otimes E|_K $. Let $(\sigma^{\star,\cJ})^\circ \subset \sigma^{\star,\cJ}$ be a $K^\cJ$-equivariant $\cO$-lattice.  If $\Minf^\star((\sigma^{\star,\cJ})^\circ)\neq 0$ then the action of $R_\infty^\star$ on $\Minf^\star((\sigma^{\star,\cJ})^\circ)$ is maximal Cohen-Macaulay over $R_\infty^\star (\sigma^{\star})$. Furthermore $M_\infty^\star(\sigma^{\star,\cJ}):=  M_\infty^\star((\sigma^{\star,\cJ})^\circ)[1/p]$ is projective of (constant) rank $d\geq 1$ over $R_\infty^\star(\sigma^\star)[1/p]$.

\item[\textbf{(PM4)}] Let $\tau$ be a regular principal series type and $w_j\in W$ for each $j\in\cJ$. The tautological action of $\cH((\sigma(\tau)^{w_j}))_\cJ := \otimes_{j\in \cJ}\cH(\sigma(\tau)^{w_j})$ on $\Minf(\sigma(\tau)^\cJ)$ is through 
\begin{equation*}
\xymatrix{ \otimes_{j\in \cJ}\cH(\sigma(\tau)^{w_j})\ar[r]^{\otimes_{\cJ} \Theta^{w_j}} & \ctensor_{\cJ} R_{\rhobar}^{\eta,\tau}[1/p]\ar[r] & R_\infty(\sigma(\tau))[1/p]. }
\end{equation*}

\item[\textbf{(PM5)}] There is a canonical isomorphism between $R_\infty/\varpi = R_\infty^\cris/\varpi$-modules $M_\infty/\varpi \simeq \Minf^\cris/\varpi$.

\end{itemize}

Fix an element $w\in \cJ$ and write $\rhobar:= \rhobar_w$, $\psi_{\star} := \psi_{w,\star}$. We write $M_{\infty,w} = \Hom_{\cO[\prod_{\cJ\backslash\{w\}}K]}(\Minf,(\sigma(\tau_0)^{\cJ\backslash\{w\}})^d)^d$, $R_{\infty,w} = R_{\rhobar}^{\square,\psi_\pcris}[[x_1,\ldots, x_h]]$, and $\m_{\infty,w}$ its maximal ideal. Note that $R_\infty$ is a $R_{\infty,w}$-algebra and $M_{\infty,w}$ is $R_{\infty,w}$-module in an obvious way.   From \textbf{(PM1)} and \textbf{(PM2)} we see that 
\begin{equation}
\pi:= (M_{\infty,w}/\m_{\infty,w})\dual 
\end{equation}
is a smooth admissible $\F[G]$-module with $G$-action induced by $w$ component in $G^\cJ$. Our abstract mod $p$ local-global compatibility result says that one can recover (the Fontaine-Laffaille invariants of) $\rhobar$ from the $G$-action on $\pi$.

\begin{definition}
We define $W(\Minf^\star)$ to be the set of irreducible $\F[K^\cJ]$-representations  (Serre weights) $F$ such that $\Minf^\star(F)\neq 0$. 
\end{definition}

Axioms \textbf{(PM1)}-\textbf{(PM5)} make sense for any $\rhobar$, but the next axiom is specialized to our particular choice.

\begin{itemize}
\item[\textbf{(PM6)}]  Let $\mu_{\rhobar}:=(a_2-a_0-2,a_1-a_0-1;a_0+a_3-3)\in X^*(T)^+$. We assume $F(\mu_{\rhobar})^\cJ \in W(M_\infty^\cris)$.  
\end{itemize}

\begin{remark}
(i) The existence of a congruent pair of patched modules  for $(\rhobar,\psi,\psi_\cris)$ is proved in \S\ref{sec:existence} using the Taylor-Wiles patching method for two different setups under certain hypotheses. The variables $x_1,\ldots,x_h$, which play no role in the argument in this section, come from the patching construction.

(ii) In \cite{CEG+}, analogues of \textbf{(PM3)} and \textbf{(PM4)} hold for all locally algebraic types $\sigma$. We have decided to work only in the generality that we need.
\end{remark}

In the next theorem we write
\begin{equation*}
\Pi:=\begin{pmatrix}&&1&\\&&&1\\p&&&\\&p&&\end{pmatrix}\in G.
\end{equation*}
It normalizes $\Iw$ and $\Iw_1$. Fix $\tau_0$ and $\lambda=(0,0,-1,-1)\in X_*(T)$ as in \S\ref{sec:rhobar}. The action of $\Pi^{-1}$ on a space of $\Iw_1$-invariants is the same as that of the element $T_{\tw_\Pi}\in \cH_1$, where $\tw_\Pi:=t_{-\lambda}w_\Pi\in \tW^{\vee(1)}$, $w_\Pi:=s_1s_0s_1\in W$. Since $\tw_\Pi$ has length zero, we have
\begin{equation}\label{eq:Pi decomposition}
T_{\lambda}T_{\tw_\Pi}=T_{w_\Pi}
\end{equation}
inside $\cH_1$ by the braid relations (cf. for example \cite{Vig}). If $w\in W$, we have the notation $\sigma(\tau_0)^w=\Ind_{\Iw}^K(w\mu_0)_{\cO}$ from \S\ref{sec:LLC} and we write $\sigma(\tau_0)^{w,i}$ for what was denoted $M(w\mu_0)(w_\Pi,i)$ in \S\ref{sec:jantzen filtration} so we have a natural inclusion $\alpha:\sigma(\tau_0)^{w,i}\rightarrow \sigma(\tau_0)^w$. Note these are isomorphic to $K$-equivariant $\cO$-lattices in $\sigma(\tau_0)$. Write $w':=w_\Pi w$ so we have the morphism of $\cO[K]$-modules $T_{w_\Pi}:\sigma(\tau_0)^w\rightarrow \sigma(\tau_0)^{w'}$. If $M$ is any $\cO$-module we write $\overline{M}:=M\otimes_\cO\F$. We will use the diagram of $\F[K]$-modules
\begin{equation}\label{eq:roof diagram}
\xymatrix{
\overline{\sigma(\tau_0)}^{w,i} \ar[r]^{\alpha\otimes\F}\ar[d]^{T'\otimes \F} & \overline{\sigma(\tau_0)}^w \\
 \overline{\sigma(\tau_0)}^{w'} & 
}
\end{equation}
where $T':\sigma(\tau_0)^{w,i}\rightarrow \sigma(\tau_0)^{w'}$ is equal to $\frac{1}{p^i}T_{w_\Pi}\circ\alpha$. 

\begin{theorem}\label{thm:abstract loc glob}
Let $(\rhobar_j,\psi_{j,\pcris},\psi_{j,\cris})_{j\in \cJ}$ be as above and $(\Minf,\Minf^\cris)$ be a congruent  pair of patched modules   for $(\rhobar_j,\psi_{j,\pcris},\psi_{j,\cris})_{j\in \cJ}$. Suppose that $\rhobar_j$ is strongly 7-generic for each $j\in \cJ$. Then with notation as in the above paragraph,
\begin{itemize}
\item[(1)] $W(\Minf)\cap \JH(\overline{\sigma(\tau_0)^{\cJ}})=\{F(\mu_{\rhobar})^{\cJ}\}$, and
\item[(2)] Let $w\in \{\id,s_0,s_1s_0,s_0s_1s_0\}$ and set $k_{w'},r_{w'}$ as in Theorem \ref{thm:main result galois}. Taking $i=k_{w'}$, both maps in \eqref{eq:roof diagram} become isomorphisms after applying $\Hom_K(-,\pi)$. Moreover the resulting composite isomorphism $\pi^{\Iw=w'\mu_0}\rightarrow \pi^{\Iw=w\mu_0}$ (by Frobenius reciprocity) is equal to $\bar{r}_{w'}\Pi^\mo$. 
\end{itemize}
\end{theorem}
\begin{remark}
In particular the scalars $\bar{r}_{w'}\in\F\x$ for $w\in \{\id,s_0,s_1s_0,s_0s_1s_0\}$ are uniquely determined by the action of $G$ on $\pi$. By Remark \ref{rk:scalars recover FL invariants}, these scalars (together with the diagonal characters of $\rhobar$) uniquely determine $\rhobar$. 
\end{remark}

\begin{proof}
By Theorem \ref{thm:main result galois}, $R_\infty(\sigma(\tau_0))$ is formally smooth over $\cO$.  Let $\sigma, \sigma^\cris$ be as in \textbf{(PM3)} with $\tau = \tau_0$ and $\mu = \mu_{\rhobar}$. It follows from \textbf{(PM3)}, \textbf{(PM6)}, and the Auslander-Buchsbaum formula that $M_\infty((\sigma^\cJ)^\circ)$ is free of rank $d$ over $R_\infty(\sigma(\tau_0))$. Similarly  by Proposition \ref{prop:fs crystalline deformation rings}, $\Minf^{\cris}((\sigma^{\cris,\cJ})^\circ)$ is free of rank $d$ over $R_\infty^{\cris}(\sigma^{\cris})$. Hence by Lemma \ref{lem:Weyl decomposition}, $\dim_\F(\Minf^{\cris}(F(\mu_{\rhobar})^{\cJ})/\m_{\infty})=d$. By \textbf{(PM5)}, we have $\Minf(F(\mu_{\rhobar})^{\cJ})\simeq \Minf^\cris(F(\mu_{\rhobar})^{\cJ})$ as modules over $R_\infty/\varpi = R_\infty^\cris/\varpi$.

We can choose $\sigma(\tau_0)^\circ$ so that its reduction modulo $\varpi$ has either socle or cosocle isomorphic to $F(\mu_{\rhobar})$ by \cite[Lemma 4.1.1]{EGS}, giving rise to maps $\Minf(F(\mu_{\rhobar})^{\cJ})\hookrightarrow \Minf(\overline{(\sigma^{\cJ})^\circ})$ and $\Minf(\overline{(\sigma^{\cJ})^\circ})\onto \Minf(F(\mu_{\rhobar})^{\cJ})$ respectively.   The former choice together with the paragraph above shows that $\Minf(F(\mu_{\rhobar})^{\cJ})$ is free of rank $d$ over $R_\infty (\sigma(\tau_0))/\varpi = R_\infty^{\cris}(\sigma)/\varpi$  and the latter choice now implies that $W(\Minf)\cap \JH(\overline{\sigma(\tau_0)^{\cJ}})=\{F(\mu_{\rhobar})^{\cJ}\}$ as claimed.

Now applying the functor $\Minf((\sigma^{\cJ\backslash\{w\}})^\circ\otimes(-))$ to \eqref{eq:roof diagram} taking $i=k_{w'}$, by (1), $\alpha\otimes\F$ will become an isomorphism if and only if $F(\mu_{\rhobar})$ is a constituent of its image, which is to say a constituent of $\overline{M}(w\mu_0)(w_\Pi,k_{w'})$. This follows when $w\in \{\id,s_0,s_1s_0,s_0s_1s_0\}$ by Proposition \ref{prop:graded pieces}. 

On the other hand, by Theorem \ref{thm:main result galois}, \textbf{(PM4)}, and \eqref{eq:Pi decomposition} we have a  commutative diagram
\begin{equation*}
\xymatrix{
\Minf((\sigma^{\cJ\backslash\{w\}})^\circ\otimes\sigma(\tau_0)^{w,k_{w'}})\ar@{^(->}[r]^{\alpha}\ar[d]^{T'} & \Minf((\sigma^{\cJ\backslash\{w\}})^\circ\otimes\sigma(\tau_0)^w) \ar[d]^{T_{\tw_\Pi}}_{\wr}\\
\Minf((\sigma^{\cJ\backslash\{w\}})^\circ\otimes\sigma(\tau_0)^{w'}) & \Minf((\sigma^{\cJ\backslash\{w\}})^\circ\otimes\sigma(\tau_0)^{w'}) \ar[l]_{r_{w'}}.
}
\end{equation*}
In particular $T'\otimes\F$ is also an isomorphism. Claim (2) follows now by reducing this diagram modulo $\m_{\infty,w}$ and applying Pontryagin duality. 
\end{proof}

\begin{remark}\label{rk:group algebra operators}
The statement of Theorem \ref{thm:abstract loc glob}(2) can be interpreted more concretely. It says that for any $v\in \pi^{\Iw=w'\mu_0}$ there exist (many) $S\in M(w\mu_0)_\F(w_\Pi,k_{w'})$ such that $S'v\neq 0$ and there is an equality
\begin{equation*}
S\Pi^{-1}v=r_{w'}^\mo S'v
\end{equation*}
where $S'\in M(w'\mu_0)_\F$ is given by $T'(\widetilde{S})\mod \varpi$, with $\widetilde{S}\in M(w\mu_0)_\cO$ being the Teichm\"uller lift of $S$. In fact, by (1) we may choose an $S$ that works for all $v$ simultaneously: choose any $S$ whose image in the unique quotient of $M(w\mu_0)_\F$ of socle $F(\mu_{\rhobar})$ is nonzero.
\end{remark}

\section{Existence of patched modules}\label{sec:existence} 
In this section, we prove in some circumstances the existence of a congruent pair of patched modules $(\Minf,M_{\infty}^\cris)$ satisfying the axioms of \S\ref{subsec:patched module axioms}, by patching spaces of automorphic forms on a compact mod centre form of $\GSp_4$. Consequently we obtain a mod $p$ local-global compatibility theorem for this group.  

\subsection{Some automorphic representations}\label{sec:automorphic representations}

Let $F$ be a totally real number field of even degree over $\Q$. Let $\cG$ be the $F$-group $\mathrm{GU}_2(D)$ where $D$ is a quaternion algebra over $F$ ramified at all infinite places and split at all finite places (such a $D$ exists because $|F:\Q|$ is even). Then $\cG$ is an inner form of $\GSp_4$ which is compact mod centre at infinity and split at all finite places. The centre is isomorphic to $\Gm$. Choosing a maximal order $\cO_D$ defines an $\cO_F$-structure on $\cG$; for each finite place $v$ we fix an isomorphism $\cO_{D,v}\cong M_2(\cO_{F_v})$. This determines an isomorphism $\iota_v:\cG_{F_v}\rightarrow \GSp_{4/F_v}$ which restricts to an isomorphism $\cG(\cO_{F_v})\cong \GSp_4(\cO_{F_v})$. For each infinite place $v$ we also fix an isomorphism $\iota_v:\cG\otimes_{F,v}\C\rightarrow \GSp_{4/\C}$. Let $S_\infty$ denote the set of infinite places of $F$.

 If $v\in S_\infty$, equivalence classes of discrete $L$-parameters $W_{F_v}\rightarrow \GSp_4(\C)$ are in bijection with a triple $(w_v;k_v;l_v)\in \Z^3$ such that $k_v > l_v \ge 0$ and $k_v+l_v\equiv w+1 \mod 2$ (cf.~\cite[\S3.1]{Mok}). Explicitly, if we write $W_{F_v}=\C\x\cup \C\x j$, the discrete $L$-parameter $\phi_{(w_v;k_v,l_v)}$ corresponding to  $(w_v;k_v;l_v)$ is given by
\begin{equation*}
re^{i\theta}\mapsto r^{-w_v}\cdot \begin{pmatrix}e^{i\theta(k_v+l_v)}&&&\\ & e^{i\theta(k_v-l_v)}&&\\&&e^{i\theta(-k_v+l_v)}&\\ &&&e^{i\theta(-k_v-l_v)}\end{pmatrix}
\end{equation*}
and
\begin{equation*}
j\mapsto\begin{pmatrix}&&&1\\ &&1&\\&(-1)^{w_v+1}&&\\(-1)^{w_v+1}&&& \end{pmatrix}.
\end{equation*}
Note the similitude character takes $z\mapsto |z|^{-2w_v}$ and $j\mapsto (-1)^{w_v}$.

For a dominant weight $\mu_v=(a_v,b_v;c_v)\in X^*(T)^+$, we let $\xi_{\mu_v}$ be the irreducible $\C$-representation of $\cG(F_v)$ corresponding to the $L$-parameter $\phi_{(-c_v-6;a_v+2,b_v+1)}$. 
If $\mu\in(X^*(T)^+)^{S_\infty}$ we define $\xi_\mu=\bigotimes_{v\in S_\infty} \xi_{\mu_v}$. On the other hand, the $L$-packet of $\mu_v$ for $\GSp_4(F_v)$ contains two representations, one of which is holomorphic and the other generic. We write $\Pi_{\mu_v}^h$ and $\Pi_{\mu_v}^g$ for these, respectively. 

Fix a Hecke character $\chi:\A_F\x\rightarrow \C\x$, serving as a central character in all that follows.  We always assume that our $\infty$-types $\mu$ are compatible with $\chi$, which is to say that $\chi_v=(a\mapsto \abs{a}^{c_v+6})$ for all $v\in S_\infty$.

Let $\cA_\chi$ denote the space of automorphic forms on $\cG(\A_F)$ of central character $\chi$, and let $\cA_{\chi,\mu}$ denote the $\xi_\mu$-isotypic part. Since $\cG(F_\infty)$ is compact mod centre, $\cA_\chi=\bigoplus_\mu \cA_{\chi,\mu}$ and
\begin{equation*}
\cA_{\chi,\mu}=\bigoplus_{\pi} m_\pi\cdot\pi
\end{equation*}
for some finite multiplicity $m_\pi\geq 1$, the sum running over automorphic representations $\pi$ of $\cG$ of central character $\chi$ such that $\pi_\infty\cong \xi_\mu$.

A multiplicity-preserving Jacquet-Langlands transfer from automorphic representations of $\cG$ to $\GSp_4$ was proven by Sorensen in \cite[Theorem B]{Sor09}  for stable and tempered representations. (See \cite[\S1]{Sor09} for the definition of stable and tempered.)

Arthur described a classification of the discrete spectrum of $\GSp_4$ with central character $\chi$ in \cite{Art04}. This classification was proven in \cite{GeeTaibi}, conditional on unpublished results of Arthur, Moeglin and Waldspurger. We freely use consequences of the classification below.

\begin{corollary}
Stable and tempered automorphic representations of $\cG$ occur with multiplicity one in the discrete spectrum. 
\end{corollary}
\begin{proof}
Automorphic representations of $\GSp_4$ of general type (meaning its Arthur parameter is a $\chi$-self dual cuspidal automorphic representation of $\GL_4$ of symplectic type - equivalently $\Pi$ is non-CAP and non-endoscopic) occur with multiplicity one in the discrete spectrum by Arthur's multiplicity theorem (cf. \cite{GeeTaibi} or \cite[Theorem 2.9.3]{BCGP}). It follows from  Arthur's classification that an automorphic representation of $\GSp_4$ being of general type is equivalent to  being stable  and tempered. Now the claim follows from \cite[Theorem B]{Sor09}.
\end{proof}

From now on we assume that there exists $w\in\Z$ such that
\begin{equation}\label{eq:mok condition}
\textrm{$c_v=w$ for each infinite place $v$}.
\end{equation}
This assumption is necessary because any algebraic Hecke character over totally real field has a constant weight.

\begin{definition}[cf. Definition 7.6.2 of \cite{BCGP}]\label{def:odd}
If $r:G_F\rightarrow \GSp_4(\Qlbar)$ is a continuous homomorphism, we say that $r$ is \emph{odd} if for each infinite place $v$ and corresponding choice of complex conjugation $c_v\in G_F$ we have $\nusim(r(c_v))=-1$. 
\end{definition}

\begin{theorem}\label{thm:galois representations}
Assume that $\chi$ satisfies \eqref{eq:mok condition}, and let $\pi$ be a stable and tempered automorphic representation of $\cG$ that contributes to $\cA_{\chi,\mu}$. For any prime $\ell$ and choice of isomorphism $\iota:\Qlbar\xrightarrow{\sim}\C$ there exists a (unique up to isomorphism) continuous semisimple representation
\begin{equation*}
r_{\pi,\ell,\iota}: G_F\rightarrow \GSp_4(\Qlbar)
\end{equation*}
satisfying:
\begin{itemize}
\item[(i)] $r_{\pi,\ell,\iota}$ is de Rham at places dividing $\ell$. 

\item[(ii)] $r_{\pi,\ell,\iota}$ obeys local-global compatibility 
\begin{equation*}
\iota\circ\WD(r_{\pi,\ell,\iota}|_{G_{F_v}})^{\mathrm{F-ss}}\cong \recGT(\pi_v\otimes |\nusim|^{-3/2})
\end{equation*}
for each finite place $v$ of $F$. If $v|\ell$, the Hodge-Tate weights of $r_{\pi,\ell,\iota}|_{G_{F_v}}$ are given by $\bar{\mu}_v+(3,2,1,0)$.

\item[(iii)] $r_{\pi,\ell,\iota}$ has similitude character $\chi_{\ell,\iota} \epsilon^{-3}$ (where $\chi_{\ell,\iota}$ is the $\ell$-adic realization of $\chi$ through $\iota$).
\end{itemize}
\end{theorem}
\begin{proof}
The existence of $r_{\pi,\ell,\iota}$ satisfying local-global compatibility with the stated Hodge-Tate weights follows from \cite[Theorem B]{Sor09} together with Theorem 3.5 of \cite{Mok}. 
The symplecticity and similitude character in (i) follow from Corollary 1.3 of \cite{BC11} (cf. \cite[Remark 3.3(3)]{Mok}). 
\end{proof}

For $\pi$ a regular cuspidal automorphic representation of $\GSp_4(\A_F)$ with central character $\chi$, we also write $r_{\pi,l,\iota}$ for a continuous semisimple representation
\begin{align*}
    r_{\pi,l,\iota}: G_F \rightarrow \GSp_4(\Qpbar)
\end{align*}
satisfying items (i)-(iii) above constructed by \cite[Theorem 3.5]{Mok}.

\begin{remark}
(i) One expects $r_{\pi,\ell,\iota}$ to always be irreducible if $\pi$ is of general type. This is known  for $F=\Q$  and $l \ge 5$ (\cite[Theorem 1.1]{Weiss2022}) and $F$ totally real and sufficiently large $l$ (\cite[Theorem A]{WeissThesis}).

(ii) The existence of these Galois representations is compatible with the conjectures of \cite{BG}: the automorphic representation of $\GSp_4$ corresponding to $\pi$ is $C$-algebraic and one twists this into an $L$-algebraic automorphic representation (see \cite[\S5]{BG}). 

(iii) It follows immediately from the theorem that 
$r_{\pi,\ell,\iota}:G_F\rightarrow \GSp_4(\Qlbar)$ is odd in the sense of Definition \ref{def:odd}.
\end{remark}

We finish this subsection by proving that all automorphic representations of $\cG(\A_F)$ of our interest are stable and tempered.

Let $\pi$ be a regular algebraic cuspidal automorphic representation of $\GSp_4(\A_F)$ or $\cG(\A_F)$. We say that a continuous representation $\rbar:G_F \rightarrow \GSp_4(\F)$ is \emph{attached to $\pi$} if for almost all finite places $v$ at which $\pi$ and $\rbar$ are unramified, we have
\begin{align*}
    \iota\circ \WD(\rbar|_{G_{F_v}}) \simeq \recGT(\pi_v\otimes |\nusim|^{-3/2}) \mod \varpi.
\end{align*}

\begin{lemma}\label{lem:weak-JL}
    Let $\pi'$ be a regular algebraic cuspidal automorphic representation of $\cG(\A_F)$. Suppose that there exists  a continuous semisimple representation $\rbar:G_F \rightarrow \GSp_4(\F)$ attached to $\pi'$. Then there exists a regular algebraic cuspidal automorphic representation $\pi$ of $\GSp_4(\A_F)$ which is a weak Jacquet--Langlands transfer of $\pi'$. In other words, we have $\pi_v\simeq\pi'_v$ for almost all finite places $v$ of $F$.
\end{lemma}
\begin{proof}
    Let $\cH:= \cG^{\mathrm{der}}$ be the derived subgroup of $\cG$. Then $\cH$ is an inner form of $\Sp_4$ which splits at all finite places. By a Theorem of Hiraga--Saito (\cite[Theorem 5.1.2]{GeeTaibi}), $\pi'|_{\cH(\A_F)}$ contains a regular algebraic cuspidal automorphic representation of $\cH(\A_F)$ which we denote by $\Tilde{\pi}'$. By \cite[Theorem 4.0.1]{Taibi}, there exists a self-dual regular algebraic discrete automorphic representation $\Pi$ of $\GL_5(\A_F)$ which is a transfer of $\Tilde{\pi}'$. It is moreover cuspidal by \cite[Theorem A]{Ramakrishnan}. Then by \cite[Theorem 1.4.1]{Art13}, there exists a regular algebraic discrete automorphic representation $\Tilde{\pi}$ of $\Sp_4(\A_F)$ whose transfer to $\GL_5(\A_F)$ is $\Pi$. Since $\Pi$ is algebraic, $\Pi_{\infty}$ is essentially tempered by \cite[Lemma 4.9]{Clozel90}, and $\pi_{\infty}$ is also essentially tempered, so that $\pi$ is cuspidal by \cite[Theorem 4.3]{Wallach}.  Note that at all finite places $v$, $\Tilde{\pi}'_v$ and $\Tilde{\pi}_v$ have the same $L$-parameter (these $L$-parameters are constructed in \cite{Art13}).

    By \cite[Proposition 3.1.4]{Patrikis-thesis}, there exists a cuspidal automorphic representation $\pi$ of $\GSp_4(\A_F)$ lifting $\Tilde{\pi}$ and whose central character is equal to the central character of $\pi'$. 
    For any finite place $v$ of $F$, $\std'(\rbar)_v$ and $\std'(\rbar_{\pi,p,\iota})_v$ are isomorphic, as they are mod $p$ reduction of $L$-parameters attached to $\Tilde{\pi}'_v$ and $\Tilde{\pi}_v$ respectively (here, we use the compatibility between the local Langlands correspondence for $\GSp_4$ and $\Sp_4$ upon restriction; see \cite[Main Theorem (v)]{GTSp4}).   By Brauer--Nesbitt Theorem, we have $\std'(\rbar)\simeq\std'(\rbar_{\pi,p,\iota})$. Since $\rbar_{\pi,p,\iota}$ and $\rbar$ also have the same similitude character, we have $\rbar \simeq \rbar_{\pi,p,\iota}\otimes \theta$ for some quadratic character $\theta: G_F \rightarrow \Qpbar^\times$. By twisting $\pi$ by $\theta^{\mo}$, we can ensure that $\rbar_{\pi,p,\iota}\simeq \rbar$.

    Let $v$ be a finite place of $F$ at which $\pi$ and $\pi'$ are unramified. Then $\recGT(\pi_v\otimes |\nusim|^{-3/2})$ and $\recGT(\pi'_v\otimes |\nusim|^{-3/2})$ have the same similitude character and are isomorphic after composed with $\std'$ (again, by using the local Langlands for $\Sp_4$). Thus,  $\recGT(\pi_v\otimes |\nusim|^{-3/2})\simeq \recGT(\pi'_v\otimes |\nusim|^{-3/2})\otimes \theta_v$ for some quadratic character $\theta_v: W_{F_v} \rightarrow \Qpbar^\times$. However, by the previous paragraph, we have $\recGT(\pi_v\otimes |\nusim|^{-3/2}) \simeq \recGT(\pi'_v\otimes |\nusim|^{-3/2}) \mod \varpi$, which means $\theta_v$ is trivial modulo $\varpi$. Since $\theta_v$ has order at most 2 and $p>2$, this implies that $\theta_v$ is trivial and $\recGT(\pi_v\otimes |\nusim|^{-3/2}) \simeq \recGT(\pi'_v\otimes |\nusim|^{-3/2})$. Since the $L$-packet of an unramified $L$-parameter valued in $\GSp_4$ is a singleton, we have $\pi_v \simeq \pi'_v$.
\end{proof}

\begin{lemma}\label{lem:stable-tempered}
    Let $\pi$ be a regular cuspidal automorphic representation of $G(\A_F)$ where $G=\GSp_4$ or $G=\cG$. Suppose that there exists a continuous irreducible representation $\rbar:G_F \rightarrow \GSp_4(\A_F)$ attached to $\pi$, then $\pi$ is stable and tempered.
\end{lemma}
\begin{proof}
    If $G=\cG$, then we can apply Lemma \ref{lem:weak-JL} to replace $\pi$ by a regular cuspidal automorphic representation of $\GSp_4(\A_F)$ to which $\rbar$ is attached. Since being stable and tempered is characterized by components at almost all finite places, it suffices to prove the claim for $G=\GSp_4$, and $\rbar \simeq \rbar_{\pi,p,\iota}$ in this case. By \cite[Theorem 2.9.3]{BCGP}, the irreducibility of $\rbar$ implies that $\pi$ is of general type, which implies that $\pi$ is stable and tempered by Arthur's classification (see item (a)-(f) at the end of \cite{Art04}). 

    To be more precise, suppose that $\pi$ is not stable, i.e.~$\pi$ is either CAP or endoscopic (again, in the sense of \cite{Sor09}). If $\pi$ is CAP, then by the results of Piatetski-Shapiro and Soudry \cite{PS,Soudry}, $\pi$ can be realized as a theta lift (such $\pi$ corresponds to either Soudry type or Saito--Kurokawa type in \cite{Art04}). Using this, one can easily see that $\rbar$ has to be reducible. If $\pi$ is endoscopic, it corresponds to Yoshida type in \emph{loc.~cit.}, and $\rbar$ in this case has to be reducible as well.

    To show that $\pi$ is tempered, note that the transfer $\Pi$ of $\pi$ to $\GL_4$ is cuspidal. By Ramanujan conjecture (proven in \cite{Shin} in this case), $\Pi_v$ is tempered for all places $v$. This implies that the $L$-parameter attached to $\Pi_v$ has bounded image. By \cite[Main Theorem (vii)]{GT}, the $L$-packet containing $\pi_v$ has a unique generic tempered element. Then for any finite place $v$ at which $\pi$ is unramified, $\pi_v$ is generic tempered because the $L$-packet containing it is a singleton. 
\end{proof}

\subsection{$p$-adic automorphic forms}\label{sec:p-adic automorphic forms}

We now define integral models of spaces of automorphic forms on $\cG$ over $\cO$. Fix a choice of isomorphism $\iota:\Qpbar\rightarrow\C$, and let $\psi:=\chi_{p,\iota}\epsilon^{-3}$. All the constructions in this section depend on $\iota$, so we omit it from the notation. If $V$ is an $\cO$-module with a linear action of $\cG(\cO_{F_p})$ and $U\leq \cG(\A_F^{\infty,p})\times \cG(\cO_{F_p})$ is a compact open subgroup, we define $S_\chi(U,V)$ to be the $\cO$-module of functions $f:\cG(F)\backslash \cG(\A_F^\infty)\rightarrow V$ such that $f(z g)=\iota^{-1}(\chi(z))f(g)$ and $f(gu)=u_p^{-1} f(g)$ for $z\in Z(\A_F^\infty)$, $g\in \cG(\A_F^\infty)$ and $u\in U$. Note this space is zero unless $V$ has central character equal to $\iota^{-1}\circ\chi^{-1}|_{\cO_{F_p}\x}$.

\begin{definition}
We say that a compact open subgroup $U$ as above is \emph{sufficiently small}
if its projection to $\cG(F_v)$ contains no element of exact order $p$ for some finite place $v$.

 We say that $U$ is \emph{unramified} at a finite place $v$ if $U=U_vU^v$ where $U_v=\cG(\cO_{F_v})$. 
\end{definition}

By finiteness of class numbers, $S_\chi(U,V)$ is finitely generated (resp. free) over $\cO$ whenever $V$ is, and if $U$ is sufficiently small then $S_\chi(U,V\otimes_\cO A)=S_\chi(U,V)\otimes_\cO A$ for any $\cO$-module $A$ with trivial $\cG(\cO_{F_p})$-action. 

Let $P_U$ denote the finite set of finite places at which $U$ is ramified.  Let $S_p$ denote the set of places of $F$ dividing $p$. If $P\supseteq P_U\cup S_p$ is any finite set of finite places we let $\T^{P,\univ}$ denote the polynomial $\cO$-algebra generated by variables $S_v,T_{v,1},T_{v,2}$ for each $v\notin P$. Then $\T^{P,\univ}$ has a natural action on $S_\chi(U,V)$, where the elements above act through the double coset operators of $\beta_0(\varpi_v),\beta_1(\varpi_v),\beta_2(\varpi_v)$ in the spherical Hecke algebra $\cO[\GSp_4(F_v)//\GSp_4(\cO_{F_v})]$ respectively.

For simplicity we now assume that $p$ splits completely in $F$. Then we can identify places above $p$ with $\Q$-embeddings $F\hookrightarrow E$ and $\iota$ induces a bijections $v\mapsto \iota(v):S_p\rightarrow S_\infty$. For $\lambda\in X^*(T)^+$, let $V_\lambda$ denote the Weyl module of $\GSp_4$ over $\cO_F$ of highest weight $\lambda$. So $V_\lambda$ is a finite free $\cO_F$-module with an action of $\GSp_4(\cO_F)$. Now if $\mu=(\mu_v)_{v|p}\in (X^*(T)^+)^{\Hom_\Q(F,\Qpbar)}$ we let $V(\mu)$ denote the $\cO[\GSp_4(\cO_{F,p})]$-module 
\begin{equation*}
V(\mu)=\bigotimes_{v\in S_p} V_{\mu_v}\otimes_{\cO_{F_v}} \cO
\end{equation*}
which we regard a $\cO[\cG(\cO_{F,p})]$-module. Then $V(\mu)\otimes_{\cO,\iota}\C$ has a natural action of $\cG(F_\infty)$ making it isomorphic to $\xi_{\iota(\mu)}$.  Moreover, $\iota$ induces an isomorphism of $\Qpbar[\cG(\A_F^\infty)]$-modules 
\begin{equation}\label{eq:integral model}
\varinjlim_U S_\chi(U,V(\mu))\cong \cA_{\chi,\iota(\mu)}\otimes_{\C,\iota^{-1}} \Qpbar,
\end{equation}
where $U$ runs over all compact open subgroups of $\cG(\A_F^{\infty,p})\times\cG(\cO_{F,p})$. See for example the proof of Proposition 3.3.2 in \cite{CHT}.  We define
\begin{align*}
    S_{\chi,\mu}(U,W) := S_{\chi}(U,V(\mu)\otimes_\cO W)
\end{align*}
for any $\cO[\cG(\cO_{F_p})]$-module $W$ and level $U\le \cG(\A_F^{\infty,p})\times \cG(\cO_{F,p})$.

Suppose that $\rbar:G_F\rightarrow \GSp_4(\F)$ is an absolutely irreducible continuous representation. Let $P_{\rbar}$ denote the set of finite places either dividing $p$ or at which $\rbar$ is ramified. For any finite set of finite places $P\supseteq P_{\rbar}$ we define a maximal ideal $\m_{\rbar,\chi}^P\leq \T^{P,\univ}$ with residue field $\F$ by demanding that for each $v\notin P$, 
\begin{equation*}
S_v\mod \m_{\rbar,\chi}^P=\bar{\psi}(\Frob_v),
\end{equation*}
 and the characteristic polynomial $\rbar(\Frob_v)$ in $\F[X]$ is given by
\begin{equation*}
X^4-T_{v,1}X^3+(q_vT_{v,2}+(q_v^3+q_v)S_v)X^2-q_v^3T_{v,1}S_vX+q_v^6S_v^2\mod \m_{\rbar,\chi}^P
\end{equation*}
(cf. \cite[\S2.4.7]{BCGP}; note $\m_{\rbar,\chi}^P$ is well-defined by symplecticity of $\rbar$). 
\begin{definition}\label{def:modularity}
We say that a pair $(\rbar,\chi)$ as above is \emph{modular of weight $\mu$ (resp.~weight $F(\mu)$) and level $U$} if there exists a weight $\mu\in (X^*(T)^+)^{S_p}$ (resp.~$\mu\in (X_1(T))^{S_p}$ and $F(\mu):=\otimes_{v\in S_p} F(\mu_v)$), a level $U\leq \cG(\A_F^{\infty,p})\times \cG(\cO_{F_p})$, and a finite set of finite places $P$ containing $P_U \cup P_{\rbar}$ such that $S_{\chi,\mu}(U,\cO)_{\m_{\rbar,\chi}^P}\neq 0$ (resp. $S_{\chi}(U,F(\mu))_{\m_{\rbar,\chi}^P}\neq 0$).
\end{definition}

\begin{remark}\label{rk:modularity-equiv}
    A pair $(\rbar,\chi)$ being modular is equivalent to asking that there exists a regular cuspidal automorphic representation $\pi'$ of $\cG(\A_F)$ to which $\rbar$ is attached  such that $\pi^U\neq 0$ and $\iota^{-1}(\pi)$ contributes to \eqref{eq:integral model}. As such we are free to shrink $U$ and enlarge $P$. Moreover, such $\pi$ is necessarily stable and tempered by Lemma \ref{lem:stable-tempered}, and we can apply Theorem \ref{thm:galois representations} to obtain 
    \begin{align*}
        r_{\pi,p,\iota} : G_F \rightarrow \GSp_4(\Qpbar)
    \end{align*}
    such that $\nusim(r_{\pi,p,\iota}) = \chi_{p,\iota}\epsilon^{-3}$ and the reduction mod $p$ of $r_{\pi,p,\iota}$ is $\GL_4(\Fpbar)$-conjugate to $\rbar$. 
    
    Also, by applying \cite[Theorem B]{Sor09}, $(\rbar,\chi)$ being modular is in turn equivalent to asking that there exists a regular algebraic cuspidal automorphic representation $\pi$ of $\GSp_4(\A_F)$ with central character $\chi$ such that $\rbar_{\pi,p,\iota}\simeq \rbar$.
\end{remark}

\begin{definition}\label{def:pd-automorphic}
We say that a pair $(\rbar,\chi)$ is \emph{potentially diagonalizably automorphic} if it is modular and $\pi$ as above can be taken so that $\std(r_{\pi,p,\iota})|_{G_{F_v}}$ is potentially diagonalizable (in the sense of \cite[\S1.4]{BLGGT}) for all $v|p$.
\end{definition}


For any finite set of finite places $P\supseteq P_U$ let $\T_{\chi,\mu}^P(U)$ denote the image of $\T^{P,\univ}$ inside $\End_\cO(S_{\chi,\mu}(U,\cO))$. Observe that $\T_{\chi,\mu}^P(U,\cO)_{\m_{\rbar}^P}$ is an object of $\COhat$. 
\begin{lemma}\label{lem:reduced}
The ring $\T_{\chi,\mu}^P(U)_{\m_{\rbar}^P}$ is reduced. 
\end{lemma}
\begin{proof}
By commutative algebra the semisimplicity of $S_{\chi,\mu}(U,
\cO)\otimes_\cO \Qpbar$, which follows from \eqref{eq:integral model}, shows that $\T_{\chi,\mu}^P(U)$ is reduced. The claim follows.
\end{proof}
\begin{proposition}\label{prop:hecke valued galois representations}
Assume that $\rbar:G_F\rightarrow \GSp_4(\F)$ is absolutely irreducible and $(\rbar,\chi)$ is modular of weight $\mu$ and level $U$. Then for any finite set $P\supset P_U$ there exists a unique continuous homomorphism
\begin{equation*}
r_{\chi,\mu}^P(U):G_F\rightarrow \GSp_4(\T_{\chi,\mu}^P(U)_{\m_{\rbar,\chi}^P})
\end{equation*}
lifting $\rbar$ such that
\begin{itemize}
\item[(i)] $\nusim\circ r_{\chi,\mu}^P(U)= \psi$,
\item[(ii)] if $v\notin P$ then $r_{\chi,\mu}^P(U)$ is unramified and the characteristic polynomial of $r_{\chi,\mu}^P(U)(\Frob_v)$ in $\T^P_{\chi,\mu}(U)_{\m^P_{\rbar,\chi}}[X]$ is equal to
\begin{equation*}
X^4-T_{v,1}X^3+(q_vT_{v,2}+(q_v^3+q_v)\chi_v(\varpi_v))X^2-q_v^3T_{v,1}\chi_v(\varpi_v)X+ q_v^6 \chi_v(\varpi_v)^2,
\end{equation*}
\item[(iii)] and for every $\cO$-algebra homomorphism $\zeta:\T_{\chi,\mu}^P(U)_{\m_{\rbar}^P} \rightarrow E'$ where $E'$ is a finite extension of $E$, the representation $\zeta\circ r_{\chi,\mu}^P(U)|_{G_{F_v}}$ is de Rham of Hodge-Tate weights $\bar{\mu}_v+(3,2,1,0)$ for all $v\in S_p$. 
\end{itemize}
\end{proposition}
\begin{proof}
By Lemma \ref{lem:stable-tempered} and the irreducibility of $\rbar$, we know that the automorphic representation associated with Hecke eigensystem $\T^P_{\chi,\mu}(U)_{\m_{\rbar}^P} \xrightarrow{} E'$ is stable and tempered.  
Then the Proposition   can be shown by using Theorem \ref{thm:galois representations} and a standard argument using Lemma \ref{lem:reduced}; see for example the proof of \cite[Proposition 3.4.4]{CHT}.   In order to ensure the resulting representation is symplectic, one applies Lemma \ref{lem:symplectic reduction mod p}. Item (ii)  follows from \cite[\S2.4.7]{BCGP}.
\end{proof}

From now on, we fix a preferred place $w|p$. If $W$ is any $\cO[\cG(\cO_{F_w})]$-module and $\sigma$ is a finite $\cO$-module with smooth $\GSp_4(\Zp)$-action, we define
\begin{equation*}
S_{\chi,\mu,\sigma}(U^w,W):=\varinjlim_{U_w\leq \cG(\cO_{F_w})} S_{\chi}(U^wU_w,V(\mu')\otimes_\cO \sigma^{S_p \backslash\{w\}} \otimes_\cO W)
\end{equation*}
where $\mu'$ is equal to $\mu$ but with $\mu_w$ replaced by $(0,0;0)$, and $\sigma^{S_p \backslash\{w\}}$ is viewed as smooth $\cO[\prod_{v\in S_p \backslash\{w\}}\cG(\cO_{F_v})]$-module by identifying $\cG(\cO_{F_v}) \simeq \GSp_4(\Zp)$. 
Note that the functor $W\mapsto S_{\chi,\mu,\sigma}(U^w,W)$ is exact if $U^w$ is sufficiently small.

By definition $S_{\chi,\mu,\sigma}(U^w,\F)$ is an admissible representation of $\cG(F_w)\cong\GSp_4(\Qp)$ over $\F$ with central character $\iota^{-1}\circ\chi_w$. Suppose that $\ov{r}:G_F \rightarrow \GSp_4(\F)$ is modular of some weight and level. We will show that $S_{\chi,\mu,\sigma}(U^w,\F)[\m_{\rbar,\chi}^P]$ is nonzero for certain  $\mu$ and  $\sigma$. We are interested in studying the relationship between the $\F[\cG(F_w)]$-module $S_{\chi,
\mu,\sigma}(U^w,\F)[\m_{\rbar,\chi}^P]$ and the local Galois representation $\rbar|_{G_{F_w}}$. The main tool for doing this is the patching method described in the next two sections.

\subsection{A patching lemma}

In order to clarify the patching method used in the next section, we present here a formalized version. In this section $G$ denotes a locally profinite group having a countable fundamental system of open neighbourhoods at the identity. 

\begin{definition}
Let $A$ be a commutative ring and $\mj=\{\mj_1\supseteq\mj_2\supseteq\cdots\}$ a system of ideals. A \emph{$(G,A,\mj)$-inverse system} is an inverse system $\M=\{M_r(H)\}_{r\geq 1,H\leq_\co G}$ of finite $A$-modules together with a collection of morphisms of inverse systems $g_*:M_r(H)\rightarrow M_r(gHg^{-1})$ for each $g\in G$ satisfying $(gh)_*=g_*h_*$ and $\id_*=\id$. We assume that the transition maps have the property that they induce isomorphisms
\begin{equation*}
M_{r+1}(H)/\mj_r\xrightarrow{\sim} M_r(H)
\end{equation*}
and
\begin{equation*}
M_r(H')_H\xrightarrow{\sim}M_r(H)
\end{equation*}
for all $r\geq 1, H\leq_\co G$, and any open normal subgroup $H'\subseteq H$. (In particular, all transition maps are surjective.) These objects form a category in an obvious way.

We say that $\M$ is \emph{projective} if $M_r(H')$ is a projective $(A/\mj_r)[H/H']$-module whenever $H'\subseteq H$ is an open normal subgroup. 
\end{definition}

\begin{remark}
If $\mb=\{\mb_1\supseteq \mb_2\supseteq\cdots\}$ is another system of ideals of $A$ such that $\mb_r\supseteq \mj_r$ then $(M_r(K)/\mb_r)_{r,K}$ is a (projective) $(G,A,\mb)$-inverse system.
\end{remark}

\begin{definition}[cf. \cite{EmeOrd}]
Let $A$ be a complete local noetherian $\cO$-algebra with finite residue field. We let $\Mod^{\textrm{pro aug}}_G(A)$ denote the abelian category of left $A[G]$-modules $M$ equipped with a profinite topology such that for any compact open $H\subseteq G$, the $A[H]$-action extends to an action of $A[[H]]$ making $M$ into a linear topological $A[[H]]$-module for the profinite topology on $A[[H]]$.
\end{definition}

\begin{lemma}\label{lem:inverse limit}
Let $A$ be a complete local noetherian $\cO$-algebra with finite residue field, and suppose that the topology determined by $\mj$ is the max-adic one. The functor
\begin{equation*}
\M\mapsto \M_\infty:= \varprojlim_{r,H} M_r(H)
\end{equation*}
is an equivalence of categories 
 between the category of $(G,A,\mj)$-inverse systems and $\Mod^{\textrm{\emph{{pro aug}}}}_G(A)$. Moreover, if $\M$ is projective then $\M_\infty$ is a finitely generated projective $A[[H]]$-module for any compact open $H\subseteq G$. 
\end{lemma}
\begin{proof}
The equivalence of categories follows from \cite[Lemma 2.2.6]{EmeOrd}. In particular, for any open subgroup $H\subseteq G$ we have $(\M_\infty)_H/\mj_r=M_r(H)$. Using this, the topological Nakayama lemma for $A[[H]]$ shows that $\M_\infty$ is finitely generated over $A[[H]]$. The argument in the proof of \cite[Proposition 2.10]{CEG+} now shows that $\M_\infty$ is a projective $A[[H]]$-module if $\M$ is projective.
\end{proof}

\begin{definition}\label{def:G patching datum}
A \emph{$G$-patching datum over $\cO$} is a tuple
\begin{equation*}
(S_\infty,R_\infty,(R_n,\varphi_n,\M_n,\alpha_n)_{n\geq 1},\M_0)
\end{equation*}
where
\begin{itemize}
\item $S_\infty$ is a power series ring over $\cO$ with augmentation ideal $\ma_\infty=\ker(S_\infty\rightarrow\cO)$,
\item $R_\infty$ is a complete local $\cO$-algebra with finite residue field,
\item $R_n$ is a local $S_\infty$-algebra for $n\geq 1$,
\item $\varphi_n:R_\infty\onto R_n$ is a local $\cO$-algebra surjection for $n\geq 1$,
\item $\M_n=\{M_r(H)_n\}_{r\geq 1,H\leq_\co G}$ is a $(G,R_n,(\varpi^r)_{r\geq 1}$-inverse system for $n\geq 0$ with $\M_0$ projective (where we define $R_0:=\cO$),
\item and $\alpha_n:\M_n/\ma_\infty\rightarrow \M_0$ is an isomorphism of $(G,S_\infty,(\varpi^r)_{r\geq 1})$-inverse systems for $n\geq 1$. 
\end{itemize}
This data is assumed to satisfying the following axioms:
\begin{enumerate}
\item[(1)] For all $r,n\geq 1$ the $S_\infty$-module $M_r(H)_n$ is free over a quotient $S_\infty/I_{r,n}$ (independent of $H$) such that $I_{r,n}\subseteq (\varpi^r,\ma_\infty)$ and 
\begin{equation}\label{eq:Irn condition}
\textrm{for any open ideal $\mb\subseteq S_\infty$, $I_{r,n}\subseteq (\varpi^r,\mb)$ for almost all $n\geq 1$.}
\end{equation}
\item[(2)] Whenever $H'\subseteq H$ is an open normal subgroup, $M_r(H')_n$ is a projective $S_\infty/I_{r,n}[H/H']$-module for all $r,n\geq 1$. Also $M_r(H')_0$ is a projective $(\cO/\varpi^r)[H/H']$-module. 
\end{enumerate}
\end{definition}

\noindent One can visualize a $G$-patching datum via the commutative diagram

\begin{equation}\label{eq:patching diagram}
\xymatrix{R_\infty \ar@{->>}[r]^{\varphi_n} & R_n\ar@{->>}[r] & R_n/\varpi^r &   \curvearrowright M_r(H)_n. \\ S_\infty \ar[ur]\ar@{-->}[u] & & &  }
\end{equation}
The dotted arrow exists for each $n\geq 1$ because $S_\infty$ is a power series ring.

\begin{definition}[cf. \cite{Man}]
Let $\cF\subset 2^{\N}$ be a nonprincipal ultrafilter. If $S$ is a ring and $\cQ=(Q_n)_{n\geq 1}$ is a sequence of left $S$-modules we let $\cU_\cF(\cQ)$ denote their ultraproduct over $\cF$. The properties of this construction that we will use are
\begin{enumerate}
\item $\cU_{\cF}(-)$ is an exact functor from the category of sequences of $S$-modules with $\cF$-morphisms, and 
\item if $S$ is finite and the cardinalities of the $Q_n$ are uniformly bounded, then $\cU_\cF(\cQ)\cong Q_i$ for $\cF$-many $i\geq 1$. 
\end{enumerate} 
\end{definition}

\begin{deflemma}\label{deflemma:patched module}
Fix a nonprincipal ultrafilter $\cF$. Given a $G$-patching datum as in Definition \ref{def:G patching datum} we define the associated \emph{patched module} to be
\begin{equation*}
M_\infty:= \varprojlim_{r,H}\,\, \cU_\cF(\{(M_r(H)_n\otimes_{S_\infty} S_\infty/\m_{S_\infty}^r\}_{n\geq 1}).
\end{equation*}
It is a finitely generated projective object in $\Mod^{\mathrm{pro\,aug}}_G(S_\infty)$ such that $M_\infty/\ma_\infty\cong (\M_0)_\infty$. It also has a compatible $R_\infty[G]$-module structure via a local $\cO$-algebra map $S_\infty\rightarrow R_\infty$. 
\end{deflemma}
\begin{proof}
For each $r,H$, the isomorphisms $\alpha_n$ imply that $M_r(H)_n\otimes_{S_\infty} S_\infty/\m_{S_\infty}^r$ is of uniformly bounded cardinality, so the ultraproduct is well behaved. Axioms (1) and (2) together with the properties of $\cU_\cF$ above imply that $(\cU_\cF(\{(M_r(H)_n\otimes_{S_\infty} S_\infty/\m_{S_\infty}^r\}_{n\geq 1})_{r\geq1,H\leq_\co G}$ is a projective $(G,S_\infty,(\m_{S_\infty}^r)_{r\geq 1})$-inverse system whose reduction mod $\ma_\infty$ is isomorphic to $\M_0$. The first claim now follows from Lemma \ref{lem:inverse limit}. 

For the second claim, at each finite level $n\geq 1$ the $S_\infty$ action on $\M_n$ factors through the local map $S_\infty\rightarrow R_\infty$ in \eqref{eq:patching diagram}. Since this map is independent of $r$ and $H$ the same is true of $M_\infty$. 
\end{proof}

\subsection{Patching}\label{sec:patching}
We give the precise conditions of the global Galois representations to which our main result applies.

\begin{definition}\label{def:suitable}
A triple $(F, \ov{r}, \chi)$ of a totally real field $F$, a continuous representation $\ov{r}: G_F \rightarrow \GSp_4(\F)$, and a Hecke character $\chi: \A_F^\times/F^\times \rightarrow \C^\times$ is called \emph{suitable of weight $(a_3,a_2,a_1,a_0)$} if it satisfies the followings:
\begin{itemize}
    \item [(\textbf{A1})] $F$ has even degree over $\Q$,
    \item [(\textbf{A2})] $\ov{r}$ is odd,
    \item [(\textbf{A3})] $\ov{r}$ is \emph{vast} and \emph{tidy} in the sense of \cite[\S 7.5]{BCGP} (this implies that $\rbar$ is absolutely irreducible),
    \item [(\textbf{A4})] $(\ov{r},\chi)$ is potentially diagonaliazably automorphic of some weight $\xi'$ and level unramified away from $p$,
    \item [(\textbf{A5})] $\ov{r}|_{G_{F_v}}$ is strongly 7-generic of weight $(a_3,a_2,a_1,a_0)$ as in Definition \ref{def:generic}(ii) for each $v|p$.
\end{itemize}
Let $\rhobar: G_{\Qp} \rightarrow \GSp_4(\F)$ be strongly 7-generic of weight $(a_3,a_2,a_1,a_0)$.  We say that a triple $(F, \ov{r}, \chi)$ is a \emph{suitable globalization of $\rhobar$} if it is suitable of weight $(a_3,a_2,a_1,a_0)$, and there is a place $w|p$ of $F$ such that $\ov{r}|_{G_{F_w}}\simeq \rhobar$.
\end{definition}

\begin{remark}\label{rk:assumptions-on-rbar}
In \textbf{(A4)}, we assume $(\rbar,\chi)$ to be potentially diagonalizably automorphic as opposed to being modular. This is used to show that $(\rbar,\chi)$ is modular of the ``obvious'' Serre weight $F(\mu)$ in Corollary \ref{lem:modularity of ordinary weight} by applying the automorphy lifting result for potentially diagonalizable lifts (\cite[4.2.1]{BLGGT}). The unramifiedness in \textbf{(A4)} is assumed for simplicity in the construction of patched modules. Assumption \textbf{(A5)} is not important for the construction of patching modules.
\end{remark}

For the following series of Lemmas, we let $K/\Qp$ be a finite unramified extension. We only need the case $K=\Qp$ for this article.

The first Lemma provides a lift of global residual character with prescribed local lifts. This must be well-known but we give the proof as we could not find it in the literature.

\begin{lemma}\label{lem:lifting-Galois-characters}
Let $F$ be a totally real field and $S$ be a finite set of places  of $F$ containing all places dividing $p$. Suppose we have a continuous character $\ov{\psi} : G_F \xrightarrow{} \F^\times$ and $\psi_v: G_{F_v} \xrightarrow{} \cO^\times$ lifting $\ov{\psi}|_{G_{F_v}}$ for all $v\in S$. We further assume that for all $v|p$, $\psi_v$ is de Rham with fixed Hodge--Tate weight $c$ and $\psi_v \epsilon^{-c}$ has finite image. Then there is a continous character $\psi: G_F \xrightarrow{} \cO^\times$ lifting $\ov{\theta}$ and $\psi|_{G_{F_v}} = \psi_v$ for all $v\in S$.
\end{lemma}  
\begin{proof}
Twisting by a power of cyclotomic character, we can assume $c=0$. Also, we use class field theory to consider $\ov{\psi}$ and $\psi_v$ as characters of $\A_F^\times /F^\times$ and $F_v^\times$ respectively. Following the argument in \cite[Lemma 4.1.1]{CHT}, we let $U\subset (\A_F^{S})^\times$ be the open compact subgroup such that $\ov{\psi}$ is trivial on $U$, and $\psi_v$ is trivial on $U\cap F^\times$. We define $\psi : U\prod_{v\in S}F_v^\times /(U\cap{F}^\times)$ by setting trivial on $U$ and $\psi_v$ on $F_v$.   Since the quotient $(\A_F^\times /F^\times)/(U\prod_{v\in S}F_v^\times /(U\cap{F}^\times))$ is finite abelian group, we can choose a sequence of subgroups
\begin{align*}
    U\prod_{v\in S}F_v^\times /(U\cap{F}^\times) = U_0 \subsetneq U_1 \subsetneq \cdots \subsetneq U_r = \A_F^\times /F^\times
\end{align*}
where each quotient $U_i/U_{i-1}$ is a finite cyclic group. We can inductively extend $\psi$ from $U_{i-1}$ to $U_i$ while ensuring that $\psi$ lifts $\ov{\psi}|_{U_i}$ at each step.
\end{proof}

The following Lemma shows that the modularity lifting for potentially diagonalizable representations for $\GSp_4$ follows from that for $\GL_4$.

\begin{lemma}\label{lem:pd-lift}
    Let $r:G_F \rightarrow \GSp_4(E)$ be a continuous representation. Let $\chi: \A_F^\times/F^\times \rightarrow \C^\times$ be a Hecke character such that $\chi_{p,\iota}\epsilon^{-3}=\nusim(r)$. Suppose that
    \begin{enumerate}
        \item for all $v|p$, $\std(r)|_{G_{F_v}}$ is potentially diagonalizable with distinct Hodge--Tate weights;
        \item $\rbar$ is vast;
        \item $(\rbar,\chi)$ is potentially diagonalizably automorphic for some $\chi: \A_F^\times/F^\times \rightarrow \C^\times$.
    \end{enumerate}
Then there exists a regular algebraic cuspidal automorphic representation $\pi$ of $\GSp_4(\A_F)$ such that $r\simeq r_{\pi,p,\iota}$. 
\end{lemma}
\begin{proof}
    By \cite[Theorem 7.4.1]{GeeTaibi}, $\std(\rbar)$ is potentially diagonalizably automorphic in the sense of \cite{BLGGT}. Thus, Theorem 4.2.1 in \emph{loc.~cit.}~implies that there exists a regular algebraic cuspidal automorphic representation $\Pi$ of $\GL_4(\A_F)$ such that a continuous representation $r_{\Pi,p,\iota}$ (see Theorem 2.1.1 in \emph{loc.~cit.}) is isomorphic to $\std(r)$. Such $\Pi$ is then $\chi$-self dual. Again by applying \cite[Theorem 7.4.1]{GeeTaibi} (also see Theorem 2.9.3 in \cite{BCGP}), this proves the existence of $\pi$ such that $r\simeq r_{\pi,p,\iota}$. 
\end{proof}

We now prove the existence of a suitable globalization of $\rhobar$ following the strategy of \cite[Appendix A]{EG}. 

\begin{lemma}\label{lem:globalizing-gal-rep}
Assume that $p>7$. Let $\rhobar : G_K \xrightarrow{} \GSp_4(\F)$ be a continuous representation. Suppose there is a lift $\rho : G_K \xrightarrow{} \GSp_4(E)$ such that $\std(\rho)$ is potentially diagonalizable.
Then there is a totally real field $F$ of even degree over $\Q$ such that $F_v \simeq K$ for all $v|p$ and a continuous representation $r : G_F \xrightarrow{} \GSp_4(E)$, with its reduction modulo $\varpi$ denoted by $\rbar$,  satisfying that
\begin{enumerate}
    \item $\ov{r}$ is odd,
    \item $\ov{r}$ is vast and tidy,
    \item $\ov{r}\vert_{G_{F_v}} \simeq \rhobar$ for all $v|p$,
    \item $\ov{r}$ is unramified at all finite places not dividing $p$,
    \item $r\vert_{G_{F_v}}$ and $\rho$ correspond to points contained in the same component of a potentially crystalline deformation ring  for all $v| p$ (in particular, $\std(r\vert_{G_{F_v}})$ is potentially diagonalizable), and
    \item there exists a regular algebraic cuspidal automorphic representation $\pi$ of $\GSp_4(\A_F)$ such that $r\simeq r_{\pi,p,\iota}$.
\end{enumerate}
\end{lemma}

\begin{proof}
In this proof, we let $S_{p,L}$ (resp.~$S_{\infty,L}$) be the set of places of a number field $L$ dividing $p$ (resp.~$\infty$).

Let $L$ be a totally real field of even degree over $\Q$ satisfying $L_v\simeq K$ for all $v | p$. We apply \cite[Proposition 3.2]{CalFM}. In the notation of \emph{loc.~cit.}, we set $G:= \GSp_4(\F)$, $E:=L$, $S:=S_{p,L}\cup S_{\infty,L}$, $F:=L(\zeta_p)$, $H_v :=\overline{L_v}^{\ker \rhobar}$ and $\phi_v$ given by $\rhobar$ with image $D_v$ and the isomorphism $L_v\simeq K$ for $v\in S_{p,L}$, and $c_v$ to be a Chevalley involution for all $v\in S_{\infty,L}$. Then there exists extensions of totally real fields $L''/L'/L$ such that
\begin{itemize}
    \item all places in $S$ split completely in $L'$,
    \item $L''/L$ is linearly disjoint from $L(\zeta_p)/L$,
    \item $L''/L$ is Galois and $\Gal(L''/L)\simeq \GSp_4(\F)$,
    \item for all places $w$ of $L'$ above $v\in S_{p,L}$, the isomorphism $\Gal(L''/L)\simeq \GSp_4(\F)$ induces  $\Gal(L''_w/L'_w)\simeq D_v$. 
\end{itemize}
We write $\rbar': G_{L'}\onto \Gal(L''/L) \xrightarrow{\sim} \GSp_4(\F)$. Then $\rbar'$ satisfies conditions (1) and (3) above. Moreover, such $\rbar'$ is tidy by \cite[Lemma 7.5.12]{BCGP}. By Lemma 7.5.15 of \emph{loc.~cit.}, 
$\rbar'$ is vast if
\begin{align*}
    \Sp_4(\F_p)\subset \ov{r}'(G_{L'(\zeta_p^N)})
\end{align*}
for all $N\ge 1$. For $N=1$, this is true because $L''/L$ is linearly disjoint from $L(\zeta_p)/L$. Then this holds for all $N\ge 1$ because $\Sp_4(\F_p)$ has no $p$-power order quotient.  Moreover, using the argument explained in the last two paragraph of the proof of \cite[Proposition A.2]{EG}, we can and do replace $L'$ by its totally real extension such that $\ov{r}'$ further satisfies (4).  

Using Lemma \ref{lem:lifting-Galois-characters}, choose an odd continuous character $\psi: G_{L'} \xrightarrow{} \cO^\times$ lifting  $\nu_{\simc}(\ov{r})$ and  $\psi|_{G_{L'_v}}\simeq \nu_{\simc} (\rho)$ for all $v\in S_{p,L'}$. Then \cite[Theorem 3.4]{PT} provides a lift $r' : G_{L'} \xrightarrow{} \GSp_4(E)$ of $\ov{r}'$ such that  $r'\vert_{G_{L'_v}}$ and $\rho_v$ lie on the same component of potentially crystalline deformation ring of $\rbar|_{G_{L'_v}}$ for all $v\in S_{p,L'}$. By \cite[Proposition A.6]{EG} (which is based on \cite[Theorem 4.5.1]{BLGGT}), there exists a finite extension $F/L'$ of totally real fields in which $p$ splits completely and a regular algebraic cuspidal automorphic representation $\Pi$ of $\GL_4(\A_F)$ such that $r_{\Pi,p,\iota}\simeq r'|_{G_{F}}$. Let $\chi:\A_F^\times/F^\times$ be a Hecke character such that $\chi_{p,\iota}\epsilon^{-3}\simeq \nusim(r')|_{G_F}$. Then $\Pi$ is $\chi$-self dual, and it descends to a regular algebraic cuspidal automorphic representation $\pi$ of $\GSp_4(\A_F)$. Thus, the representation $r:= r'|_{G_F}$ satisfies all desired properties.
\end{proof}

\begin{corollary}\label{cor:globalization}
Suppose $\rhobar : G_{\Qp} \xrightarrow{} \GSp_4(\F)$ is $7$-generic of  weight $(a_3,a_2,a_2,a_0)$. Then there is a suitable globalization $(F,\ov{r}, \chi_\cris)$ of $\rhobar$ such that $(\chi_{\cris,p,\iota}\otimes \epsilon^{-3})|_{G_{F_v}} = \epsilon^b \nr_{\xi_v}$.
\end{corollary}
\begin{proof}
By \cite[Lemma 7.6.7]{GG}, $\rhobar$ has an ordinary crystalline symplectic lift $\rho$ of weight $\ov{\mu}_{\rhobar,v}+\eta$. Twisting by unramified character, we can assume that $\nu_{\simc}(\rho) = \epsilon^{b+3}\nr_{\xi}$ with finite order $\nr_\xi$. Then $\std(\rho)$ is potentially diagonalizable by \cite[Lemma 1.4.3]{BLGGT}. The claim follows from Lemma \ref{lem:globalizing-gal-rep} and Remark \ref{rk:modularity-equiv} with $\chi_{\cris,p,\iota} = \nu_{\simc}(r) \otimes \epsilon^{3}$.
\end{proof}

\begin{corollary}\label{lem:modularity of ordinary weight}
    For any $(F,\ov{r},\chi)$ suitable of weight $(a_3,a_2,a_1,a_0)$, there is a Hecke character $\chi_\cris$ such that $(\chi_\cris\otimes \epsilon^{-3})|_{G_{F_v}} = \epsilon^b \nr_{\xi_v}$ for $v|p$, and $(\ov{
r},\chi_\cris)$ is modular of weight $F(\mu)$ and level unramified outside $p$ where $\mu_v=\mu_{\rhobar}$ for $v|p$.
\end{corollary}
\begin{proof}
     As in the proof of Lemma \ref{lem:globalizing-gal-rep}, we can apply \cite[Theorem 3.4]{PT} to construct a lift $r$ of $\ov{r}$ such that $r|_{G_{F_v}}$ is crystalline of weight $\ov{\mu}_{v}+\eta$ for each $v|p$ and unramified outside $p$, and $\chi_\cris = \nu_{\simc}(r) \otimes \epsilon^3$. By Lemma \ref{lem:pd-lift}, \ref{lem:stable-tempered}, and \cite[Theorem B]{Sor09}, there exists an automorphic representation $\pi$ of $\cG(\A_F)$ such that $r\simeq r_{\pi,p,\iota}$. Moreover, such $\pi$ arise from $\mathcal{A}_{\chi_\cris, \iota(\mu)}$.  Since ${\mu}_{v}$ is a lowest alcove weight the claim follows from Lemma \ref{lem:Weyl decomposition}.
\end{proof}

From now on, we let $(F,\ov{r},\chi_\cris)$ be suitable of weight $(a_3,a_2,a_1,a_0)$ as in Corollary \ref{lem:modularity of ordinary weight} and $\cG$ be as in \S \ref{sec:automorphic representations}.  We write  $\psi_\cris = \chi_\cris \epsilon^{-3}: G_F \xrightarrow{} \cO^\times$ and $\psi_{v,\cris}:= \psi|_{G_{F_v}}=\epsilon^{b}\nr_{\xi_v}$ for all $v|p$.  We also write $\psi_{v,\pcris}=\tomega^{b-3}\epsilon^3\nr_{\xi_v}$ for $v|p$. By Lemma \ref{lem:lifting-Galois-characters}, there exists $\psi_\pcris: G_F \xrightarrow{} \cO^\times$ lifting $\nu_{\simc}(\rbar)$ such that $(\psi_\pcris)|_{G_{F_v}} \simeq \psi_{v,\pcris}$ for all $v|p$. Finally, we set $\chi = \chi_\pcris := \psi_\pcris \otimes \epsilon^3$.
 Let $S_p$ be the set of places of $F$ dividing $p$. Fix a choice of preferred place $w|p$.  Since $\ov{r}$ is vast and tidy, we can find a finite place $v_0\notin S_p$ such that $q_{v_0}\not\equiv 1\mod p$, no two eigenvalues of $\rbar(\Frob_{v_0})$ have ratio $q_{v_0}$, and the residue characteristic of $v_0$ is $>5$ (\cite[\S7.7]{BCGP}). 
 
 Let $S$ be a finite set of finite places of $F$ containing $S_p\cup \{v_0\}$. For $v\neq v_0$, we set $U_v:=\cG(\cO_{F_v})$, and we also define $U_{v_0}=\Iw_1(v_0)$. Then set $U^p:=\prod_{v\nmid p} U_v$ and $U_p:= \prod_{v|p} U_v$.  Corollary \ref{lem:modularity of ordinary weight} ensures that $S_{\chi_\cris}(U^pU_p,F(\mu))[\m_{\rbar,{\chi_\cris}}^{S}]\neq 0$. The choice of $v_0$ ensures that $U^p U_p$ is sufficiently small.

 Let $T$ be a subset of $S$. For each $v\in S$, if $\cD_v$ (resp. $\cD_v^\cris$)  denotes a deformation problem on symplectic lifts of $\rbar|_{G_{F_v}}$  and $R_v$ (resp. $R_v^\cris$) is the corresponding quotient of the universal symplectic lifting ring, the tuples
\begin{equation*}
\cS=(S,T,\{\cD_v\}_{v\in S},\psi_\pcris), \  \cS_\cris=(S,T,\{\cD_v^\cris\}_{v\in S},\psi_\cris)
\end{equation*}
are global deformation problems and we write $\cD_{\cS}^T$ (resp. $\cD_{\cS_\cris}^T$) for the the functor of $T$-framed symplectic deformations of type $\cS$ (resp $\cS_\cris)$. To simplify the notations, we write $\star \in \{\emptyset,\cris\}$. Since $\rbar$ is absolutely irreducible, $\cD_{\cS_\star}^{T}$ is representable by a ring $R_{\cS_\star}^T$. When $T=\emptyset$, we omit the superscript $T$. The natural map $R_{\cS_\star}\rightarrow R_{\cS_\star}^T$ is formally smooth of relative dimension $11|T|-1$. Write $\cT:=\cO[[y_1,\ldots,y_{11|T|-1}]]$. The choice of a universal lift $r_{\cS_\star}:G_F\rightarrow \GSp_4(R_{\cS_\star})$ determines a canonical isomorphism $R_{\cS_\star}^T\cong R_{\cS_\star}\hat{\otimes}_\cO \cT$. 
Define $R^{T,\loc}_{\cS_\star}=\widehat{\bigotimes}_{v\in T} R_v^\star$. There is a natural map $R^{T,\loc}_{\cS_\star}\rightarrow R_{\cS_\star}^T$. 

\begin{definition}
A \emph{Taylor-Wiles datum} is a tuple $(Q,(\bar{\alpha}_{v,1},\ldots,\bar{\alpha}_{v,4})_{v\in Q})$ where $Q$ is a finite set of finite places disjoint from $S$ such that $q_v\equiv 1\mod p$ for all $v\in Q$ and an ordering of the eigenvalues of $\rbar(\Frob_v)$ which are assumed to be $\F$-rational and pairwise distinct: $\bar{\alpha}_{v,1},\bar{\alpha}_{v,2}$ and $\bar{\alpha}_{v,3}=\nu_{\simc}(\rbar)(\Frob_v)\bar{\alpha}_{v,2}^{-1}$ and $\bar{\alpha}_{v,4}=\nu_{\simc}(\rbar)(\Frob_v)\bar{\alpha}_{v,1}^{-1}$.
\end{definition}

Given a Taylor-Wiles datum we define the augmented deformation problem $\cS_{Q}=(S\cup Q,\{\cD_v\}_{v\in S}\cup \{\cD_v^\square\}_{v\in Q},\psi_\pcris)$ and $\cS_{\cris,Q}$ similarly. For $v\in Q$ let $\Delta_v=k_v\x(p)^2$ where $k_v\x(p)$ is the maximal $p$-power quotient of $k_v\x$. Set $\Delta_Q:=\prod_{v\in Q}\Delta_v$. There is a canonical map $\cO[\Delta_Q]\rightarrow R_{\cS_{\star,Q}}^T$ for any subset $T\subseteq S$ such that the natural surjection $R_{\cS_{\star,Q}}^T\rightarrow R_{\cS_\star}^T$ has kernel $\ma_QR_{\cS_{\star,Q}}^T$. Here $\ma_Q\leq\cO[\Delta_Q]$ denotes the augmentation ideal. 

\begin{proposition}\label{prop:TW primes}
Let $q\geq h^1(F_S/F,\ad(\rbar)(1))$. For every $n\geq 1$ there exists a choice of Taylor-Wiles datum $Q_n$ such that
\begin{enumerate}
\item $|Q_n|=q$;
\item $q_v\equiv 1\mod p^n$ for each $v\in Q_n$;
\item $R_{\cS_{\star,Q_n}}^S$ is a quotient of a power series ring over $R_{\cS_\star}^{S,\loc}=R_{\cS_{\star,Q_n}}^{S,\loc}$ in 
\begin{equation*}
g:= 2q-4|F:\Q|+|S|-1
\end{equation*}
variables.
\end{enumerate}
\end{proposition}
\begin{proof}
This follows immediately from Corollary 7.6.3 of \cite{BCGP} because $\rbar$ is odd with large image. The assumption that $\psi=\epsilon^{-1}$ in this reference is not necessary for the proof.
\end{proof}

\begin{remark}
Note that we can and do choose the same Taylor--Wiles datum $Q_n$ for both $\star = \emptyset$ and $\star =\cris$.
\end{remark}

From now on, we take $S= S_p \cup \{v_0\}$,  $\cD_{v}=\cD_{v}^{\square,\psi_{v,\pcris}}$ (resp. $\cD_{v}^\cris=\cD_{v}^{\square,\psi_{v,\cris}}$) for $v\in S$.  
Write $R_n^\star:=R_{\cS_{\star,\cQ_n}}^S$. 
Let $R_\infty^\star:=R_{\cS_\star}^{S,\loc}\hat{\otimes}[[x_1,\ldots,x_g]]$. By Proposition \ref{prop:TW primes}(3) we may choose surjections $\varphi_n^\star:R^\star_\infty\onto R^\star_n$ for each $n\geq 1$. Define $S_\infty:=\cO[[\Z_p^{2q}]]\hat\otimes_{\cO}\cT$. For each $n\geq 1$, choose a surjection $S_\infty\onto \cO[\Delta_{Q_n}]\hat\otimes_\cO\cT$. This makes $R_n^\star$ into an $S_\infty$ algebra. Note that if we let $I_{r,n}\leq S_\infty$ be the ideal generated by the kernel of this map together with $\varpi^r$ then the collection $\{I_{r,n}\}$ obeys \eqref{eq:Irn condition} by Proposition \ref{prop:TW primes}(2). 

\begin{lemma}\label{lem:dimS=dimR}
Suppose that  $\sigma = \sigma(\tau)$ for a regular principal series $\tau$ and
    $\sigma^\cris = V(\mu_{\rhobar})_{\Qp} \otimes E|_K$,
then we have $\dim(S_\infty)=\dim(R^\star_\infty(\sigma^\star))$ for $\star \in \{\emptyset,\cris\}$.
\end{lemma}
\begin{proof}
This follows from Lemma \ref{lem:dimension of p crys def ring} and Proposition \ref{prop:fs crystalline deformation rings}.
\end{proof}

 For each $n\geq 1$ we now define $U_1^p(Q_n)$ to be $\prod_{v\nmid p} U_1(Q_n)_v$, where $U_1(Q_n)_v=U_v$ if $v\notin Q_n\cup S_p$ and $U_1(Q_n)_v=\Iw_1(v)$ if $v\in Q_n$. Let $G:=\cG(F_w)$, $G_p:= \prod_{v|p}\cG(F_v)$, and $K_p = \prod_{v|p}\cG(\cO_{F_v})$.  
 For $r,n\geq 1$ and $H\leq_\co G_p$,  we define
\begin{equation*}
M_r^\star(H)_n = [S_{\chi_{\star}}(H\cdot U_1^p(Q_n),\cO/\varpi^r)_{\m_{\rbar,\chi_\star}^{S\cup Q_n,r},\m_{Q_n}^\star}]\dual \otimes_{R_{\cS_{\star,Q_n}}}R_n^\star.
\end{equation*}
where $\m_{\rbar,\chi_\star}^{S\cup Q_n,r}$ is the image of $\m_{\rbar,\chi_\star}^{S\cup Q_n}$ inside $\T_{\chi_\star}^{S\cup Q_n}(H\cdot U_1^p(Q_n),\cO/\varpi^r)$ and $\m_{Q_n}^\star$ is the maximal ideal of $\bigotimes_{v\in Q_n} \cO[T(F_v)/T(\cO_{F_v})_1]$ in $\End_{\cO}(S_{\chi_\star}(H\cdot U_1^p(Q_n),\cO/\varpi^r))$ containing $\bigotimes_{v\in Q_n} \cO[T(\cO_{F_v})/T(\cO_{F_v})_1]$ such that $\beta_0(\varpi_v)\equiv \bar{\chi}_v(\varpi_v)$, $\beta_1(\varpi_v)\equiv \bar{\alpha}_{v,1}$ and $\beta_2(\varpi_v)\equiv \bar{\alpha}_{v,1}\bar{\alpha}_{v,2}\mod \m_{Q_n}$ (cf. \cite[\S2.4.29]{BCGP}).
 Note that $S_{\chi_\star}(H\cdot U_1^p(Q_n),\cO/\varpi^r)_{\m_{\rbar,\chi_\star}^{S\cup Q_n,r,\star},\m_{Q_n}^\star}$ is a $R^S_{\cS_{\star,Q_n}}$-module by the natural map 
\begin{equation}\label{eq:R->T}
R_{\cS_{\star,Q_n}}\rightarrow \T_{\chi_\star}^{S\cup Q_n}(H\cdot U_1^p(Q_n),\cO/\varpi^r)
\end{equation}
 which exists by Proposition \ref{prop:hecke valued galois representations}, and the action of $\cO[\Delta_{Q_n}]$ on $M_r^\star(H)_n$ via $\cO[\Delta_{Q_n}]\rightarrow R_{\cS_{\star,Q_n}}^S$ agrees with the canonical action via the surjection $U^p/U_1^p(Q_n)\cong \Delta_{Q_n}$. 
 Note that \eqref{eq:R->T} is surjective by Proposition \ref{prop:hecke valued galois representations}(ii) (since $\chi_\star$ is fixed). 
Moreover we define
\begin{equation*}
M_r^\star(H)_0 = [S_{\chi_\star}(H\cdot U^p, \cO/\varpi^r)_{\m_{\rbar}^{S,r,\star}}]\dual.
\end{equation*}
It follows from \cite[\S2.4.29]{BCGP} that there is an isomorphism 
\begin{equation*}
\alpha_n^\star : M^\star_r(H)_n/\ma_\infty \cong M^\star_r(H)_0
\end{equation*}
compatible among different $r\geq 1$ and $H\leq_\co G_p$. 

\begin{lemma}\label{lem:patching-datum}
The data 
\begin{equation*}
(S_\infty,R^\star_\infty,(R^\star_n,\varphi^\star_n,\{M^\star_r(H)_n\}_{r\geq 1,H\leq_\co G},\alpha^\star_n)_{n\geq 1},\{M^\star_r(H)_0\}_{r\geq 1,H\leq_\co G})
\end{equation*}
 defined in the paragraphs above forms a $G_p$-patching datum for both $\star\in \{\emptyset,\cris\}$.
\end{lemma}
\begin{proof}
Axiom (1) follows from the remarks above and axiom (2) holds because $H\cdot U^p$ is sufficiently small.
\end{proof}

Hence we may form the patched module $M^\star_\infty$ corresponding to this $G$-patching datum in Lemma \ref{deflemma:patched module}.

We now take $\cJ = S_p$ and choose an isomorphism $K^\cJ \simeq K_p$. The following is the main result of this section.
\begin{lemma}\label{lem:patched module axioms}
The pair $(\Minf, \Minf^\cris)$ is a congruent patched module pair for $(\rbar|_{G_{F_v}}, \psi_{v,\pcris}, \psi_{v,\cris})_{v\in S_p}$ satisfying \textbf{(PM1)}--\textbf{(PM6)}, and we have $d=|W|=8$ in \textbf{(PM3)}.
\end{lemma}
\begin{proof}
\textbf{(PM1)} and \textbf{(PM2)} follow already from Lemma \ref{deflemma:patched module}. Axiom \textbf{(PM3)} follows from Lemma \ref{lem:Lpacket}, \ref{lem:K-type} (which can be seen as the inertial local Langlands for principal series inertial type),  Proposition \ref{prop:hecke valued galois representations}(iii) and Lemma \ref{lem:dimS=dimR} by the standard commutative algebra argument; see the proof of Lemma 4.18 of \cite{CEG+}. We may take $d=|W|$ because this is the dimension of the space of pro-$p$ Iwahori invariants in a principals series representation. Axiom \textbf{(PM4)} follows also from Proposition \ref{prop:hecke valued galois representations}(iii) similar to the proof of Lemma 4.17 of \cite{CEG+}.
\textbf{(PM5)} follows from the fact that $M_r^\star(H)_n$ in Lemma \ref{lem:patching-datum} for $\star \in \{\emptyset,\cris\}$ are congruent modulo $\varpi$ together with Lemma \ref{lem:inverse limit}.
Finally, \textbf{(PM6)} follows from Corollary \ref{lem:modularity of ordinary weight}.
\end{proof}

Fix a place $w$ of $F$ dividing $p$.  We write $U^w = \prod_{v\neq w} U_v$. Recall that 
\begin{align*}
    M_{\infty,w} &= \Hom_{\cO[\prod_{v\in S_p\backslash\{w\}}\cG(\cO_{F_v})]} (M_\infty, (\sigma(\tau_0)^{S_p\backslash\{w\}})^d)^d \\
    R_{\infty,w} &= R_{\ov{r}|_{G_{F_w}}}^{\square,\psi_{w,\pcris}}[[x_1,\ldots, x_h]]
\end{align*}
and $\m_{\infty,w}$ the maximal ideal of $R_{\infty,w}$. By Theorem \ref{thm:abstract loc glob}(1), we know that $M_{\infty,w}$ is non-zero. 
\begin{lemma}\label{lem:mod m}
We have
\begin{equation*}
(M_{\infty,w}/\m_{\infty,w})\dual \cong S_{\chi,0,\sigma(\tau_0)}(U^w,\F)[\m_{\rbar,\chi}^S].
\end{equation*}
\end{lemma}
\begin{proof}
By Lemma \ref{deflemma:patched module} there is a natural inclusion 
$
(M_{\infty,w}/\m_{\infty,w})\dual\hookrightarrow S_{\chi,0,\sigma(\tau_0)}(U^w,\F)[\m_{\rbar,\chi}^S].
$
 It is surjective because we have arranged the map \eqref{eq:R->T} to be surjective.
\end{proof}

\begin{remark}
The $R_{\infty,w}$-module $M_{\infty,w}$ is the patched module constructed by \cite{CEG+}, adapted to the group $\GSp_4$, using the space of automorphic forms with infinite level at $w$. In contrast, the construction of the $R_\infty$-module $M_\infty$ uses the space of mod $p$ automorphic forms with infinite level at all places dividing $p$.  We find it is more convenient to make a comparison between $M_\infty$ and $M_\infty^\cris$ rather than between $M_{\infty,w}$ and its crystalline version to prove Theorem \ref{thm:abstract loc glob}(1).   However, it is plausible that one can obtain the same result using $M_{\infty,w}$ and its crystalline version.  
\end{remark}

\subsection{Mod $p$ local-global compatibility} 

We can now prove the main result. 
\begin{theorem}\label{thm:main result}
Let $(F,\ov{r},\chi)$ be suitable of weight $(a_3,a_2,a_1,a_0)$ with $\chi=\chi_\pcris$ as in \S \ref{sec:patching}, and $\cG$ be as in \S \ref{sec:automorphic representations}. Let $w$ be a place of $F$ dividing $p$. Then the Fontaine-Laffaille parameters of $\ov{r}|_{G_{F_w}}$ can be recovered from the admissible $\F[\GSp_4(F_w)]$-module $S_{\chi, 0,\sigma(\tau_0)}(U^w,\F)[\m_{\rbar,\chi}^S]$ by the recipe described in Remark \ref{rk:group algebra operators}.

\end{theorem}
\begin{proof}
This follows from Lemmas \ref{lem:patched module axioms} and \ref{lem:mod m} and Theorem \ref{thm:abstract loc glob}.
\end{proof}

\newpage
\appendix
\numberwithin{table}{section}
\numberwithin{figure}{section}
\section{Tables and Figures}\label{appendix:tables}
\vspace{35pt}
\begin{table}[h]
\centering
\begin{tabular}{|c|c|c|c|c|}
\hline
$\lambda\in\,$Alcove & $\pair{\lambda+\rho',\alpha_0\dual}$ & $\pair{\lambda+\rho',\alpha_1\dual}$ & $\pair{\lambda+\rho',(2\alpha_0+\alpha_1)\dual}$ & $\pair{\lambda+\rho',(\alpha_0+\alpha_1)\dual}$ \\ 
\hline
$C_0$ & $(0,p)$ & $(0,p)$ & $(0,p)$ & $(0,p)$\\
\hline
$C_1$& $(0,p)$ & $(0,p)$ & $(0,p)$ & $(p,2p)$ \\
\hline
$C_2$ & $(0,p)$ & $(0,p)$ & $(p,2p)$ & $(p,2p)$ \\
\hline
$C_3$ & $(0,p)$ & $(0,p)$ & $(p,2p)$ & $(2p,3p)$ \\
\hline
$D_0$ & $(0,p)$ & $(p,2p)$ & $(p,2p)$ & $(2p,3p)$ \\
\hline
$D_1$ & $(0,p)$ & $(p,2p)$ & $(p,2p)$ & $(3p,4p)$ \\
\hline 
$E_0$ & $(p,2p)$ & $(0,p)$ & $(p,2p)$ & $(p,2p)$ \\
\hline
$E_1$ & $(p,2p)$ & $(0,p)$ & $(p,2p)$ & $(2p,3p)$ \\
\hline
$E_2$ & $(p,2p)$ & $(0,p)$ & $(2p,3p)$ & $(2p,3p)$ \\
\hline 
$E_3$ & $(2p,3p)$ & $(0,p)$ & $(2p,3p)$ & $(2p,3p)$ \\
\hline
\end{tabular}
\caption{List of some $p$-alcoves for $\GSp_4$.}\label{table:alcoves}
\end{table}

\begin{figure}[b!]
\centering
\begin{tikzpicture}[scale=1.385]
\path[draw, line width = 2pt] (0,0) -- (4.05,4.05); 
\path[draw] (1,0) -- (4,3);
\path[draw] (2,0) -- (4,2);
\path[draw] (3,0) -- (4,1);
\path[draw, line width = 2pt] (0,0) -- (4.05,0); 
\path[draw] (1,1) -- (4,1);
\path[draw] (2,2) -- (4,2);
\path[draw] (2,2) -- (4,2);
\path[draw] (3,3) -- (4,3);
\path[draw, style=dashed] (-1,1) -- (1,-1);
\path[draw] (.5,.5) -- (1,0);
\path[draw] (1,1) -- (2,0);
\path[draw] (1.5,1.5) -- (3,0);
\path[draw] (2,2) -- (4,0);
\path[draw] (2.5,2.5) -- (4,1);
\path[draw] (3,3) -- (4,2);
\path[draw] (3.5,3.5) -- (4,3);
\path[draw, style=dashed] (0,1) -- (0,-1);
\path[draw] (1,1) -- (1,0);
\path[draw] (2,2) -- (2,0);
\path[draw] (3,3) -- (3,0);
\path[draw] (4,4) -- (4,0);
\node [below left] at (0,0) {$-\rho'$};

\node [right] at (4.05,4.05) {$s_{\alpha_0}$};

\node [right] at (4,0) {$s_{\alpha_1}$};

\node[above] at (0,1) {$s_{2\alpha_0+\alpha_1}$};

\node [right] at (1,-1) {$s_{\alpha_0+\alpha_1}$};

\node at (.5,.15) {$C_0$};
\node at (.85,.5) {$C_1$};
\node at (1.2,.5) {$C_2$};
\node at (1.5,.83) {$C_3$};
\node at (1.5,1.15) {$D_0$};
\node at (1.83,1.5) {$D_1$};
\node at (1.5,.15) {$E_0$};
\node at (1.83,.5) {$E_1$};
\node at (2.17,.5) {$E_2$};
\node at (2.5,.15) {$E_3$};
\end{tikzpicture}
\caption{A picture of the dominant Weyl chamber with labelled alcoves and reflection axes under the dot action.}\label{fig:weyl chamber}
\end{figure}
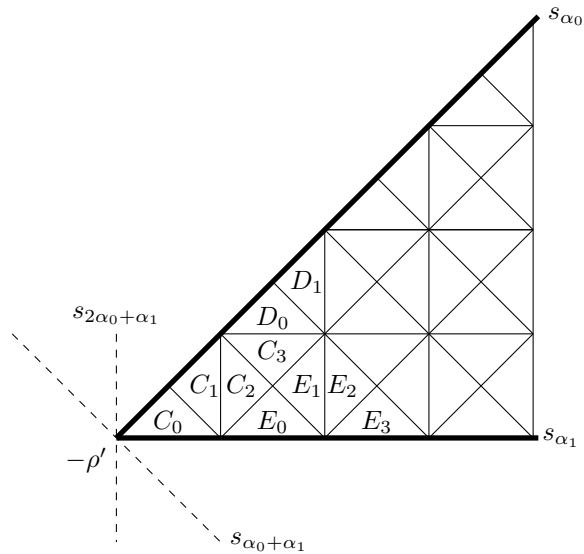

\newpage

\begin{table}[h]
\centering
\begin{tabular}{|c|c|}
\hline
$\lambda\in\,$Alcove & $\chi(\lambda)=$ \\ 
\hline
$C_0$ & $\chi_p(\lambda)$ \\
\hline
$C_1$ &  $\chi_p(\lambda)+\chi_p(C_0)$ \\
\hline
$C_2$ &  $\chi_p(\lambda)+\chi_p(C_1)$ \\
\hline
$C_3$ &  $\chi_p(\lambda)+\chi_p(C_2)$ \\
\hline
$D_0$ &  $\chi_p(\lambda)+\chi_p(C_3)$  \\
\hline
$D_1$ &  $\chi_p(\lambda)+ \chi_p(D_0)+ \chi_p(E_1)+\chi_p(C_3)$ \\
\hline 
$E_0$ & $\chi_p(\lambda)+\chi_p(C_2)$ \\
\hline
$E_1$ & $\chi_p(\lambda)+\chi_p(E_0)+\chi_p(C_3)+\chi_p(C_2)+\chi_p(C_1)$ \\
\hline
$E_2$ & $\chi_p(\lambda)+\chi_p(E_1)+\chi_p(D_0)+\chi_p(C_3)+\chi_p(C_2)+\chi_p(C_1)+\chi_p(C_0)$ \\
\hline 
$E_3$ & $\chi_p(\lambda)+\chi_p(E_2)+\chi_p(C_2)$ \\
\hline
\end{tabular}
\caption{This table gives the Jordan-H\"older factors of Weyl modules for $\GSp_4$. In it, we write $\chi_p(A)$ instead of $\chi_p(\lambda_A)$, where $\lambda_A$ is the unique weight linked to $\lambda$ in alcove $A$.}\label{table:Weyl decomposition}
\end{table}
 
\newpage
\begin{table}[t!]
\begin{center}
\begin{tabular}{|c|c|c|}
\Xhline{3\arrayrulewidth}
\multirow{4}{*}{$\lambda\in C_0$} & $\gr^0$ & $F(x,y)_0+F'(y+p-1,x)_2$ \\ \cline{2-3}
& $\gr^1$ & \makecell{$F(x-2,y)_0+F(x-1,y-1)_0+F'(p-3-y,x)_1$\\$+F'(y+p-1,p-1-x)_2+F(y+p-2,p-2-x)_2+F(x+p-1,p-1-y)_3$} \\ \cline{2-3}
& $ \gr^2$ & \makecell{$F'(p-2-x,y-1)_0+F(x-2,y-2)_0+F'(p-1-x,y)_0$\\$+F'(p-3-x,y-2)_0+F'(p-3-x,y)_0+F'(p-2-y,x-1)_1$\\$+F(p-3-y,p-1-x)_1+F'(p-1-y,x)_1+F(p-2-y,p-2-x)_1$\\$+F'(y+p-2,x-1)_2$} \\ \cline{2-3}
& $\gr^3$ & $F(p-1-y,p-1-x)_1+F(2p-2-x,p-1-y)_3$ \\ \Xhline{3\arrayrulewidth}

\multirow{4}{*}{$\lambda\in C_1$} & $\gr^0$ & $F(x,y)_1+F'(y+p-1,x)_3$ \\ \cline{2-3}
& $\gr^1$ & \makecell{$F'(y-2,p-1-x)_0+F'(y,p-1-x)_0+F(p-3-y,p-3-x)_0$\\$+F'(y-1,p-2-x)_0 +F'(y-2,p-3-x)_0+F(x-2,y)_1$\\$+F(x-1,y-1)_1+F'(x-1,p-2-y)_1+F'(x,p-1-y)_1$\\$+ F'(2p-3-x,p-2-y)_2$} \\ \cline{2-3}
& $ \gr^2$ & \makecell{$F(p-3-y,p-1-x)_0+F(p-2-y,p-2-x)_0+F'(x-2,p-1-y)_1$\\$+F(2p-2-x,y)_2+F(2p-3-x,y-1)_2+F(2p-2-y,x)_3$} \\ \cline{2-3}
& $\gr^3$ & $F(p-1-y,p-1-x)_0+F'(2p-2-x,p-1-y)_2$ \\ \Xhline{3\arrayrulewidth}

\multirow{4}{*}{$\lambda\in C_2$} & $\gr^0$ & $F'(y,x-p+1)_0+F(x,y)_2$ \\ \cline{2-3}
& $\gr^1$ & \makecell{$F'(y-1,x-p)_0+F'(y-2,x-p+1)_0+F(2p-4-x,y)_1$\\$+F'(x-1,p-2-y)_2+F'(x,p-1-y)_2+F'(y+p-1,2p-2-x)_3$} \\ \cline{2-3}
& $ \gr^2$ & \makecell{$F'(y-2,x-p-1)_0+F(p-2-y,x-p)_0+F(p-3-y,x-p-1)_0$\\$+F(p-3-y,x-p+1)_0+F(p-1-y,x-p-1)_0+F(2p-3-x,y-1)_1$\\$+F(2p-2-x,y)_1+F'(2p-3-x,p-2-y)_1+F'(2p-4-x,p-1-y)_1$\\$+F(x-1,y-1)_2$} \\ \cline{2-3}
& $\gr^3$ & $F'(2p-2-x,p-1-y)_1+F(2p-2-y,2p-2-x)_3$ \\ \Xhline{3\arrayrulewidth}

\multirow{4}{*}{$\lambda\in C_3$} & $\gr^0$ & $F'(y,x-p+1)_1+F(x,y)_3$ \\ \cline{2-3}
& $\gr^1$ & \makecell{$F(x-p-1,p-1-y)_0+F(x-p-1,p-3-y)_0+F(x-p,p-2-y)_0$\\$+F'(2p-4-x,p-3-y)_0+F(x-p+1,p-1-y)_0+F'(y-1,x-p)_1$\\$+F(y,2p-2-x)_1+F(y-1,2p-3-x)_1+F'(y-2,x-p+1)_1$\\$+F(2p-2-y,2p-2-x)_2+F(2p-3-y,2p-3-x)_2$} \\ \cline{2-3}
& $ \gr^2$ & \makecell{$F'(2p-3-x,p-2-y)_0+F'(2p-4-x,p-1-y)_0+F(y-2,2p-2-x)_1$\\$+ F'(2p-3-y,x-p)_2+F'(2p-2-y,x-p+1)_2+F'(3p-3-x,y)_3$} \\ \cline{2-3}
& $\gr^3$ & $F'(2p-2-x,p-1-y)_0$ \\ \Xhline{3\arrayrulewidth}

\end{tabular}
\caption{For $\lambda=(x,y;z)$ 7-deep inside a $p$-restricted alcove, this table gives the characters of $\gr^i:=\overline{M}(\lambda)_\F(s_1s_0s_1,i)/\overline{M}(\lambda))_\F(s_1s_0s_1,i+1)$ (we have $\gr^i=0$ for $i\geq 4$ by Theorem \ref{thm:filtration length}). We use the shorthand $F(a,b)_{i}:=F(a,b;z)$ and $F'(a,b)_{i}:=F(a,b;z+p-1)$, the subscript $i$ referring to the alcove $C_i$ in which the Serre weight lies.} \label{table:Jantzen filtration pieces}
\end{center}
\end{table}

\bibliography{compatibilitybibliography}{}
\bibliographystyle{amsalpha}

\noindent \small \textsc{Department of Mathematics, Northwestern University, 2033 Sheridan Road,
Evanston, IL 60208, USA} \\
\textit{Email address:} \href{mailto:john.enns@northwestern.edu}{\texttt{john.enns@northwestern.edu}} 

\vspace{1em}

\noindent \small \textsc{Department of Mathematics,
University of Toronto,
40 St. George Street,
Toronto, ON M5S 2E4, Canada} \\
\textit{Email address:} \href{mailto:heejong.lee@mail.utoronto.ca}{\texttt{heejong.lee@mail.utoronto.ca}} 

\end{document}

%% file: newcommands.tex
\newcommand{\Minf}{M_{\infty}}

\newcommand{\WD}{\operatorname{WD}}

\newcommand{\Fpbar}{\overline{\F}_p}

\newcommand{\Fp}{\F_p}

\newcommand{\rbar}{\bar{r}}
\newcommand{\rhobar}{\bar{\rho}}

\newcommand{\nr}{\mathrm{nr}}

\newcommand{\Rep}{\mathrm{Rep}}

\newcommand{\Diag}{\mathrm{Diag}}
\newcommand{\id}{\mathrm{id}}

\newcommand{\Qpf}{\mathbb{Q}_{p^f}}

\newcommand{\Iw}{\mathrm{Iw}}

\newcommand{\Minfty}{M_{\infty}}

\newcommand{\nusim}{\nu_{\mathrm{sim}}}

\newcommand{\recGT}{\mathrm{rec}_{\mathrm{GT}}}
\newcommand{\recGTp}{\mathrm{rec}_{\mathrm{GT},p}}
\newcommand{\CO}{\mathcal{C}_{\mathcal{O}}}
\newcommand{\COhat}{\widehat{\mathcal{C}}_{\mathcal{O}}}
\newcommand{\cts}{\mathrm{cts}}
\newcommand{\ev}{\mathrm{ev}}
\newcommand{\loc}{\mathrm{loc}}
\newcommand{\ad}{\mathrm{ad}}

\newcommand{\can}{\mathrm{can}}

\newcommand{\GKinfty}{G_{K_{\infty}}}

\newcommand{\overbar}[1]{\mkern 1.5mu\overline{\mkern-1.5mu#1\mkern-1.5mu}\mkern 1.5mu}
\newcommand{\Mbar}{\overbar{M}}

\newcommand{\LG}{{}^LG}

\newcommand{\co}{\mathrm{c.o.}}

\newcommand{\T}{\mathbb{T}}
\newcommand{\A}{\mathbb A}
\newcommand{\C}{\mathbb C}
\newcommand{\F}{\mathbb F}

\newcommand{\Z}{\mathbb Z}
\newcommand{\N}{\mathbb{N}}

\newcommand{\m}{\mathfrak{m}}

\newcommand{\ma}{\mathfrak{a}}
\newcommand{\mj}{\mathfrak{j}}
\newcommand{\M}{\mathfrak{M}}

\newcommand{\p}{\mathfrak{p}}

\newcommand{\mb}{\mathfrak{b}}

\newcommand{\cA}{\mathcal{A}}

\renewcommand{\cD}{\mathcal{D}}
\newcommand{\cE}{\mathcal{E}}
\newcommand{\cF}{\mathcal{F}}
\newcommand{\cG}{\mathcal{G}}
\renewcommand{\cH}{\mathcal{H}}
\newcommand{\cI}{\mathcal{I}}
\newcommand{\cJ}{\mathcal{J}}

\newcommand{\cM}{\mathcal{M}}

\newcommand{\cO}{\mathcal{O}}

\newcommand{\cQ}{\mathcal{Q}}

\newcommand{\cS}{\mathcal{S}}
\newcommand{\cT}{\mathcal{T}}
\newcommand{\cU}{\mathcal{U}}

\newcommand{\dd}{\mathrm{dd}}

\newcommand{\tomega}{\widetilde{\omega}}
\newcommand{\tw}{\widetilde{w}}
\newcommand{\tW}{\widetilde{W}}

\newcommand{\Q}{\mathbb Q}
\newcommand{\Qp}{\Q_p}
\newcommand{\Zp}{\Z_p}

\newcommand{\Qpbar}{\overline\Q_p}

\newcommand{\Qlbar}{\overline{\Q}_{\ell}}

\newcommand{\onto}{\twoheadrightarrow}

\newcommand{\ab}{^{\mathrm{ab}}}
\newcommand{\dual}{^\vee}
\newcommand{\x}{^{\times}}

\newcommand{\abs}[1]{\left\vert#1\right\vert}

\newcommand{\pair}[1]{\langle#1\rangle} 
\newcommand{\und}{\underline}

\newcommand{\GL}{\operatorname{GL}}
\newcommand{\SL}{\operatorname{SL}}
\newcommand{\GSp}{\operatorname{GSp}}
\newcommand{\Sp}{\operatorname{Sp}}

\newcommand{\Gm}{\mathbb{G}_m}

\newcommand{\Fil}{\operatorname{Fil}} 

\newcommand{\gr}{\operatorname{gr}}

\newcommand{\univ}{{\operatorname{univ}}}
\newcommand{\Spf}{\operatorname{Spf}}

\DeclareMathOperator{\Mod}{Mod}
\DeclareMathOperator{\Art}{Art}

\DeclareMathOperator{\JH}{JH}

\DeclareMathOperator{\End}{End}

\DeclareMathOperator{\Hom}{Hom}
\DeclareMathOperator{\Spec}{Spec}

\DeclareMathOperator{\Gal}{Gal}

\DeclareMathOperator{\Ind}{Ind}
\DeclareMathOperator{\ind}{c-Ind}

\DeclareMathOperator{\im}{im}
\DeclareMathOperator{\Frob}{Frob}

\DeclareMathOperator{\Mat}{Mat}

\newcommand{\cris}{\mathrm{cris}}
\newcommand{\st}{\mathrm{st}}

\newcommand{\un}{\mathrm{un}}

\theoremstyle{plain} 
\newtheorem{lemma}[equation]{Lemma}
\newtheorem*{lemma*}{Lemma}
\newtheorem{proposition}[equation]{Proposition}
\newtheorem*{proposition*}{Proposition}
\newtheorem{theorem}[equation]{Theorem}
\newtheorem{corollary}[equation]{Corollary}
\newtheorem*{corollary*}{Corollary}

\newtheorem*{theorem*}{Theorem}
\newtheorem{deflemma}[equation]{Lemma/Definition}

\theoremstyle{definition}
\newtheorem{definition}[equation]{Definition}
\newtheorem*{definition*}{Definition}
\newtheorem{example}[equation]{Example}
\newtheorem*{example*}{Example}
\newtheorem{notation}[equation]{Notation}
\newtheorem{situation}[equation]{Situation}

\theoremstyle{remark}
\newtheorem{remark}[equation]{Remark}
\newtheorem*{remark*}{Remark}

\numberwithin{equation}{subsection}
\numberwithin{figure}{subsection} 
\numberwithin{table}{subsection} 

\makeatletter
\let\c@equation\c@figure
\makeatother
\makeatletter
\let\c@table\c@figure
\makeatother

\newcommand{\mo}{{-1}}
\newcommand{\std}{{\mathrm{std}}}

\newcommand{\ov}[1]{\overline{#1}}
\DeclareMathOperator{\simc}{sim}    
\newcommand{\pcris}{{\mathrm{p}\text{-}\mathrm{cris}}}
\newcommand{\ctensor}{\widehat{\otimes}}    
\newcommand{\SO}{\mathrm{SO}}